\documentclass[12pt]{amsart}
\usepackage{amsmath}
\usepackage{amsxtra}
\usepackage{amscd}
\usepackage{amsthm}
\usepackage{amsfonts}
\usepackage{amssymb}
\usepackage{eucal}
\usepackage{verbatim}
\usepackage[all]{xy}
\usepackage{color}
\definecolor{deepjunglegreen}{rgb}{0.0, 0.29, 0.29}
\definecolor{darkspringgreen}{rgb}{0.09, 0.45, 0.27}

\usepackage{stmaryrd}

\usepackage[pdftex,bookmarks=false,colorlinks=true,citecolor=darkspringgreen,debug=true,
  naturalnames=true,pdfnewwindow=true]{hyperref}

\newtheorem{thm}{Theorem}[subsection]

\newtheorem{lem}[thm]{Lemma}
\newtheorem{prop}[thm]{Proposition}

\theoremstyle{definition}
\newtheorem{defn}[thm]{Definition}

\theoremstyle{remark}
\newtheorem{rem}[thm]{Remark}

\newcommand{\nc}{\newcommand}
\nc{\renc}{\renewcommand} \nc{\ssec}{\subsection}
\nc{\sssec}{\subsubsection}

\makeatletter
\renewcommand{\subsection}{\@startsection{subsection}{2}{0pt}{-3ex
plus -1ex minus -0.2ex}{-2mm plus -0pt minus
-2pt}{\normalfont\bfseries}} \makeatother

\numberwithin{equation}{subsection}
\allowdisplaybreaks

\nc{\on}{\operatorname} \nc{\wh}{\widehat}
\nc\ol{\overline} \nc\ul{\underline} \nc\wt{\widetilde}

\nc{\BA}{{\mathbb{A}}} \nc{\BC}{{\mathbb{C}}} \nc{\BF}{{\mathbb{F}}}
\nc{\BD}{{\mathbb{D}}} \nc{\BG}{{\mathbb{G}}} \nc{\BQ}{{\mathbb{Q}}}
\nc{\BM}{{\mathbb{M}}} \nc{\BN}{{\mathbb{N}}} \nc{\BO}{{\mathbb{O}}}
\nc{\BP}{{\mathbb{P}}} \nc{\BR}{{\mathbb{R}}} \nc{\BU}{{\mathbb{U}}} \nc{\BV}{{\mathbb{V}}}
\nc{\BZ}{{\mathbb{Z}}} \nc{\BS}{{\mathbb{S}}} \nc{\BW}{{\mathbb{W}}}

\nc{\CA}{{\mathcal{A}}} \nc{\CB}{{\mathcal{B}}} \nc{\CalD}{{\mathcal{D}}}
\nc{\CE}{{\mathcal{E}}} \nc{\CF}{{\mathcal{F}}}
\nc{\CG}{{\mathcal{G}}} \nc{\CH}{{\mathcal{H}}}
\nc{\CI}{{\mathcal{I}}} \nc{\CK}{{\mathcal{K}}} \nc{\CL}{{\mathcal{L}}}
\nc{\CM}{{\mathcal{M}}} \nc{\CN}{{\mathcal{N}}}
\nc{\CO}{{\mathcal{O}}} \nc{\CP}{{\mathcal{P}}}
\nc{\CQ}{{\mathcal{Q}}} \nc{\CR}{{\mathcal{R}}}
\nc{\CS}{{\mathcal{S}}} \nc{\CT}{{\mathcal{T}}}
\nc{\CU}{{\mathcal{U}}} \nc{\CV}{{\mathcal{V}}}  \nc{\CY}{{\mathcal Y}}
\nc{\CW}{{\mathcal{W}}} \nc{\CZ}{{\mathcal{Z}}}

\nc{\cM}{{\check{\mathcal M}}{}} \nc{\csM}{{\check{\mathcal A}}{}}
\nc{\ocM}{{\overset{\circ}{\mathcal M}}{}}
\nc{\obM}{{\overset{\circ}{\mathbf M}}{}}
\nc{\oCA}{{\overset{\circ}{\mathcal A}}{}}
\nc{\obA}{{\overset{\circ}{\mathbf A}}{}}
\nc{\ooM}{{\overset{\circ}{M}}{}}
\nc{\osM}{{\overset{\circ}{\mathsf M}}{}}
\nc{\vM}{{\overset{\bullet}{\mathcal M}}{}}
\nc{\nM}{{\underset{\bullet}{\mathcal M}}{}}
\nc{\oD}{{\overset{\circ}{\mathcal D}}{}}
\nc{\obD}{{\overset{\circ}{\mathbf D}}{}}
\nc{\oA}{{\overset{\circ}{\mathbb A}}{}}
\nc{\op}{{\overset{\bullet}{\mathbf p}}{}}
\nc{\cp}{{\overset{\circ}{\mathbf p}}{}}
\nc{\oU}{{\overset{\bullet}{\mathcal U}}{}}
\nc{\ofZ}{{\overset{\circ}{\mathfrak Z}}{}}

\nc{\ff}{{\mathfrak{f}}} \nc{\fv}{{\mathfrak{v}}}
\nc{\fa}{{\mathfrak{a}}} \nc{\fb}{{\mathfrak{b}}}
\nc{\fd}{{\mathfrak{d}}} \nc{\fe}{{\mathfrak{e}}}
\nc{\fg}{{\mathfrak{g}}} \nc{\fgl}{{\mathfrak{gl}}}
\nc{\fh}{{\mathfrak{h}}} \nc{\fri}{{\mathfrak{i}}}
\nc{\fj}{{\mathfrak{j}}} \nc{\fk}{{\mathfrak{k}}} \nc{\fl}{{\mathfrak{l}}}
\nc{\fm}{{\mathfrak{m}}} \nc{\fn}{{\mathfrak{n}}}
\nc{\ft}{{\mathfrak{t}}} \nc{\fu}{{\mathfrak{u}}}
\nc{\fw}{{\mathfrak{w}}} \nc{\fz}{{\mathfrak{z}}}
\nc{\fp}{{\mathfrak{p}}} \nc{\fq}{{\mathfrak{q}}} \nc{\frr}{{\mathfrak{r}}}
\nc{\fs}{{\mathfrak{s}}} \nc{\fsl}{{\mathfrak{sl}}}
\nc{\hsl}{{\widehat{\mathfrak{sl}}}}
\nc{\hgl}{{\widehat{\mathfrak{gl}}}}
\nc{\hg}{{\widehat{\mathfrak{g}}}}
\nc{\chg}{{\widehat{\mathfrak{g}}}{}^\vee}
\nc{\hn}{{\widehat{\mathfrak{n}}}}
\nc{\chn}{{\widehat{\mathfrak{n}}}{}^\vee}

\nc{\fA}{{\mathfrak{A}}} \nc{\fB}{{\mathfrak{B}}} \nc{\fC}{{\mathfrak{C}}}
\nc{\fD}{{\mathfrak{D}}} \nc{\fE}{{\mathfrak{E}}}
\nc{\fF}{{\mathfrak{F}}} \nc{\fG}{{\mathfrak{G}}} \nc{\fH}{{\mathfrak{H}}}
\nc{\fI}{{\mathfrak{I}}} \nc{\fJ}{{\mathfrak{J}}}
\nc{\fK}{{\mathfrak{K}}} \nc{\fL}{{\mathfrak{L}}}
\nc{\fM}{{\mathfrak{M}}} \nc{\fN}{{\mathfrak{N}}}
\nc{\fP}{{\mathfrak{P}}} \nc{\fQ}{{\mathfrak{Q}}}
\nc{\fS}{{\mathfrak{S}}} \nc{\fT}{{\mathfrak{T}}} \nc{\fU}{{\mathfrak{U}}}
\nc{\fV}{{\mathfrak{V}}} \nc{\fW}{{\mathfrak{W}}}
\nc{\fX}{{\mathfrak{X}}} \nc{\fY}{{\mathfrak{Y}}}
\nc{\fZ}{{\mathfrak{Z}}}

\nc{\ba}{{\mathbf{a}}}
\nc{\bb}{{\mathbf{b}}} \nc{\bc}{{\mathbf{c}}}
\nc{\be}{{\mathbf{e}}} \nc{\bj}{{\mathbf{j}}} \nc{\bm}{{\mathbf{m}}}
\nc{\bn}{{\mathbf{n}}} \nc{\bp}{{\mathbf{p}}}
\nc{\bq}{{\mathbf{q}}} \nc{\br}{{\mathbf{r}}} \nc{\bt}{{\mathbf{t}}}
\nc{\bfu}{{\mathbf{u}}} \nc{\bv}{{\mathbf{v}}}
\nc{\bx}{{\mathbf{x}}} \nc{\by}{{\mathbf{y}}} \nc{\bz}{{\mathbf{z}}}
\nc{\bw}{{\mathbf{w}}} \nc{\bA}{{\mathbf{A}}}
\nc{\bB}{{\mathbf{B}}} \nc{\bC}{{\mathbf{C}}}
\nc{\bD}{{\mathbf{D}}} \nc{\bF}{{\mathbf{F}}} \nc{\bG}{{\mathbf{G}}}
\nc{\bH}{{\mathbf{H}}} \nc{\bI}{{\mathbf{I}}} \nc{\bJ}{{\mathbf{J}}}
\nc{\bK}{{\mathbf{K}}} \nc{\bM}{{\mathbf{M}}} \nc{\bN}{{\mathbf{N}}}
\nc{\bO}{{\mathbf{O}}} \nc{\bS}{{\mathbf{S}}} \nc{\bT}{{\mathbf{T}}}
\nc{\bU}{{\mathbf{U}}} \nc{\bV}{{\mathbf{V}}} \nc{\bW}{{\mathbf{W}}}
\nc{\bX}{{\mathbf{X}}}
\nc{\bY}{{\mathbf{Y}}} \nc{\bP}{{\mathbf{P}}}
\nc{\bZ}{{\mathbf{Z}}} \nc{\bh}{{\mathbf{h}}}

\nc{\sA}{{\mathsf{A}}} \nc{\sB}{{\mathsf{B}}}
\nc{\sC}{{\mathsf{C}}} \nc{\sD}{{\mathsf{D}}}
\nc{\sE}{{\mathsf{E}}} \nc{\sF}{{\mathsf{F}}} \nc{\sG}{{\mathsf{G}}}
\nc{\sI}{{\mathsf{I}}} \nc{\sK}{{\mathsf{K}}} \nc{\sL}{{\mathsf{L}}}
\nc{\sfm}{{\mathsf{m}}} \nc{\sM}{{\mathsf{M}}} \nc{\sO}{{\mathsf{O}}}
\nc{\sQ}{{\mathsf{Q}}} \nc{\sP}{{\mathsf{P}}}
\nc{\sT}{{\mathsf{T}}} \nc{\sZ}{{\mathsf{Z}}}
\nc{\sV}{{\mathsf{V}}} \nc{\sW}{{\mathsf{W}}}
\nc{\sfp}{{\mathsf{p}}} \nc{\sq}{{\mathsf{q}}} \nc{\sr}{{\mathsf{r}}}
\nc{\st}{{\mathsf{t}}} \nc{\sfb}{{\mathsf{b}}}
\nc{\sfc}{{\mathsf{c}}} \nc{\sd}{{\mathsf{d}}}
\nc{\sy}{{\mathsf{y}}} \nc{\sz}{{\mathsf{z}}}

\nc{\tA}{{\widetilde{\mathbf{A}}}}
\nc{\tB}{{\widetilde{\mathcal{B}}}}
\nc{\tg}{{\widetilde{\mathfrak{g}}}} \nc{\tG}{{\widetilde{G}}}
\nc{\TM}{{\widetilde{\mathbb{M}}}{}}
\nc{\tO}{{\widetilde{\mathsf{O}}}{}}
\nc{\tU}{{\widetilde{\mathfrak{U}}}{}} \nc{\TZ}{{\tilde{Z}}}
\nc{\tx}{{\tilde{x}}} \nc{\tbv}{{\tilde{\bv}}}
\nc{\tfP}{{\widetilde{\mathfrak{P}}}{}} \nc{\tz}{{\tilde{\zeta}}}
\nc{\tmu}{{\tilde{\mu}}}

\nc{\urho}{\underline{\rho}} \nc{\uB}{\underline{B}}
\nc{\uC}{{\underline{\mathbb{C}}}} \nc{\ui}{\underline{i}}
\nc{\uj}{\underline{j}} \nc{\ofP}{{\overline{\mathfrak{P}}}}
\nc{\oB}{{\overline{\mathcal{B}}}}
\nc{\og}{{\overline{\mathfrak{g}}}} \nc{\oI}{{\overline{I}}}

\nc{\eps}{\varepsilon} \nc{\hrho}{{\hat{\rho}}} \nc{\bomega}{{\boldsymbol{\omega}}}
\nc{\blambda}{{\boldsymbol{\lambda}}} \nc{\bmu}{{\boldsymbol{\mu}}} \nc{\bnu}{{\boldsymbol{\nu}}}

\nc{\one}{{\mathbf{1}}} \nc{\two}{{\mathbf{t}}}

\nc{\Sym}{\mathop{\operatorname{\rm Sym}}}
\nc{\Tot}{{\mathop{\operatorname{\rm Tot}}}}
\nc{\Spec}{\mathop{\operatorname{\rm Spec}}}
\nc{\Ker}{{\mathop{\operatorname{\rm Ker}}}}
\nc{\Isom}{{\mathop{\operatorname{\rm Isom}}}}
\nc{\Hilb}{{\mathop{\operatorname{\rm Hilb}}}}
\nc{\deeq}{{\mathop{\operatorname{\rm deeq}}}}
\nc{\End}{{\mathop{\operatorname{\rm End}}}}
\nc{\Ext}{{\mathop{\operatorname{\rm Ext}}}}
\nc{\Hom}{{\mathop{\operatorname{\rm Hom}}}}
\nc{\CHom}{{\mathop{\operatorname{{\mathcal{H}}\it om}}}}
\nc{\GL}{{\mathop{\operatorname{\rm GL}}}}
\nc{\SL}{{\mathop{\operatorname{\rm SL}}}}
\nc{\gr}{{\mathop{\operatorname{\rm gr}}}}
\nc{\Id}{{\mathop{\operatorname{\rm Id}}}}
\nc{\perf}{{\mathop{\operatorname{\rm perf}}}}
\nc{\defi}{{\mathop{\operatorname{\rm def}}}}
\nc{\length}{{\mathop{\operatorname{\rm length}}}}
\nc{\supp}{{\mathop{\operatorname{\rm supp}}}}
\nc{\HC}{{\mathcal H}{\mathcal C}}

\nc{\pr}{{\operatorname{pr}}}
\nc{\Cliff}{{\mathsf{Cliff}}}
\nc{\loc}{{\operatorname{loc}}}
\nc{\Fl}{{\mathbf{Fl}}} \nc{\Ffl}{{\mathcal{F}\ell}}
\nc{\Fib}{{\mathsf{Fib}}}
\nc{\Coh}{{\mathsf{Coh}}} \nc{\FCoh}{{\mathsf{FCoh}}}
\nc{\Perf}{{\mathsf{Perf}}}

\nc{\reg}{{\text{\rm reg}}}
\nc{\gvee}{{\mathfrak g}^{\!\scriptscriptstyle\vee}}
\nc{\tvee}{{\mathfrak t}^{\!\scriptscriptstyle\vee}}
\nc{\nvee}{{\mathfrak n}^{\!\scriptscriptstyle\vee}}
\nc{\bvee}{{\mathfrak b}^{\!\scriptscriptstyle\vee}}
       \nc{\rhovee}{\rho^{\!\scriptscriptstyle\vee}}

\nc{\cplus}{{\mathbf{C}_+}} \nc{\cminus}{{\mathbf{C}_-}}
\nc{\cthree}{{\mathbf{C}_*}} \nc{\Qbar}{{\bar{Q}}}

\newcommand{\oM}{\vphantom{j^{X^2}}\smash{\overset{\circ}{\vphantom{\rule{0pt}{0.55em}}\smash{M}}}}
\newcommand{\oS}{\vphantom{j^{X^2}}\smash{\overset{\circ}{\vphantom{\rule{0pt}{0.55em}}\smash{S}}}}
\newcommand{\oZ}{\vphantom{j^{X^2}}\smash{\overset{\circ}{\vphantom{\rule{0pt}{0.55em}}\smash{Z}}}}
\newcommand{\uoZ}{\vphantom{j^{X^2}}\smash{\overset{\circ}{\vphantom{\rule{0pt}{0.55em}}\smash{\underline{Z}}}}}

\newcommand\iso{\mathbin{\vphantom{j^{X^2}}\smash{\overset{\sim}{\vphantom{\rule{0pt}{0.20em}}\smash{\longrightarrow}}}}}
\nc{\Gtimes}{\vphantom{j^{X^2}}\smash{\overset{G}{\vphantom{\rule{0pt}{0.30em}}\smash{\times}}}}
\nc{\sGtimes}{\vphantom{j^{X^2}}\smash{\overset{\mathsf G}{\vphantom{\rule{0pt}{0.30em}}\smash{\times}}}}

\nc{\svee}{{\!\scriptscriptstyle\vee}}

\nc{\bOmega}{{\overline{\Omega}}}

\nc{\seq}[1]{\stackrel{#1}{\sim}}

\nc{\aff}{{\operatorname{aff}}}
\nc{\der}{{\operatorname{der}}}
\nc{\fin}{{\operatorname{fin}}}
\nc{\mir}{{\operatorname{mir}}}
\nc{\rat}{{\operatorname{rat}}}
\nc{\triv}{{\operatorname{triv}}}
\nc{\ext}{{\operatorname{ext}}}
\nc{\righ}{{\operatorname{right}}}
\nc{\lef}{{\operatorname{left}}}
\nc{\forg}{{\operatorname{forg}}}
\nc{\fid}{{\operatorname{fd}}}
\nc{\modu}{{\operatorname{-mod}}}
\nc{\Gr}{{\operatorname{Gr}}}
\nc{\bGr}{{\mathbf{Gr}}}
\nc{\FT}{{\operatorname{FT}}}
\nc{\Mat}{{\operatorname{Mat}}}
\nc{\MSt}{{\operatorname{MSt}}}
\nc{\sph}{{\operatorname{sph}}}
\nc{\GR}{{\mathbf{Gr}}}
\nc{\Perv}{{\operatorname{Perv}}}
\nc{\Rep}{{\operatorname{Rep}}}
\nc{\Ind}{{\operatorname{Ind}}}
\nc{\IC}{{\operatorname{IC}}}
\nc{\Bun}{{\operatorname{Bun}}}
\nc{\Proj}{{\operatorname{Proj}}}
\nc{\Stab}{{\operatorname{Stab}}}
\nc{\pt}{{\operatorname{pt}}}
\nc{\tr}{{\operatorname{tr}}}
\nc{\tw}{{\operatorname{tw}}}
\nc{\elli}{{e\ell\ell}}
\nc{\bfmu}{{\boldsymbol{\mu}}}
\nc{\bfomega}{{\boldsymbol{\omega}}}

\nc{\tslash}{/\!\!/\!\!/}

\nc{\calD}{\mathcal D}

\nc\RHom{\operatorname{RHom}}
\nc\Res{\operatorname{Res}}
\nc\Av{\operatorname{Av}}

\newcommand*\circled[1]
{*{#1}}
\makeatletter
\newcommand{\raisemath}[1]{\mathpalette{\raisem@th{#1}}}
\newcommand{\raisem@th}[3]{\raisebox{#1}{$#2#3$}}
\nc{\binlim}[2][]{\def\@tempa{#1}\@ifnextchar^{\@binlim{#2}}{\@binlim{#2}^{}}}
\def\@binlim#1^#2{\mathbin{\@ifempty{#2}{\mathop{#1}}{\mathop{#1}\@xp\displaylimits\@tempa^{#2}}}}

\makeatother

\def\arxiv#1{\href{http://arxiv.org/abs/#1}{\tt arXiv:#1}} \let\arXiv\arxiv
\newcommand{\dbkts}[1]{[\![#1]\!]}
\newcommand{\dprts}[1]{(\!(#1)\!)}
\nc\Gm{{\mathbb G_m}}
\nc{\isoto}
    {\stackrel{\displaystyle\raisemath{-.2ex}{\sim}}\to}
\nc\cD\CalD
\nc\setm\setminus
\nc\bsl\backslash
\nc\Fq{\mathbb F_q}
\nc\x\times
\nc\bGO{{\bG_\bO}}
\nc\bla\blambda
\nc\blt\bullet
\nc\opp{{\on{op}}}
\nc\srel\stackrel
\nc\oast\circledast
\nc\tbx{\binlim{\widetilde\boxtimes{}}}
\nc\into\hookrightarrow
\nc\otto\leftrightarrow
\renc\setminus\smallsetminus
\newenvironment{i-ii-iii}{%
\begin{enumerate}
}%
{\end{enumerate}}
\nc\ceil[1]{\lceil#1\rceil}  \nc\floor[1]{\lfloor#1\rfloor}
\nc\Lie{\on{Lie}}

\begin{document}

\author[M.~Finkelberg]{Michael Finkelberg}
\address{Department of Mathematics,
National Research University Higher School of Economics,
Russian Federation, Usacheva st.\ 6, 119048, Moscow;
\newline Skolkovo Institue of Science and Technology;
\newline Institute for Information Transmission Problems of RAS}
\email{fnklberg@gmail.com}
\author[M.~Matviichuk]{Mykola Matviichuk}
\address{Department of Mathematics, McGill University, Burnside Hall, 805 Sherbrooke st.~W.,
  Montr\'eal, Qu\'ebec, H3A 2K6, Canada}
  \email{mykola.matviichuk@gmail.com}
\author[A.~Polishchuk]{Alexander Polishchuk}
\address{Department of Mathematics, University of Oregon, Eugene, OR 97403;
  \newline National Research University Higher School of Economics;
\newline Korea Institute for Advanced Study, Seoul, Korea}
\email{apolish@uoregon.edu}

\title{Elliptic zastava}

\dedicatory{To Tony Joseph on his 80th birthday, with admiration}

\begin{abstract}
  We study the elliptic zastava spaces, their versions (twisted, Coulomb, Mirkovi\'c local spaces,
  reduced) and relations with monowalls moduli spaces and Feigin-Odesskii moduli spaces of
  $G$-bundles with parabolic structure on an elliptic curve. 
\end{abstract}
\maketitle
\tableofcontents

\section{Introduction}

\subsection{Zastava spaces: general overview}
Let $G$ be an almost simple simply connected algebraic group over $\BC$. Let us also fix a pair
of opposite Borel subgroups $B$, $B_-$ whose intersection is a maximal torus $T$. To a smooth
projective complex curve $C$, one can associate the zastava moduli space $Z(C)$ (the definition
goes back to V.~Drinfeld, see e.g.~\cite{bfgm}). It is the moduli space of $G$-bundles on $C$
equipped with a generalized $B$-structure and a generically transversal $U_-$-structure (here $U_-$
stands for the unipotent radical of $B_-$). It is actually a scheme with infinitely many connected
components numbered by the degrees of $B$-bundles. It has numerous applications in geometric
representation theory and especially in the geometric Langlands program (see e.g.~\cite{g,bf}).

The zastava space $Z(C)$ is equipped with a morphism $\pi$ to the colored configuration space
$\on{Conf}_G(C)$ of $C$
(it keeps track of the points of $C$ where the $B$- and $U_-$-structures fail to be transversal),
and one of the key features of $Z(C)$ is its factorization structure over the configurations
(locality over $C$). It allows to define $Z(C)$ for arbitrary smooth complex curve; not necessarily
projective: $Z(C)$ is defined as the preimage $\pi^{-1}\on{Conf}_G(C)\subset Z(\ol{C})$ for a
smooth compactification $\ol{C}\supset C$.

A special role is played by three smooth curves carrying the structure of 1-dimensional complex
algebraic groups: the additive group $\BG_a$, the multiplicative group $\BG_m$, and an elliptic
curve $E$. The {\em open} zastava $\oZ(C)\subset Z(C)$ (given by the open condition that a
$B$-structure is genuine as opposed to generalized) for these three curves play a prominent role
in physics as various versions of the {\em monopole} moduli spaces.

More precisely, the additive (or rational) open zastava are isomorphic to the euclidean monopoles' moduli
spaces~\cite{j,j'}, while the multiplicative (or trigonometric) open zastava are expected to be related
to the periodic monopoles' moduli spaces~\cite{ck}, and elliptic open zastava are expected to be related
to the doubly periodic monopoles' (or monowalls') moduli spaces~\cite{cw}.
Yet more precisely, the open zastava spaces are equipped with a natural $T$-action and a map to
$C^{\on{rk}G}$ playing the role of the moment map. These allow to define a sort of (quasi)-Hamiltonian
reduction $\uoZ(C)$. The reduced zastava in additive case is isomorphic to the moduli space
of {\em centered} euclidean monopoles; in multiplicative (resp.\ elliptic) case, the reduced zastava
is expected to be isomorphic to the moduli space of periodic monopoles (resp.\ monowalls).
The monopole moduli spaces come equipped with a natural hyperk\"ahler structure, and the zastava
spaces carry the corresponding holomorphic symplectic structure that can be defined in modular terms
and explicitly computed in appropriate coordinates.

Furthermore, the euclidean monopole moduli spaces are known to be isomorphic to the Coulomb branches
of 3-dimensional $\CN=4$-supersymmetric quiver gauge theories (for the Dynkin quiver of $G$; with
symmetrizers if $G$ is not simply laced). See~\cite{bfn2} for a mathematically rigorous identification
of the Coulomb branch with $\oZ(\BG_a)$. Similarly, the K-theoretic Coulomb branch can be identified
with $\oZ(\BG_m)$, see~\cite{ft}. One of the main topics of the present paper is an identification of
$\uoZ(E)$ with an appropriate version of {\em elliptic} Coulomb branch (whose rigorous mathematical
definition is not formulated yet). From this point of view, the above holomorphic symplectic structures
on open zastava arise from the natural quantizations of the Coulomb branches. These quantizations
are, respectively the truncated shifted Yangians~\cite{bfn2}, the truncated shifted quantum affine
algebras~\cite{ft}, and supposedly related to the elliptic quantum groups.

Actually, the reduced elliptic open zastava $\uoZ(E)$ appeared in mathematics long ago in another
disguise in the works of B.~Feigin and A.~Odesskii. Namely, let us modify the definition of $\oZ(E)$,
replacing a $U_-$-structure by a $\CU_-^\CK$-structure, where $\CU_-^\CK$ is a unipotent group scheme
over $E$ obtained from $U_-$ via twisting by a regular $T$-torsor $\CK_T$. Then the resulting reduced
zastava $\uoZ{}_\CK(E)$ is isomorphic to the Feigin-Odesskii moduli space of complete flags in the
$G$-bundle $\on{Ind}_T^G\CK_T$ with a fixed isomorphism class of the associated graded bundle.
B.~Feigin and A.~Odesskii constructed a natural symplectic structure on their moduli spaces
(along with its quantization), and it turns out that this symplectic structure coincides with the
one of the previous paragraph.

In the remaining sections of Introduction we provide a more detailed overview of the above topics
along with some other aspects of our work, like {\em Mirkovi\'c local spaces} needed for identification
of various types of elliptic zastava.

\subsection{Rational zastava and euclidean monopoles}
\label{euclid}
We denote by $\CB$ the flag variety of $G$. Let $\Lambda$ denote the cocharacter lattice of $T$;
since $G$ is assumed to be simply connected, this is also the coroot lattice of $G$.
We denote by $\Lambda_{\on{pos}}\subset \Lambda$ the sub-semigroup spanned by positive coroots.

It is well-known that $H_2(\CB,\BZ)=\Lambda$ and that an element $\alpha\in H_2(\CB,\BZ)$ is
representable by an effective algebraic curve if and only if $\alpha\in \Lambda_{\on{pos}}$.
The (open) {\em zastava} $\oZ^{\alpha}$ is the moduli space of maps $C=\BP^1\to \CB$ of degree $\alpha$
sending $\infty\in \BP^1$ to $B_-\in \CB$. It is known~\cite{fkmm} that this is a smooth symplectic
affine algebraic variety, which can be identified with the hyperk\"ahler moduli space of
framed $G$-monopoles on $\BR^3$ with maximal symmetry breaking at infinity of charge
$\alpha$~\cite{j,j'}. Let us mention one more equivalent definition of $\oZ^{\alpha}$:
it is the moduli space of $G$-bundles on $\BP^1$ equipped with a $B$-structure of degree $\alpha$
and a $U_-$-structure transversal to the $B$-structure at $\infty\in\BP^1$.

The zastava space is equipped with a {\em factorization} morphism
$\pi^\alpha\colon \oZ^\alpha\to\BA^\alpha$ with
a simple geometric meaning: for a based map
$\varphi\in\oZ^\alpha$ the colored divisor $\pi^\alpha(\varphi)$ is just the
pullback of the colored Schubert divisor $D\subset\CB$ equal to the complement
of the open $B$-orbit in $\CB$. The morphism
$\pi^\alpha\colon \oZ^\alpha\to\BA^\alpha$ is the {\em Atiyah-Hitchin} integrable
system (with respect to the above symplectic structure): all the fibers of
$\pi^\alpha$ are Lagrangian.

\medskip

A system of \'etale birational coordinates on $\oZ^\alpha$ was introduced
in~\cite{fkmm}. Let us recall the definition
for $G=SL(2)$. In this case $\alpha$ is $a$ times the simple coroot, and $\oZ^a:=\oZ^\alpha$
consists of all maps $\BP^1\to \BP^1$ of degree $a$
which send $\infty$ to $0$. We can represent such a map by a rational function $\frac{R}{Q}$ where
$Q$ is a monic polynomial of degree $a$ and $R$ is a polynomial of degree $<a$. Let
$w_1,\ldots,w_a$ be the zeros of $Q$. Set $y_r=R(w_r)$.
Then the functions $(y_1,\ldots,y_a,w_1,\ldots, w_a)$ form a system of \'etale birational
coordinates on $\oZ^a$, and the above mentioned
symplectic form in these coordinates reads
$\Omega_\rat=\sum_{r=1}^a\frac{dy_r\wedge dw_r}{y_r}$.

For general $G$ the definition of the above coordinates is quite similar. In this case given a
point in $\oZ^\alpha$ we can define polynomials $R_i,Q_i$ where $i$
runs through the set $I$ of vertices of the Dynkin diagram of $G,\ \alpha=\sum a_i\alpha_i$, and

(1) $Q_i$ is a monic polynomial of degree $a_i$,

(2) $R_i$ is a polynomial of degree $<a_i$.

Hence, we can define (\'etale, birational) coordinates
$(y_{i,r},w_{i,r})$ where $i\in I$ and
$r=1,\ldots,a_i$. Namely, $w_{i,r}$ are the roots of $Q_i$, and
$y_{i,r}=R_i(w_{i,r})$. The Poisson brackets of these coordinates
with respect to the above symplectic form are as follows:
$\{w_{i,r},w_{j,s}\}_\rat=0,\
\{w_{i,r},y_{j,s}\}_\rat=d\,^\svee\!\!\!_i\delta_{ij}\delta_{rs}y_{j,s},\
\{y_{i,r},y_{j,s}\}_\rat=
(\alpha^\svee_i,\alpha^\svee_j)\frac{y_{i,r}y_{j,s}}{w_{i,r}-w_{j,s}}$ for
$i\ne j$, and finally $\{y_{i,r},y_{i,s}\}_\rat=0$. Here $\alpha^\svee_i$ is a
simple root, $(,)$ is the invariant scalar product on $(\on{Lie}T)^*$ such
that the square length of a short root is 2, and
$d\,^\svee\!\!\!_i=(\alpha^\svee_i,\alpha^\svee_i)/2$.

\medskip

Finally, let us mention that the zastava space $\oZ^\alpha$ is isomorphic to the Coulomb branch
of a $3d\ \CN=4$ supersymmetric quiver gauge theory (for a Dynkin quiver of $G$, with no
framing; with symmetrizers for a non simply laced $G$), see~\cite{bfn2,nw}.

\subsection{Trigonometric zastava and periodic monopoles}
We have an open subset $\BG_m^\alpha\subset\BA^\alpha$ (colored divisors not
meeting $0\in\BA^1$), and the {\em trigonometric zastava} is defined as the
open subvariety $^\dagger\!\oZ^\alpha:=(\pi^\alpha)^{-1}(\BG_m^\alpha)\subset\oZ^\alpha$.
It can be identified with a solution of a certain moduli problem on the 
irreducible nodal curve of arithmetic genus~1 obtained by gluing the points $0,\infty\in\BP^1$,
see~\cite{fkr}. From this point of view it acquires a natural symplectic structure with the
corresponding bracket $\{,\}_{\on{trig}}$. Note that $\{,\}_{\on{trig}}$ is {\em not} the
restriction of $\{,\}_\rat$ from $\oZ^\alpha$, but rather its trigonometric version.

For example, when $G=\SL(2)$ and $\alpha$ is $a$ times the simple coroot, the Atiyah-Hitchin
integrable system $\pi^a\colon\oZ^a\to\BA^{(a)}$ is nothing but the classical Toda lattice
for $\GL(a)$,
while its trigonometric version $\pi^a\colon{}^\dagger\!\oZ^a\to\BG_m^{(a)}$ can be identified
with the {\em relativistic} Toda lattice for $\GL(a)$, see~\cite[\S2]{ft}.

An explicit formula for $\{,\}_{\on{trig}}$ in $w,y$-coordinates is obtained in~\cite{fkr}.

\medskip

The composed morphism
\[^\dagger\!\oZ^\alpha\xrightarrow{\pi^\alpha}\BG_m^\alpha\xrightarrow{\prod}\BG_m^I\cong T\]
(recall that $I$ is the set of simple coroots of $G$) is the group valued moment map of the
Hamiltonian action
of $T$ on $^\dagger\!\oZ^\alpha$. The quotient of a level of this moment map by the action of $T$
is the {\em reduced trigonometric zastava} $^\dagger\!\uoZ{}^\alpha$: the (quasi-)Hamiltonian
reduction of $^\dagger\!\oZ^\alpha$.

It is likely that the reduced trigonometric zastava is isomorphic to the moduli space of
{\em periodic} monopoles (see e.g.~\cite{ck}) in one of its complex structures (it has a natural
hyperk\"ahler structure, and among the $S^2$-worth of the underlying complex structures we need
a generic one, in which this moduli space is an affine variety). The corresponding holomorphic
symplectic structure on the moduli space of periodic monopoles matches the reduction of
$\{,\}_{\on{trig}}$. Note an important difference with the rational case: the usual zastava was
isomorphic to the euclidean monopoles' moduli space, and its Hamiltonian reduction with respect
to the $T$-action was isomorphic to the {\em centered} monopole moduli space. In the periodic
case the monopoles come centered by definition.

Finally, the trigonometric zastava $^\dagger\!\oZ^\alpha$ is isomorphic to the $K$-theoretic
Coulomb branch of a $3d\ \CN=4$ supersymmetric quiver gauge theory (for a Dynkin quiver of $G$,
with no framing; with symmetrizers for a non simply laced $G$), see~\cite{ft} for the simply
laced case. The reduced trigonometric zastava $^\dagger\!\uoZ{}^\alpha$ is isomorphic to the
$K$-theoretic Coulomb branch where the gauge group must be taken as the product of $\SL(V_i)$
(as opposed to the product of $\GL(V_i)$ for the trigonometric zastava).

\subsection{Elliptic zastava}
The explicit formulas for $\{,\}_\rat$ and $\{,\}_{\on{trig}}$ look like rational and trigonometric
degenerations of the Feigin-Odesskii bracket~\cite{fo} on the moduli space of $G$-bundles with
a parabolic structure on an elliptic curve. The goal of the present paper is to give a precise
meaning to this observation.\footnote{This goal is achieved in~Theorem~\ref{myksas} where we
establish a symplectomorphism of the Feigin-Odesskii moduli space with a reduced elliptic zastava
space. Compare the formula at the end of~\S\ref{explicit FO} with the one at the end
of~\S\ref{euclid}.}

For a $T$-bundle $\CK_T$ on an elliptic curve $E$ we consider the moduli space
$\oZ^\alpha_\CK$ of the following data:

(a) a $G$-bundle $\CF_G$ on $E$,

(b) a $B$-structure $\varphi_+$ on $\CF_G$ such that the induced $T$-bundle $\CL_T=\Ind_B^T\varphi_+$
has degree $-\alpha$,

(c) a $\CU_-^\CK$-structure $\varphi_-$ on $\CF_G$ generically transversal to $\varphi_+$.
Here $\CU_-^\CK$ is a sheaf of unipotent groups locally isomorphic to $U_-$, obtained from the
trivial sheaf by twisting with $T$-bundle $\CK_T$ (we view $T$ as a subgroup of $\on{Aut}U_-$ via
the adjoint action).

The open elliptic zastava $\oZ^\alpha_\CK$ is a smooth connected variety of dimension $2|\alpha|$
equipped with an affine factorization morphism $\pi^\alpha\colon \oZ^\alpha_\CK\to E^\alpha$
to a configuration space of $E$. It has a relative compactification (compactified elliptic zastava)
\[\oZ^\alpha_\CK\subset\ol{Z}{}_\CK^\alpha\xrightarrow{\pi^\alpha}E^\alpha\]
where we allow both a $B$-structure and a $\CU_-^\CK$-structure to be generalized in the sense
of Drinfeld. There is also an intermediate version
$\oZ^\alpha_\CK\subset Z_\CK^\alpha\subset\ol{Z}{}_\CK^\alpha$ (elliptic zastava) where only a
$B$-structure is allowed to be generalized.

For example, when $G=\SL(2),\ \CK_T$ is trivial, and $\alpha$ is $a$ times the simple coroot, there is
an isomorphism $Z_{\CK_{\on{triv}}}^a\simeq TE^{(a)}$ with the total space of the tangent bundle of the
$a$-th symmetric power of $E$. Unfortunately, neither $TE^{(a)}$ nor its open subvariety
$\oZ_{\CK_{\on{triv}}}^a$ carry any natural Poisson structure.

\subsection{Coulomb elliptic zastava}
Similarly to the rational and trigonometric cases, one can consider the elliptic Coulomb branch
of a $3d\ \CN=4$ supersymmetric quiver gauge theory for a Dynkin quiver of $G$
with no framing. We restrict ourselves to the case of simply laced
$G$.\footnote{In the non simply laced
case one should use the approach of~\cite{nw} with symmetrizers.} The elliptic Coulomb
branch is the (relative) spectrum of the equivariant Borel-Moore elliptic homology of a certain
variety of triples. The theory of equivariant Borel-Moore elliptic homology is not developed yet;
it is to appear in a forthcoming work of I.~Perunov and A.~Prikhodko. We sketch some results
in~\S\ref{piat}. The resulting elliptic Coulomb branch is denoted $^C\!\oZ_{\CK_{\on{triv}}}^\alpha$.
It is equipped with a natural Poisson (in fact, symplectic) structure due to the existence of
quantized elliptic Coulomb branch.

For example, when $G=\SL(2)$, there is an isomorphism
$^C\!\oZ_{\CK_{\on{triv}}}^a\simeq\Hilb^a_{\on{tr}}(E\times\BG_m)$ with the {\em transversal Hilbert
  scheme} of the surface $E\times\BG_m$ (an open subvariety of the Hilbert scheme of points on
$E\times\BG_m$ classifying those subschemes whose projection to $E$ is a closed embedding).
Note that we have an open embedding $\Hilb^a_{\on{tr}}(E\times\BG_m)\subset T^*E^{(a)}$ into the
total space of the cotangent bundle of the $a$-th symmetric power of $E$.
Contrary to the rational and trigonometric cases, there is {\em no} isomorphism
$^C\!\oZ_{\CK_{\on{triv}}}^a\not\simeq\oZ_{\CK_{\on{triv}}}^a$ of the open elliptic zastava with the
elliptic Coulomb branch.

Still, the elliptic Coulomb branch is not so much different from the elliptic zastava.
Namely, they can be both obtained by the Mirkovi\'c construction of {\em local spaces} over
(the configuration spaces of) $E$, see e.g.~\cite[\S2]{myz}. This construction depends on a choice
of a {\em local line bundle}; one choice gives rise to the elliptic zastava; another gives rise
to the elliptic Coulomb branch, see~\S\ref{tri}. Moreover, this way we can define the Coulomb
elliptic zastava $^C\!\oZ_\CK^\alpha$ depending on an arbitrary $T$-bundle $\CK_T$, not necessarily
trivial.

\subsection{Feigin-Odesskii moduli space}
Another closely related moduli space $M(\CF_G,\CL_T)$ depending on a choice of a $G$-bundle
$\CF_G$ and a $T$-bundle $\CL_T$ on $E$ classifies the $B$-structures $\varphi$ on $\CF_G$
equipped with an isomorphism $\Ind_B^T\varphi\iso\CL_T$. It can be equipped with a natural
structure of a derived stack with a (0-shifted) symplectic form, see~\S\ref{shest}.
B.~Feigin and A.~Odesskii construct in~\cite{fo} a Poisson structure on the moduli space
$\on{Bun}_P$ of $P$-bundles on $E$ (where $P$ is a parabolic subgroup of $G$). The above moduli
spaces $M(\CF_G,\CL_T)$ coincide with certain symplectic leaves of $\on{Bun}_B$. For instance,
if $G=\on{SL}(2)$, then $M(\CF_G,\CL_T)$ is the moduli space of extensions of a line bundle
$\CL^{-1}$ by $\CL$ with a fixed isomorphism class of the resulting rank~2 bundle $\CV_\CF$. If 
$\CV_\CF$ is assumed to be stable, then $M(\CF_G,\CL_T)$ is a symplectic leaf of the
Feigin-Odesskii bracket on $\on{Bun}_B$.

If we fix a {\em regular} $T$-bundle $\CK_T$ (this means that all the line bundles associated to
the roots of $G$ are {\em nontrivial}), take $\CF_G=\Ind_T^G\CK_T$ and $\deg\CL_T=-\alpha$, then
$M(\CF_G,\CL_T)$ can be identified with a certain ``quasi-Hamiltonian'' reduction
$^{\vphantom{\alpha}}_\CalD\uoZ^\alpha_\CK$ of $\oZ^\alpha_\CK$. Namely, the reduction is defined as
the quotient with respect to the natural $T$-action of a fiber over $\CalD\in E^I$
of the composed morphism
\[\oZ^\alpha_\CK\xrightarrow{\pi^\alpha}E^\alpha\xrightarrow{\sum}E^I\]
(recall that $I$ is the set of simple coroots of $G$).

By the very construction, the Coulomb elliptic zastava $^C\!\oZ_\CK^\alpha$ is also equipped
with the factorization morphism $\pi^\alpha\colon ^C\!\oZ_\CK^\alpha\to E^\alpha$, and so we can
define the reduced Coulomb elliptic zastava $^C_\CalD\uoZ^\alpha_\CK$ in a similar way.
The important difference with the usual elliptic zastava is that the Coulomb elliptic zastava
$^C\!\oZ_\CK^\alpha$ carries a symplectic form, and the above reduction is really a
(quasi-)Hamiltonian reduction. In particular, the reduced Coulomb elliptic zastava
$^C_\CalD\uoZ^\alpha_\CK$ inherits a symplectic form.

The two main results of the present paper are as follows:

A. The reduced elliptic zastava and reduced Coulomb elliptic zastava are isomorphic:
$^{\vphantom{\alpha}}_\CalD\uoZ^\alpha_\CK\simeq{}^C_\CalD\uoZ^\alpha_{\CK'}$ for an appropriate choice
of a $T$-bundle $\CK'_T$ depending on $\CK_T$ and on the level $\CalD$ of the ``moment map''
(Theorem~\ref{reductio}).

B. If $\CK_T$ is regular, the composed isomorphism
$M(\Ind_T^G\CK_T,\CL_T)\simeq{}^{\vphantom{\alpha}}_\CalD\uoZ^\alpha_\CK\simeq{}^C_\CalD\uoZ^\alpha_{\CK'}$
is a symplectomorphism (Theorem~\ref{myksas}).

\medskip

It is also likely that the reduced elliptic zastava $^{\vphantom{\alpha}}_\CalD\uoZ^\alpha_\CK$ is isomorphic
to the moduli space of {\em monowalls} (doubly periodic monopoles)~\cite{cw}.
The situation is similar to the case of periodic monopoles: the monowalls come centered
by definition. In the corresponding elliptic Coulomb branch of a quiver
gauge theory the gauge group must be taken as the product of $\SL(V_i)$ (as opposed to the
product of $\GL(V_i)$ for the nonreduced Coulomb elliptic zastava).

\subsection{An explicit formula for the Feigin-Odesskii Poisson bracket}
\label{explicit FO}
We are finally in a position to address the problem of explicit computation of the
Feigin-Odesskii Poisson bracket. The Coulomb elliptic zastava $^C\!\oZ_\CK^\alpha$ comes equipped
with \'etale rational coordinates that are ``trigonometric Darboux'' for its symplectic form by
the very construction. The usual elliptic zastava also carry \'etale rational coordinates
$(y_{i,r},w_{i,r})_{i\in I}^{1\leq r\leq a_i}$ similar to the ones in~\S\ref{euclid} (but now
$w_{i,r}$ is a point of $E$). The reduced elliptic zastava (alias the Feigin-Odesskii moduli
space in the regular case) inherits these coordinates with the following caveats:

(a) The $w$-coordinates are constrained: for each $i\in I$ the sum $\sum_{r=1}^{a_i}w_{i,r}\in E$
is fixed;

(b) The $y$-coordinates are homogeneous: only the ratios $\frac{y_{i,r}}{y_{i,r'}}$ are well
defined for $i\in I,\ 1\leq r,r'\leq a_i$.

Then the only nontrivial Poisson brackets arising from the Feigin-Odesskii symplectic form are
as follows:

\begin{multline*}
  \Big\{\frac{y_{i,r}}{y_{i,r'}},w_{i,r}\Big\}_{FO}=\frac{y_{i,r}}{y_{i,r'}},\ \hfil
  \Big\{\frac{y_{i,r}}{y_{i,r'}},w_{i,r'}\Big\}_{FO}=-\frac{y_{i,r}}{y_{i,r'}},\ \hfil
  \Big\{\frac{y_{i,r'}}{y_{i,p'}},\frac{y_{j,r}}{y_{j,p}}\Big\}_{FO}=\\
\frac{y_{i,r'}}{y_{i,p'}}\cdot\frac{y_{j,r}}{y_{j,p}}
\big(\zeta(w_{i,r'}-w_{j,r})-\zeta(w_{i,r'}-w_{j,p})-\zeta(w_{i,p'}-w_{j,r})+\zeta(w_{i,p'}-w_{j,p})\big)
\end{multline*}
in case $i\ne j$ are joined by an edge in the Dynkin diagram of $G$, and zero otherwise
(recall that we assume $G$ simply laced). Here $\zeta(w)$ is the Weierstra\ss\ zeta function.

\subsection{Acknowledgments}
This project was initiated together with A.~Kuznetsov in~2000 right after~\cite{fo,fkmm}
appeared. It is clear from the above discussion that too many important concepts were
missing at the time, so the project had to wait for~20 years till its completion.
We are also grateful to A.~Braverman, S.~Cherkis, P.~Etingof, B.~Feigin,
I.~Mirkovi\'c, H.~Nakajima, I.~Perunov, A.~Prikhodko, E.~Rains, L.~Rybnikov, A.~Tsymbaliuk and
V.~Vologodsky for very helpful discussions. The work of M.F.\ and A.P.\ has been funded within the
framework of the HSE University Basic Research Program and the Russian Academic Excellence
Project `5-100'. A.P.\ is also partially supported by the NSF grant DMS-2001224.

\section{Elliptic zastava}
\label{dva}

\subsection{A group $G$}
\label{group}
Let $G$ be an almost simple simply connected algebraic group over $\BC$. We fix a pair of opposite
Borel subgroups $B$, $B_-$ whose intersection is a maximal torus $T$. The unipotent radical
of $B$ (resp.\ $B_-$) is denoted $U$ (resp.\ $U_-$).
Let $\Lambda$ (resp.\ $\Lambda^\vee$) denote the cocharacter (resp.\ character)
lattice of $T$; since $G$ is assumed to be simply connected, this is also the coroot lattice of $G$.
We denote by $\Lambda_{\on{pos}}\subset \Lambda$ the sub-semigroup spanned by positive coroots. We say that
$\alpha\geq \beta$ (for $\alpha,\beta\in \Lambda$)
if $\alpha-\beta\in\Lambda_{\on{pos}}$. The simple coroots are $\{\alpha_i\}_{i\in I}$;
the simple roots are $\{\alpha^\svee_i\}_{i\in I}$;
the fundamental weights are $\{\omega^\svee_i\}_{i\in I}$.
An irreducible $G$-module with a dominant highest weight $\lambda^\svee\in\Lambda^{\vee+}$
is denoted $V_{\lambda^\svee}$; we fix its highest vector $v_{\lambda^\svee}$.
For a weight $\mu^\svee\in\Lambda^\vee$ the
$\mu^\svee$-weight subspace of a $G$-module $V$ is denoted $V(\mu^\svee)$.

\subsection{Elliptic zastava}
We recall some results of~\cite{g} about various versions of zastava on a curve.
From now on we always consider an elliptic curve $E$.
We fix a degree zero $T$-torsor $\CK_T$ on $E$.
It gives rise to a collection of line bundles $\CK^{\mu^\svee}$ on $E$ associated to
characters $\mu^\svee\colon T\to\BC^\times$.

\begin{defn}
  \label{ell zas}
  \textup{(1)} Given $\alpha\in\Lambda_{\on{pos}}$, we define the {\em compactified elliptic zastava}
  $\ol{Z}{}_\CK^\alpha$ as the moduli space of the following data:

  \textup{(a)} a $G$-bundle $\CF_G$ on $E$;

  \textup{(b)} a $T$-bundle $\CL_T$ of degree $-\alpha$ on $E$;

  \textup{(c)} for any dominant weight $\lambda^\svee\in\Lambda^{\vee+}$, a nonzero morphism from
  the associated vector bundle
  $\xi^{\lambda^\svee}\colon\CV_\CF^{\lambda^\svee}\to\CK^{\lambda^\svee}$;

  \textup{(d)} for any $\lambda^\svee\in\Lambda^{\vee+}$, a sheaf embedding
  $\eta^{\lambda^\svee}\colon \CL^{\lambda^\svee}\hookrightarrow\CV_\CF^{\lambda^\svee}$,

  \medskip
  
  \noindent subject to the following conditions:

  \medskip
  
  \textup{(i)} the collection of sheaf embeddings
  $\CL^{\lambda^\svee}\hookrightarrow\CV_\CF^{\lambda^\svee}$ satisfy the Pl\"ucker relations,
  i.e.\ define a degree $\alpha$ generalized $B$-structure in $\CF_G$;

  \textup{(ii)} the collection of morphisms
  $\CV_\CF^{\lambda^\svee}\to\CK^{\lambda^\svee}$ satisfy the Pl\"ucker relations,
  i.e.\ define a generalized $\CK$-twisted $U_-$-structure in $\CF_G$;

  \textup{(iii)} the composition
  $\CL^{\lambda^\svee}\hookrightarrow\CV_\CF^{\lambda^\svee}\twoheadrightarrow\CK^{\lambda^\svee}$
  is not zero for any $\lambda^\svee$, i.e.\ the above generalized $B$- and $U_-$-structures are
  generically transversal.

  \medskip

  \textup{(2)} The {\em elliptic zastava} $Z^\alpha_\CK\subset\ol{Z}{}^\alpha_\CK$ is an open
  subspace given by the extra condition that the morphisms
  $\xi^{\lambda^\svee}\colon\CV_\CF^{\lambda^\svee}\to\CK^{\lambda^\svee}$ are surjective,
  i.e.\ the corresponding twisted $U_-$-structure is genuine, not generalized.

  \medskip
  
  \textup{(3)} The {\em open elliptic zastava} $\oZ^\alpha_\CK\subset Z^\alpha_\CK$ is given by the
  extra condition that the embeddings
  $\eta^{\lambda^\svee}\colon\CL^{\lambda^\svee}\hookrightarrow\CV_\CF^{\lambda^\svee}$
  are embeddings of vector bundles, i.e.\ $\CL^{\lambda^\svee}$ is a line subbundle in
  $\CV_\CF^{\lambda^\svee}$ for any $\lambda^\svee\in\Lambda^{\vee+}$. In other words, the
  corresponding $B$-structure is genuine, not generalized.

  \medskip
  
  \textup{(4)} The {\em factorization} morphism $\pi^\alpha\colon \ol{Z}{}^\alpha_\CK\to E^\alpha$
  associates to the data of zastava the $I$-colored divisor $D\in E^\alpha$ such that for any
  $\lambda^\svee\in\Lambda^{\vee+}$, the zero divisor of the composition
  $\CL^{\lambda^\svee}\to\CV_\CF^{\lambda^\svee}\to\CK^{\lambda^\svee}$ equals $\langle D,\lambda^\svee\rangle$.

  \medskip

  \textup{(5)} The Cartan torus $T$ acts on $\ol{Z}{}^\alpha_\CK$ by rescaling the morphisms in
  (c) above: for $t\in T$ we set $t(\xi^{\lambda^\svee}):=\lambda^\svee(t)\cdot\xi^{\lambda^\svee}$.
This action factors through the adjoint quotient $T^{\on{ad}}$.
\end{defn}

\begin{rem}
  \label{reduced zastava}
The moduli stack $\ol{Z}{}^\alpha_\CK$ is actually a finite type scheme, irreducible of dimension
$2|\alpha|$, see e.g.~\cite[\S4, \S7.2]{g}. The open subscheme
$\oZ^\alpha_\CK\subset\ol{Z}{}^\alpha_\CK$ is smooth. The scheme $\ol{Z}{}^\alpha_\CK$
can be nonreduced in general,
cf.~\cite[Example~2.13]{fem} for $G=\SL(5)$. This example features a formal arc scheme, but
according to the Grinberg-Kazhdan theorem and~\cite[\S4.4]{d} it implies that an appropriate
(rational) zastava space $Z^\alpha$ for $G=\SL(5)$ is nonreduced as well. Finally, the rational
zastava $Z^\alpha$ and the elliptic zastava $Z^\alpha_\CK$ are isomorphic locally in the \'etale
topology.

In~\S\ref{tri} we will consider the variety $(\ol{Z}{}^\alpha_\CK)_{\on{red}}$ equipped with the
reduced scheme structure.
\end{rem}

\begin{rem}
  \label{reductive}
  In~\S\ref{shest} we will need elliptic zastava for a {\em reductive} group $\sG$. It is defined
  similarly to~Definition~\ref{ell zas} making use of the trick~\cite[\S7]{sch} with the help
  of a central extension $1\to{\mathcal Z}\to\widehat{\sG}\to\sG\to1$ such that $\mathcal Z$ is
  a (connected) central torus in $\widehat{\sG}$, and the derived subgroup
  $[\widehat{\sG},\widehat{\sG}]\subset\widehat{\sG}$ is simply connected. Namely, we
  apply~Definition~\ref{ell zas} to $\widehat{\sG}$ instead of $\sG$ itself. The result is
  independent of the choice of $\widehat{\sG}$ and gets rid of some undesirable irreducible
  components that appear if we naively apply~Definition~\ref{ell zas} to $\sG$ itself.  
\end{rem}

The following definition is motivated by the notion of {\em centered} euclidean monopoles.

\begin{defn}
  \label{red zas}
  We have the Abel-Jacobi morphisms $E^{(a_i)}\to\on{Pic}^{a_i}E$ and their product
$\on{AJ}\colon E^\alpha\to\prod_{i\in I}\on{Pic}^{a_i}E$. We denote the composed morphism by
  \[\on{AJ}_Z\colon\oZ^\alpha_\CK\stackrel{\pi^\alpha}{\longrightarrow}
  E^\alpha\stackrel{\on{AJ}}{\longrightarrow}\prod_{i\in I}\on{Pic}^{a_i}E.\]
Given a collection $\CalD=(\CalD_i)_{i\in I}\in\on{Pic}^{a_i}E$, we define the
{\em reduced open elliptic zastava}
$^{\vphantom{\alpha}}_\CalD\uoZ^\alpha_\CK$ as $\on{AJ}_Z^{-1}(\CalD)/T$ (stack quotient).
\end{defn}

The reduced open elliptic zastava $^{\vphantom{\alpha}}_\CalD\uoZ^\alpha_\CK$ is an irreducible
stack.\footnote{Indeed, a general fiber of $\pi^\alpha$ is isomorphic to $\BG_m^{|\alpha|}$, hence
  irreducible. Any fiber of $\on{AJ}$ is irreducible as well. Finally, all the fibers of $\on{AJ}_Z$
  are smooth equidimensional by a computation of the differential of $\on{AJ}_Z$. Hence any fiber
  of $\on{AJ}_Z$ is irreducible.}
Let $\alpha=\sum_{i\in I}a_i\alpha_i$. If $a_i=0$ for some $i\in I$, then all the zastava
spaces $\ol{Z}{}^\alpha_\CK,\ Z^\alpha_\CK,\ \oZ^\alpha_\CK,\ ^{\vphantom{\alpha}}_\CalD\uoZ^\alpha_\CK$ coincide
with the corresponding zastava spaces for the derived group of the corresponding Levi factor
of $G$. If $a_i>0$ for all $i\in I$, then the action of $T^{\on{ad}}$ on the open elliptic zastava
$\oZ^\alpha_\CK$ is effective, and the dimension of $^{\vphantom{\alpha}}_\CalD\uoZ^\alpha_\CK$ is
$2|\alpha|-2\on{rk}G$.

\begin{rem}
  \label{triv can}
  Throughout the paper we will use a trivialization of the canonical line bundle
  $\bomega_E$. We fix this trivialization once and for all.
\end{rem}

\subsection{Example of $G=\SL(2)$ and Hilbert schemes}
\label{sl2 vs hilb}
We denote by $\omega^\svee$ the fundamental weight of $G=\SL(2)$, and we denote by
$\alpha^\svee=2\omega^\svee$ the simple root of $G$. We denote by $\alpha$ the simple coroot
of $G$. We denote the total space of the line bundle $\CK^{-\alpha^\svee}$ over $E$ by
$S_{\CK^{-\alpha^\svee}}$, and we denote the complement to the zero section by $\oS_{\CK^{-\alpha^\svee}}$.
These are algebraic surfaces equipped with a projection to $E$.
For $a\in\BN$, we denote $\ol{Z}{}^{a\alpha}_\CK$ simply by $\ol{Z}{}^a_\CK$.
We denote by $\Hilb^a(S_{\CK^{-\alpha^\svee}})\supset\Hilb^a(\oS_{\CK^{-\alpha^\svee}})$
the degree $a$ Hilbert schemes of points on the surfaces
$S_{\CK^{-\alpha^\svee}}\supset \oS_{\CK^{-\alpha^\svee}}$. We denote by
$\Hilb^a_\tr(S_{\CK^{-\alpha^\svee}})\subset\Hilb^a(S_{\CK^{-\alpha^\svee}})$
(resp.\ $\Hilb^a_\tr(\oS_{\CK^{-\alpha^\svee}})\subset\Hilb^a(\oS_{\CK^{-\alpha^\svee}})$)
the open {\em transversal Hilbert subscheme}
classifying all quotients of $\CO_{S_{\CK^{-\alpha^\svee}}}$ (resp.\ of $\CO_{\oS_{\CK^{-\alpha^\svee}}}$) whose
direct images to $E$ are also cyclic, i.e.\ are quotients of $\CO_E$.

Thus we have projections
$\Hilb^a_\tr(\oS_{\CK^{-\alpha^\svee}})\to\Hilb^a(E)=E^{(a)}\leftarrow\Hilb^a_\tr(S_{\CK^{-\alpha^\svee}})$.
The transversal Hilbert scheme $\Hilb^a_\tr(S_{\CK^{-\alpha^\svee}})$ is canonically isomorphic to the total
space of the following vector bundle $\CU_\CK$ on $E^{(a)}$. Let
  $\bq\colon E\times E^{(a-1)}\to E^{(a)}$ be the addition morphism (aka the universal
family over $\Hilb^a(E)=E^{(a)}$). Then $\CU_\CK:=\bq_*\pr_E^*\CK^{-\alpha^\svee}$.
We will also need another closely related vector bundle on $E^{(a)}$. Namely, let
$\Delta^{1,a-1}\subset E\times E^{(a-1)}$ be the incidence divisor (note that the line bundle
$\CO(\Delta^{1,a-1})$ on $E\times E^{(a-1)}$ is isomorphic to the normal bundle
to the closed embedding $E\times E^{(a-1)}\hookrightarrow E\times E^{(a)},\ (x,D')\mapsto(x,x+D')$,
see e.g.~\cite[Proposition 19.1]{p}).
We set $\CT_\CK:=\bq_*(\pr_E^*\CK^{-\alpha^\svee}\otimes\CO(\Delta^{1,a-1}))$.
Note that in case $\CK$ is trivial,
the corresponding vector bundle $\CT$ is nothing but the tangent bundle of $E^{(a)}$, and
the corresponding vector bundle $\CU$ is dual to $\CT$, i.e.\ $\CU\simeq\CT^*$ is the
cotangent bundle of $E^{(a)}$.

Furthermore, we have the Abel-Jacobi morphism $E^{(a)}\to\on{Pic}^a(E)$. For an arbitrary
line bundle $\CK'$ on $E$, we denote the composed
morphism by $\on{AJ}\colon\Hilb^a_\tr(\oS_{\CK'})\to E^{(a)}\to\on{Pic}^a(E)$.
For a line bundle $\CalD$ of degree $a$ on $E$,
the fiberwise dilation action of $\BC^\times$ on $\oS_{\CK'}$ induces an
action of $\BC^\times$ on $\on{AJ}^{-1}(\CalD)\subset
\Hilb^a_\tr(\oS_{\CK'})$,
and we define the {\em reduced transversal Hilbert scheme}
$^{\vphantom{\alpha}}_\CalD\ul\Hilb{}_\tr^a(\oS_{\CK'})$ as
$\on{AJ}^{-1}(\CalD)/\BC^\times$ (stack quotient).

\begin{prop}
  \label{hilb vs zas}
  \textup{(a)} There are natural isomorphisms \[\oZ^1_\CK\cong \oS_{\CK^{-\alpha^\svee}},\
  Z^1_\CK\cong S_{\CK^{-\alpha^\svee}},\ \ol{Z}{}^1_\CK\cong\BP(\CK^{-\alpha^\svee}\oplus\CO_E).\]

  \textup{(b)} For $a\in\BN$, the zastava space $Z^{a}_\CK$ is naturally isomorphic to the
  total space of the vector bundle $\CT_\CK$ on $E^{(a)}$.

  \textup{(c)} For $a\in\BN,\ \CalD\in\on{Pic}^a(E)$, the reduced open zastava
  $^{\vphantom{\alpha}}_\CalD\uoZ^{a}_\CK$ is naturally isomorphic to the reduced transversal
  Hilbert scheme $^{\vphantom{\alpha}}_\CalD\ul\Hilb{}_\tr^a(\oS_{\CK'})$ for
  $\CK'=\CK^{-\alpha^\svee}\otimes\CalD$.
\end{prop}

\begin{proof}
  By definition, $\ol{Z}{}^{a}_\CK$ is the moduli space of the data
  $\CL^{\omega^\svee}\to\CV_\CF^{\omega^\svee}\to\CK^{\omega^\svee}$ such that the composition
  $\CL^{\omega^\svee}\to\CK^{\omega^\svee}$ is not zero. Here $\CV_\CF^{\omega^\svee}$ is a vector bundle
  on $E$ of rank 2 with trivialized determinant, and $\CL^{\omega^\svee}$ is a line bundle of
  degree $-a$. Hence the composition $\CL^{\omega^\svee}\hookrightarrow\CK^{\omega^\svee}$ identifies
  $\CL^{\omega^\svee}$ with $\CK^{\omega^\svee}(-D)$ for an effective divisor $D$ on $E$ of
  degree $a$. The trivialization of $\det\CV_\CF^{\omega^\svee}$ makes $\CV_\CF^{\omega^\svee}$ canonically
  selfdual, so the dual of our data is $\CK^{-\omega^\svee}\to\CV_\CF^{\omega^\svee}\to\CL^{-\omega^\svee}$.
  In particular, we obtain the sheaf embeddings \[\CK^{-\omega^\svee}\oplus\CK^{\omega^\svee}(-D)=
  \CK^{-\omega^\svee}\oplus\CL^{\omega^\svee}\hookrightarrow\CV_\CF^{\omega^\svee}\hookrightarrow
  \CL^{-\omega^\svee}\oplus\CK^{\omega^\svee}=\CK^{-\omega^\svee}(D)\oplus\CK^{\omega^\svee}.\]
  In other words, $\CV_\CF^{\omega^\svee}$ is a degree $a$ upper modification of
  $\CK^{-\omega^\svee}\oplus\CK^{\omega^\svee}(-D)$ at $D$. The open subvariety
  $Z^{a}_\CK\subset\ol{Z}{}^{a}_\CK$ is given by the open condition that the projection
  of $\CV_\CF^{\omega^\svee}$ to $\CK^{\omega^\svee}$ is surjective, and the open subvariety
  $\oZ^{a}_\CK\subset Z^{a}$ is given by the extra open condition that the projection of
  $\CV_\CF^{\omega^\svee}$ to $\CK^{-\omega^\svee}(D)$ is surjective. Yet in other words,
  $\ol{Z}{}^{a}_\CK$ is the moduli space of $a$-dimensional $\CO_E$-submodules
  $V\subset(\CK^{-\omega^\svee}(D)/\CK^{-\omega^\svee})\oplus(\CK^{\omega^\svee}/\CK^{\omega^\svee}(-D))$,
  the open subvariety $Z^{a}_\CK\subset\ol{Z}{}^{a}_\CK$ is given by the open condition that
  $V$ is transversal to $\CK^{-\omega^\svee}(D)/\CK^{-\omega^\svee}$, and the open subvariety
  $\oZ^{a}_\CK\subset Z^{a}_\CK$ is given by the extra open condition that $V$ is transversal
  to $\CK^{\omega^\svee}/\CK^{\omega^\svee}(-D)$.

  If $a=1$, then $D$ is a single point $x\in E$, and the fiber of $\ol{Z}{}^1_\CK$ over
  $x\in E$ is a projective line
  $\BP\big((\CK^{-\omega^\svee}(x)/\CK^{-\omega^\svee})\oplus(\CK^{\omega^\svee}/\CK^{\omega^\svee}(-x))\big)$.
  Hence $\ol{Z}{}^1_\CK$ is the projectivization of the rank 2 vector bundle
  $\CK^{-\omega^\svee}\otimes\CT_E\oplus\CK^{\omega^\svee}$ over $E$. The trivialization of
  the canonical line bundle $\bomega_E$ in~Remark~\ref{triv can} gives rise to a trivialization
  of the tangent line bundle $\CT_E$, and we obtain an isomorphism
  $\ol{Z}{}^1_\CK\cong\BP(\CK^{-\omega^\svee}\oplus\CK^{\omega^\svee})=
  \BP(\CK^{-\alpha^\svee}\oplus\CO_E)$. Furthermore, a point of $Z^1_\CK$ over $x\in E$ can
  be viewed as the graph of a homomorphism from $\CK^{\omega^\svee}_x$ to $\CK^{-\omega^\svee}_x$, so
  $Z^1_\CK$ gets identified with the total space of the line bundle
  $\CH om(\CK^{\omega^\svee},\CK^{-\omega^\svee})=\CK^{-\alpha^\svee}$. Finally, a point of $\oZ^1_\CK$
  over $x\in E$ can be viewed as the graph of an isomorphism from $\CK^{\omega^\svee}_x$ to
  $\CK^{-\omega^\svee}_x$. This completes our proof of (a).

\medskip

  Recall that the fiber of $\Hilb^a_\tr(S_{\CK^{-\alpha^\svee}})$ (respectively, of $\Hilb^a_\tr(\oS_{\CK^{-\alpha^\svee}})$) 
  over $D\in E^{(a)}$ is canonically isomorphic to
  $\Hom_{\CO_E}(\CO_D,\CK^{-\alpha^\svee}/\CK^{-\alpha^\svee}(-D))$ (respectively, to
  $\Isom_{\CO_E}(\CO_D,\CK^{-\alpha^\svee}/\CK^{-\alpha^\svee}(-D))$), where $\CO_D=\CO_E/\CO_E(-D)$.
  On the other hand, an $a$-dimensional $\CO_E$-submodule
  $V\subset(\CK^{-\omega^\svee}(D)/\CK^{-\omega^\svee})\oplus(\CK^{\omega^\svee}/\CK^{\omega^\svee}(-D))$
  transversal to $\CK^{-\omega^\svee}(D)/\CK^{-\omega^\svee}$ is the graph of a homomorphism
  $h_V\in\Hom_{\CO_E}(\CK^{\omega^\svee}/\CK^{\omega^\svee}(-D),\CK^{-\omega^\svee}(D)/\CK^{-\omega^\svee})=
  \Hom_{\CO_E}(\CO_D,\CK^{-\alpha^\svee}(D)/\CK^{-\alpha^\svee})$.
  Furthermore, $V$ is
  also transversal to $\CK^{\omega^\svee}/\CK^{\omega^\svee}(-D)$ iff $h_V$ is invertible.

  Since $\CO_D$ is a cyclic $\CO_E$-module with generator 1, a homomorphism
  $h_V\in\Hom_{\CO_E}(\CO_D,\CK^{-\alpha^\svee}(D)/\CK^{-\alpha^\svee})$ is uniquely determined
  by $h_V(1)$, so that $\Hom_{\CO_E}(\CO_D,\CK^{-\alpha^\svee}(D)/\CK^{-\alpha^\svee})=
  \CK^{-\alpha^\svee}(D)/\CK^{-\alpha^\svee}$, and the latter space is nothing but the fiber of
  the vector bundle $\CT_\CK$ at $D\in E^{(a)}$. This completes the proof of~(b).

  \medskip

  We have just seen that the fiber of $\oZ^{a}_\CK$ over $D\in E^{(a)}$ is canonically
  isomorphic to $\Isom_{\CO_E}(\CO_D,\CK^{-\alpha^\svee}(D)/\CK^{-\alpha^\svee})$. If $D$ runs over
  the fiber of the Abel-Jacobi map over $\CalD=\CK^{\omega^\svee}\otimes\CL^{-\omega^\svee}$, then
  $\CK^{-\alpha^\svee}(D)/\CK^{-\alpha^\svee}\simeq(\CK^{-\alpha^\svee}\otimes\CalD)|_D$,
  and the isomorphism is well defined up to a multiplicative constant.
  Hence $\Isom_{\CO_E}(\CO_D,\CK^{-\alpha^\svee}(D)/\CK^{-\alpha^\svee})\simeq
  \Isom_{\CO_E}(\CO_D,(\CK^{-\alpha^\svee}\otimes\CalD)|_D)$,
  and the isomorphism is well defined up to a multiplicative constant.
  The latter space is the fiber of $\Hilb{}_\tr^a(\oS_{\CK^{-\alpha^\svee}\otimes\CalD})$ over $D$.
  Finally, taking quotient by the action of $\BC^\times$ removes the ambiguity in the choice
  of the above isomorphism, and produces the desired canonical isomorphism.

  The above argument generalizes straightforwardly to the case of families over a base $B$.
  For example, the isomorphism
  $\Isom_{\CO_{E\times B}}(\CO_{D\times B},\CK^{-\alpha^\svee}(D\times B)/\CK^{-\alpha^\svee})\simeq
  \Isom_{\CO_{E\times B}}(\CO_{D\times B},(\CK^{-\alpha^\svee}\otimes\CalD)|_{D\times B})$ is well defined
  up to $\CO^\times_B$.
  
  This completes the proof of~(c).
  \end{proof}

\section{Mirkovi\'c construction}
\label{tri}
From now on we assume that $G$ is simply laced. We choose an orientation of the Dynkin diagram of
$G$. We obtain a quiver $Q$ with the set of vertices $Q_0=I$, and the set of arrows $Q_1$. For an
arrow $h=(i\to j)$ we use the standard notation $j=\on{i}(h),\ i=\on{o}(h)$.

\subsection{Compactified zastava}
\label{mir com}
For a $T$-torsor $\CK_T$ on $E$ and $i\in I$, we define a line bundle $\CK_i$ on $E$ associated
to the simple root character $\alpha^\svee_i\colon T\to\BC^\times$.
Given a collection of line bundles $\CK_i,\ i\in I$, and $\beta=\sum b_i\alpha_i\in\Lambda_{\on{pos}}$, we
define a line bundle $\CK^\beta:=\boxtimes_{i\in I}\CK_i^{(b_i)}$ on $E^\beta=\prod_{i\in I}E^{(b_i)}$.
Here $\CK_i^{(b_i)}$ is the descent of $\CK_i^{\boxtimes b_i}$ from $E^{b_i}$ to $E^{(b_i)}$ obtained
by passing to $S_{b_i}$-invariant sections on $U^{(b_i)}$, where $U\subset E$ is an affine open
subset.
Given $\beta,\gamma\in\Lambda_{\on{pos}}$ with $\beta+\gamma=\alpha$, we consider the diagram
\[E^\beta\xleftarrow{\bp}E^\beta\times E^\gamma\xrightarrow{\bq}E^\alpha,\]
where $\bp$ is the projection, and $\bq$ is the addition of colored effective divisors.
For $i,j\in I$ we define $\Delta_{ij}^{\beta,\gamma}\subset E^\beta\times E^\gamma$ as the incidence divisor
where a point of color $i$ in $E^\beta$ meets a point of color $j$ in $E^\gamma$ (the case $j=i$ is
allowed).
We also define $\Delta_{ij}^\beta\subset E^\beta$ as the divisor formed by configurations where a point
of color $i$ meets a point of color $j$.
We define the factorizable vector bundle $\BV_\CK^\alpha$ on $E^\alpha$ as\footnote{Our definition looks
  different from~\cite[\S\S2.4.1,2.4.2]{myz}. This is due to dualization, cf.~Lemma~\ref{duo} below.}
\begin{equation}
  \label{mir bun}
  \BV_\CK^\alpha:=
  \bigoplus_{\beta+\gamma=\alpha}\bq_*\left(\bp^*\Big(\CK^\beta\big(\sum_{i\in I}\Delta_{ii}^\beta
  -\sum_{h\in Q_1}\Delta_{\on{o}(h)\on{i}(h)}^\beta\big)\Big)\big(\sum_{i\in I}\Delta^{\beta,\gamma}_{ii}\big)\right).  
\end{equation}
It contains two codimension 1 subbundles: $\BV_{\CK,\on{low}}^\alpha$ and $\BV_\CK^{\alpha,\on{up}}$, where in
the above direct sum we omit summands corresponding to $\beta=0$ (resp.\ $\gamma=0$).

The factorization structure is a canonical isomorphism for any decomposition
$\alpha=\alpha'+\alpha''$, between the pullbacks of $\BV_\CK^\alpha$ and
$\BV_\CK^{\alpha'}\boxtimes\BV_\CK^{\alpha''}$ to $(E^{\alpha'}\times E^{\alpha''})_{\on{disj}}$ (an open subset of
$E^{\alpha'}\times E^{\alpha''}$ formed by all the pairs of configurations where all the points of the
first configuration are distinct from all the points of the second one).
In particular, the rank of $\BV_\CK^\alpha$ equals $2^{|\alpha|}$, and the pullback of
$\BV_\CK^\alpha$ to $(\prod_{i\in I}E^{a_i})_{\on{disj}}$ is canonically isomorphic to
$\boxtimes_{i\in I}((\CK_i\oplus\CO_E)^{\boxtimes a_i})|_{(\prod_{i\in I}E^{a_i})_{\on{disj}}}$
(here $\alpha=\sum_{i\in I}a_i\alpha_i$). Let
$p^\alpha\colon (\prod_{i\in I}E^{a_i})_{\on{disj}}\to E^\alpha_{\on{disj}}$ stand for the unramified
Galois cover with Galois group $S_\alpha=\prod_{i\in I}S_{a_i}$ (the product of symmetric groups).
Then the vector bundle
$\boxtimes_{i\in I}((\CK_i\oplus\CO_E)^{\boxtimes a_i})|_{(\prod_{i\in I}E^{a_i})_{\on{disj}}}$ carries a
natural $S_\alpha$-equivariant structure, and $\BV_\CK^\alpha|_{E^\alpha_{\on{disj}}}
=\big(p^\alpha_*\boxtimes_{i\in I}((\CK_i\oplus\CO_E)^{\boxtimes a_i})|_{(\prod_{i\in I}E^{a_i})_{\on{disj}}}\big)^{S_\alpha}$.

Thus the projectivization
$\BP\big(\boxtimes_{i\in I}((\CK_i\oplus\CO_E)^{\boxtimes a_i})\big)|_{(\prod_{i\in I}E^{a_i})_{\on{disj}}}$ contains
the product of the ruled surfaces ($\BP^1$-bundles over $E$)
$\prod_{i\in I}\BP(\CK_i\oplus\CO_E)^{a_i}|_{(\prod_{i\in I}E^{a_i})_{\on{disj}}}$ (Segre embedding). Hence
$\BP\BV_\CK^\alpha|_{E^\alpha_{\on{disj}}}$ contains
$\big(\prod_{i\in I}\BP(\CK_i\oplus\CO_E)^{a_i}|_{(\prod_{i\in I}E^{a_i})_{\on{disj}}}\big)/S_\alpha$.

\begin{defn} [I.~Mirkovi\'c]
  \label{mir def}
  \textup{(a)} {\em Mirkovi\'c compactified zastava} $^{\on{Mir}}\ol{Z}{}^\alpha_\CK$ is defined as the
    closure of
    $\big(\prod_{i\in I}\BP(\CK_i\oplus\CO_E)^{a_i}|_{(\prod_{i\in I}E^{a_i})_{\on{disj}}}\big)/S_\alpha$ in
    $\BP\BV_\CK^\alpha$ (with the reduced closed subscheme structure).

    \textup{(b)} The {\em upper} (resp.\ {\em lower}) {\em boundary}
    $\partial_{\on{up}}{}^{\on{Mir}}\ol{Z}{}^\alpha_\CK$ (resp.\
    $\partial_{\on{low}}{}^{\on{Mir}}\ol{Z}{}^\alpha_\CK$) is defined as the intersection
    $^{\on{Mir}}\ol{Z}{}^\alpha_\CK\cap\BP\BV_\CK^{\alpha,\on{up}}$ (resp.\
    $^{\on{Mir}}\ol{Z}{}^\alpha_\CK\cap\BP\BV_{\CK,\on{low}}^\alpha$).

    \textup{(c)} {\em Mirkovi\'c zastava} $^{\on{Mir}}\!Z^\alpha_\CK$ is defined as the open
    subscheme in $^{\on{Mir}}\ol{Z}{}^\alpha_\CK$ obtained by removing the upper boundary
    $\partial_{\on{up}}{}^{\on{Mir}}\ol{Z}{}^\alpha_\CK$.

    \textup{(d)} {\em Mirkovi\'c open zastava} $^{\on{Mir}}\!\oZ^\alpha_\CK$ is defined as the open
    subscheme in $^{\on{Mir}}\!Z^\alpha_\CK$ obtained by further removing the lower boundary
    $\partial_{\on{low}}{}^{\on{Mir}}\ol{Z}{}^\alpha_\CK$.
\end{defn}

Returning to the usual compactified zastava (Definition~\ref{ell zas}), we set
\begin{equation}
  \label{393}
  \CK_i:=\CK^{-\alpha^\svee_i}.
\end{equation}
Then the factorization property of zastava along
with~Proposition~\ref{hilb vs zas}(a) gives rise to a canonical isomorphism
$\ol{Z}{}^\alpha_\CK|_{E^\alpha_{\on{disj}}}
\cong\big(\prod_{i\in I}\BP(\CK_i\oplus\CO_E)^{a_i}|_{(\prod_{i\in I}E^{a_i})_{\on{disj}}}\big)/S_\alpha$.
Thus we obtain a birational isomorphism
$\Theta^\circ\colon ^{\on{Mir}}\ol{Z}{}^\alpha_\CK\stackrel{\sim}\dasharrow\ol{Z}{}^\alpha_\CK.$

\begin{thm} [I.~Mirkovi\'c]\cite[2.4.6]{myz}
  \label{mir thm}
  The birational isomorphism $\Theta^\circ$ extends to a regular isomorphism
  $\Theta\colon ^{\on{Mir}}\ol{Z}{}^\alpha_\CK\iso(\ol{Z}{}^\alpha_\CK)_{\on{red}}$ with the
  compactified zastava equipped with the reduced scheme structure.
  Moreover, $\Theta$ restricts to the same named isomorphisms
  $^{\on{Mir}}\!Z^\alpha_\CK\iso(Z^\alpha_\CK)_{\on{red}}$ and also
  $^{\on{Mir}}\!\oZ^\alpha_\CK\iso\oZ^\alpha_\CK$.
\end{thm}

\begin{proof}
For the readers' convenience we sketch a proof. We consider a twisted version $\Gr_{BD,\CK}$ of
the Beilinson-Drinfeld Grassmannian: the moduli space of $|\alpha|$-tuples of points in $E$,
and $G$-bundles $\CF_G$ on $E$ equipped with a rational isomorphism
$\sigma\colon\CF_G\stackrel{\sim}{\dasharrow}\Ind_T^G\CK_T$ regular away from the above
$|\alpha|$-tuple. The product of symmetric groups $S_\alpha\subset S_{|\alpha|}$ acts on
$\Gr_{BD,\CK}$, and we denote by $\Gr^\alpha_{BD,\CK}$ the categorical quotient (partially
symmetrized twisted Beilinson-Drinfeld Grassmannian). The generically transversal generalized
$B$- and twisted $U_-$-structures in the data of zastava define a generic isomorphism
$\CF_G\stackrel{\sim}{\dasharrow}\Ind_T^G\CK_T$; this way we obtain a closed embedding
$\ol{Z}{}^\alpha_\CK\hookrightarrow\Gr^\alpha_{BD,\CK}$.

We consider the corresponding closed embedding of the $T$-fixed point subschemes
$(\ol{Z}{}^\alpha_\CK)^T\hookrightarrow(\Gr^\alpha_{BD,\CK})^T$. One can construct an isomorphism
$(\ol{Z}{}^\alpha_\CK)^T\simeq\bigsqcup_{\beta+\gamma=\alpha}E^\beta\times E^\gamma$. Furthermore,
one can identify the restriction of the ample determinant line bundle $\fL$ on
$\Gr^\alpha_{BD,\CK}$ to the connected component $E^\beta\times E^\gamma\subset(\ol{Z}{}^\alpha_\CK)^T$
with the line bundle $\bp^*\Big(\CK^{-\beta}\big(-\sum_{i\in I}\Delta_{ii}^\beta
+\sum_{h\in Q_1}\Delta_{\on{o}(h)\on{i}(h)}^\beta\big)\Big)$,
cf.~\cite[Proposition~2.4.1]{myz}.

Now consider the restrictions
$\bq_*\fL\to\bq_*\left(\fL|_{\ol{Z}{}^\alpha_\CK}\right)\to\bq_*\left(\fL|_{(\ol{Z}{}^\alpha_\CK)^T}\right)$,
where $\bq\colon\Gr^\alpha_{BD,\CK}\to E^\alpha$ is the natural projection. The composition
$\bq_*\fL\to\bq_*\left(\fL|_{(\ol{Z}{}^\alpha_\CK)^T}\right)$ is surjective since it equals another
composition of
$\bq_*\fL\to\bq_*\left(\fL|_{(\Gr^\alpha_{BD,\CK})^T}\right)\to\bq_*\left(\fL|_{(\ol{Z}{}^\alpha_\CK)^T}\right)$
that is surjective e.g.\ by~\cite{z}. Hence the restriction
$r_0\colon \bq_*\left(\fL|_{\ol{Z}{}^\alpha_\CK}\right)\to\bq_*\left(\fL|_{(\ol{Z}{}^\alpha_\CK)^T}\right)$
is surjective as well.

The restriction of $\bq$ to $\ol{Z}{}^\alpha_\CK$ is the factorization morphism $\pi^\alpha$.
By factorization, a general fiber of $\pi^\alpha$ is isomorphic to a product of projective
lines, and the restriction of $\fL$ to a general fiber is isomorphic to the exterior product
of line bundles $\CO_{\BP^1}(1)$. Hence the restriction $r_0$ to the $T$-fixed points is an
isomorphism over the generic point of $E^\alpha$. If the coherent sheaf
$\bq_*\left(\fL|_{\ol{Z}{}^\alpha_\CK}\right)$ were torsion free, $r_0$ would be injective, and hence
an isomorphism. However, the direct image $\bq_*\left(\fL|_{\ol{Z}{}^\alpha_\CK}\right)$ does have
torsion (essentially due to the nonreducedness of the compactified zastava,
cf.~Remark~\ref{reduced zastava}).

We denote by $\CT_0\subset\bq_*\left(\fL|_{\ol{Z}{}^\alpha_\CK}\right)$ the torsion subsheaf.
We impose the relations $\CT_0$ on the image of the projective embedding of
$\ol{Z}{}^\alpha_\CK$ into $\BP(\bq_*\fL)$. The resulting closed subscheme of $\ol{Z}{}^\alpha_\CK$
is denoted $^{(1)}\ol{Z}{}^\alpha_\CK$. The fixed point subschemes $({}^{(1)}\ol{Z}{}^\alpha_\CK)^T$
and $(\ol{Z}{}^\alpha_\CK)^T$ coincide since the latter one is reduced. Hence the restriction
$r_1\colon \bq_*\left(\fL|_{^{(1)}\ol{Z}{}^\alpha_\CK}\right)\to\bq_*\left(\fL|_{(\ol{Z}{}^\alpha_\CK)^T}\right)$
is surjective. We denote by $\CT_1\subset\bq_*\left(\fL|_{^{(1)}\ol{Z}{}^\alpha_\CK}\right)$ the torsion
subsheaf. We impose the relations $\CT_1$ on the image of the projective embedding of
$^{(1)}\ol{Z}{}^\alpha_\CK$ into $\BP(\bq_*\fL)$. The resulting closed subscheme of
$^{(1)}\ol{Z}{}^\alpha_\CK$ is denoted $^{(2)}\ol{Z}{}^\alpha_\CK$.

Continuing like this we obtain a chain of closed subschemes
\[\ol{Z}{}^\alpha_\CK\supset{}^{(1)}\ol{Z}{}^\alpha_\CK\supset{}^{(2)}\ol{Z}{}^\alpha_\CK\supset\ldots\]
By the noetherian property of $\ol{Z}{}^\alpha_\CK$ this chain stabilizes with certain closed
subscheme to be denoted
$^{(\infty)}\ol{Z}{}^\alpha_\CK\subset\ol{Z}{}^\alpha_\CK$. If this subscheme is not reduced, we apply
the above procedure to $_{(1)}\ol{Z}{}^\alpha_\CK:=\big({}^{(\infty)}\ol{Z}{}^\alpha_\CK\big)_{\on{red}}$
to obtain its closed subscheme $^{(\infty)}_{(1)}\ol{Z}{}^\alpha_\CK$. Continuing like this we
obtain a chain of closed subschemes \[^{(\infty)}\ol{Z}{}^\alpha_\CK\supset
{}^{(\infty)}_{(1)}\ol{Z}{}^\alpha_\CK\supset{}^{(\infty)}_{(2)}\ol{Z}{}^\alpha_\CK\supset\ldots\]
By the noetherian property of $^{(\infty)}\ol{Z}{}^\alpha_\CK$ this chain stabilizes with
certain reduced closed subscheme
to be denoted $^{(\infty)}_{(\infty)}\ol{Z}{}^\alpha_\CK\subset\ol{Z}{}^\alpha_\CK$. Since
$^{(\infty)}_{(\infty)}\ol{Z}{}^\alpha_\CK$ and $\ol{Z}{}^\alpha_\CK$ coincide over the generic
point of $E^\alpha$, the subscheme $^{(\infty)}_{(\infty)}\ol{Z}{}^\alpha_\CK$ must coincide with
$(\ol{Z}{}^\alpha_\CK)_{\on{red}}$.
The restriction morphism $r_\infty\colon
\bq_*\left(\fL|_{^{(\infty)}_{(\infty)}\ol{Z}{}^\alpha_\CK}\right)\to\bq_*\left(\fL|_{(\ol{Z}{}^\alpha_\CK)^T}\right)$
is surjective. By construction, $\bq_*\left(\fL|_{^{(\infty)}_{(\infty)}\ol{Z}{}^\alpha_\CK}\right)$ is torsion
free, so $r_\infty$ is an isomorphism. 
Thus $^{(\infty)}_{(\infty)}\ol{Z}{}^\alpha_\CK$ is embedded into $\BP\BV^\alpha_\CK$, and must coincide
there with the closure of its generic fiber, i.e.\ with $^{\on{Mir}}\ol{Z}{}^\alpha_\CK$.
\end{proof}

\subsection{Example of type $A_1$ for trivial $\CK$ \`a la Mirkovi\'c}
\label{A1 mir}
Recall the setup and notation of~\S\ref{sl2 vs hilb}. We assume $\CK$ is trivial and denote
$\ol{Z}{}^a_\CK$ by $\ol{Z}{}^a$ for short. The argument in the proof
of~Proposition~\ref{hilb vs zas}(a) defines an embedding of $\ol{Z}{}^a$ into the symmetrized
version $\Gr_{\SL(2),E^{(a)}}$ of Beilinson-Drinfeld Grassmannian of $G=\SL(2)$ of degree $a$,
cf.~\cite[\S4, \S7.2]{g}. We consider the determinant (relatively very ample) line bundle $\fL$ on
$\Gr_{\SL(2),E^{(a)}}$ and its restriction to $\ol{Z}{}^a$. The projection $\ol{Z}{}^a\to E^{(a)}$ is
denoted by $\pi^a$. We claim that there is a natural isomorphism
\[(\pi^a_*\fL)^\vee\simeq
\bigoplus_{b+c=a}\bq_*\Big(\bp^*\big(\CO_{E^{(b)}}(\Delta^b)\big)(\Delta^{b,c})\Big)\]
(notation of~\S\ref{mir com}). Indeed, let $(\ol{Z}{}^a)^T$ be the fixed point subscheme of
$\ol{Z}{}^a$. Then $(\ol{Z}{}^a)^T=\bigsqcup_{b+c=a}E^{(b)}\times E^{(c)}$:
to $D_b\in E^{(b)},\ D_c\in E^{(c)}$ we associate the $a$-dimensional vector subspace
\[V_{D_b,D_c}:=\CO_E(D_b)/\CO_E\oplus\CO_E(-D_b)/\CO_E(-D_b-D_c)\subset
(\CO_E(D)/\CO_E)\oplus(\CO_E/\CO_E(-D))\]
(notation of the proof of~Proposition~\ref{hilb vs zas}(a))
and denote the corresponding rank~2 vector bundle on $E$ by $\CV_\CF^{\omega^\svee}$.
The restriction to fixed points induces an isomorphism $\pi^a_*\fL\iso\pi^a_*(\fL|_{(\ol{Z}{}^a)^T})$,
see e.g.~\cite[\S2.4]{myz}. The fiber $\fL_{\CV_\CF^{\omega^\svee}}$ is
$\det^{-1}\!R\Gamma(E,\CV_\CF^{\omega^\svee})$ by definition, so that the fiber
$\fL_{V_{D_b,D_c}}$ equals
\begin{multline*}
  \det{}\!^{-1}H^0\big(E,\CO_E(D_b)/\CO_E\big)
  \otimes\det{}\!^{-1}H^0\big(E,\CO_E(-D_b)/\CO_E(-D_b-D_c)\big)\\
  \otimes\det H^0\big(E,\CO_E/\CO_E(-D_b-D_c)\big)
  =\det{}\!^2H^0\big(E,\CO_E/\CO_E(-D_b)\big)
\end{multline*}
(we are making use of the trivialization of $\bomega_E$ in~Remark~\ref{triv can} and of the
Serre duality to identify $\det^{-1}H^0(D,\CO_D(D))$ with $\det H^0(D,\CO_D)$). The latter line
is canonically isomorphic to the fiber of $\bfomega_{E^{(b)}}^2$ at $D_b\in E^{(b)}$. We conclude
that $\pi^a_*\fL=\bigoplus_{b+c=a}\bq_*(\bp^*\bfomega^2_{E^{(b)}})$. Furthermore, the dual vector
bundle of $\bq_*(\bp^*\bfomega^2_{E^{(b)}})$ is
$\bq_*\big(\bp^*\bfomega^{-2}_{E^{(b)}}(\Delta^{b,c})\big)$ by the relative
Grothendieck-Serre duality for $\bq$ since $\Delta^{b,c}$ is the ramification divisor of $\bq$.
Finally, $\bfomega^{-2}_{E^{(b)}}=\CO_{E^{(b)}}(\Delta^b)$.

\subsection{Example of type $A_2$ for trivial $\CK$ \`a la Mirkovi\'c}
\label{A2 mir}
In this section $I$ consists of two vertices $i,j$ connected by a single arrow $i\to j$, and
$\alpha=\alpha_i+\alpha_j$. We assume $\CK$ is trivial and denote $\ol{Z}{}^\alpha_\CK$ by
$\ol{Z}{}^\alpha$ for short. We consider the embedding of $\ol{Z}{}^\alpha$ into the
Beilinson-Drinfeld Grassmannian $\Gr_{\SL(3),E^2}$ of degree 2, cf.~\cite[\S4, \S7.2]{g}.
We consider the determinant (relatively very ample) line bundle $\fL$ on $\Gr_{\SL(3),E^2}$
and its restriction to $\ol{Z}{}^\alpha$. The projection $\ol{Z}{}^\alpha\to E\times E$ is
denoted by $\pi^\alpha$. We have
\[(\pi^\alpha_*\fL)^\vee=\CO_{E\times E}^{\oplus3}\oplus\CO_{E\times E}(-\Delta_{ij}).\]
Indeed, let $(\ol{Z}{}^\alpha)^T$ be the fixed point subscheme of $\ol{Z}{}^\alpha$.
Then $(\ol{Z}{}^\alpha)^T$ is isomorphic to the disjoint union of 4 copies of $E\times E$.
Namely, let $v_1,v_2,v_3$ denote the standard basis in the tautological representation of
$\SL(3)$ (so that $T$ acts diagonally). Let us think of points of
$\ol{Z}{}^\alpha\subset\Gr_{\SL(3),E^2}$ as of vector bundles $\CV$ on $E$ identified with
$\CO_Ev_1\oplus\CO_Ev_2\oplus\CO_Ev_3$ away from points $x_i,x_j\in E$. Then:\\
the first copy of $E\times E$ consists of $\CV=\CO_Ev_1\oplus\CO_Ev_2\oplus\CO_Ev_3$;\\
the second copy of $E\times E$ consists of $\CV=\CO_E(x_i)v_1\oplus\CO_E(-x_i)v_2\oplus\CO_Ev_3$;\\
the third copy of $E\times E$ consists of $\CV=\CO_Ev_1\oplus\CO_E(x_j)v_2\oplus\CO_E(-x_j)v_3$;\\
the fourth copy of $E\times E$ consists of
$\CV=\CO_E(x_i)v_1\oplus\CO_E(x_j-x_i)v_2\oplus\CO_E(-x_j)v_3$.

The restriction to fixed points induces an isomorphism
$\pi^\alpha_*\fL\iso\pi^\alpha_*(\fL|_{(\ol{Z}{}^\alpha)^T})$, see e.g.~\cite[\S2.4]{myz}.
The fiber $\fL_\CV$ is $\det\!^{-1}R\Gamma(E,\CV)$ by definition. The restriction
of $\fL$ to the first three copies of $E\times E$ is trivial, while the restriction of
$\fL$ to the fourth copy of $E\times E$ is $\CO_{E\times E}(\Delta_{ij})$.

\subsection{Example of type $A_1$ for regular $\CK$}
\label{A1 reg}
We consider the situation complementary to the one of~\S\ref{A1 mir}: we assume that
$\CK^2$ is {\em nontrivial}. The open elliptic zastava of degree $a$ is the moduli space
$\oZ_\CK^a$ of line subbundles $\CL\subset\CK\oplus\CK^{-1}$ of degree
$-a$. In other words, $\oZ_\CK^a$ is the moduli space of triples
$(\CL,\ s\in H^0(E,\CL^{-1}\CK),\ t\in H^0(E,\CL^{-1}\CK^{-1}))$ such that $s$ and $t$ have no
common zeros, viewed up to common rescaling. The factorization morphism
$\pi^a\colon\oZ_\CK^a\to E^{(a)}$ associates to $(\CL,s,t)$ the zero divisor $D$ of $s$.
We set $t':=t/s\in H^0(E,\CK^{-2}(D))$, a regular section that does not vanish on $D$.
We can also view $\oZ_\CK^a$ as the moduli space of triples $(\CL,D,t')$.
We have an embedding
\[\Upsilon_{t'}=\begin{pmatrix}1&t'\\ 0&1\end{pmatrix}\colon\CK^{-1}\oplus\CK(-D)\to\CK^{-1}\oplus\CK.\]
We consider the determinant line bundle $\fL$ on $\oZ_\CK^a$ whose fiber at
$(\CL,D,t')$ is $\det^{-1}H^0(E,\on{Coker}(\Upsilon_{t'}))$. Consider the dual map
$\Upsilon_{t'}^\vee\colon\CK^{-1}\oplus\CK\hookrightarrow\CK^{-1}(D)\oplus\CK$.
Then $H^0(E,\on{Coker}(\Upsilon_{t'}))$ gets identified with an $a$-dimensional subspace in
$H_D:=H^0\big(E,(\CK^{-1}(D)/\CK^{-1})\oplus(\CK/\CK(-D))\big)$.
This defines an embedding of $\oZ_\CK^a$ into a relative Grassmannian over $E^{(a)}$.
The closure of $\oZ_\CK^a$ in this relative Grassmannian is nothing but the
compactified zastava $\ol{Z}{}_\CK^a$. The determinant line bundle $\fL$ extends to the same
named line bundle on $\ol{Z}{}_\CK^a$. The fixed point subscheme $\left(\ol{Z}{}_\CK^a\right)^T$
(with respect to the Cartan torus $T\subset\SL(2)$) is finite over $E^{(a)}$, and the
restriction morphism
\begin{equation}
  \label{uno}
  \pi^a_*\fL\to\pi^a_*(\fL|_{\left(\ol{Z}{}_\CK^a\right)^T})
\end{equation}
is an isomorphism.

We set $\BV_{\CK^{-2}}^a=\bigoplus_{b+c=a}
\bq_*\left(\bp^*\Big((\CK^{-2})^{(b)}\big(\Delta^b\big)\Big)\big(\Delta^{b,c}\big)\right)$
(notation of~\S\ref{A1 mir}). We also consider a line bundle $\CM$ on $E^{(a)}$ with the
fiber $\det^{-1}H^0(D,\CK|_D)$ over $D\in E^{(a)}$. We will need the following well known result.

\begin{lem}
  \label{well known}
  For any $b>0$, there is an isomorphism $\bomega_{E^{(b)}}^{-2}\simeq\CO_{E^{(b)}}(\Delta^b)$.
\end{lem}

\begin{proof}
  We were unable to locate a reference, so we give a proof. Let $p\colon E^b\to E^{(b)}$ be
  the natural symmetrization morphism. We have a natural map
  $\bomega^{-1}_{E^b}\to p^*\bomega^{-1}_{E^{(b)}}$ vanishing on the union of diagonals in $E^b$.
  Thus, if $v$ is a global nonvanishing differential on $E$, then
  $s=p_1^*\wedge\ldots\wedge p_b^*v$ can be viewed as a global section of
  $p^*\bomega^{-1}_{E^{(b)}}$ (here $p_r\colon E^b\to E$ is the projection to the $r$-th factor).
  A local computation shows that $s^2$ comes from a global section of $\bomega^{-2}_{E^{(b)}}$
  vanishing on $\Delta^b$. This gives the required isomorphism.
\end{proof}

  Now we are in a position to identify the direct image of the determinant line bundle.

\begin{lem}
  \label{duo}
  We have an isomorphism $\pi^a_*(\fL|_{\left(\ol{Z}{}_\CK^a\right)^T})\simeq
  \CM\otimes\left(\BV_{\CK^{-2}}^a\right)^\vee$.
\end{lem}

\begin{proof}
For every splitting $D=D_b+D_c$ into the sum of effective divisors of degrees $b,c$, we have
a $T$-fixed point in $\ol{Z}{}_\CK^a$ corresponding to the subspace
\begin{multline*}
  H^0\big(E,(\CK^{-1}(D_b)/\CK^{-1})\oplus(\CK(-D_b)/\CK(-D))\big)\\
  \subset H^0\big(E,(\CK^{-1}(D)/\CK^{-1})\oplus(\CK/\CK(-D))\big)=H_D.
\end{multline*}
This gives rise to an isomorphism
$\tilde\bq\colon\bigsqcup_{b+c=a}E^{(b)}\times E^{(c)}\iso\left(\ol{Z}{}_\CK^a\right)^T$,
where $\pi^a\tilde\bq=\bq$.

In order to calculate $\tilde\bq_*\fL$, note that by the Serre duality on $D_b$ we have
\[H^0(E,\CK^{-1}(D_b)/\CK^{-1})=H^0(D_b,\CK^{-1}(D_b)|_{D_b})\simeq
H^0(D_b,\bomega_{D_b}\otimes\CK(-D_b)|_{D_b})^\vee.\]
Furthermore, by adjunction we have $\bomega_{D_b}\simeq\bomega_E(D_b)|_{D_b}\simeq\CO_E(D_b)|_{D_b}$.
Thus we get a natural isomorphism
$\det^{-1}H^0(E,\CK^{-1}(D_b)/\CK^{-1})\simeq\det H^0(D_b,\CK|_{D_b})$. The exact sequence
\[0\to\CK(-D_b)/\CK(-D)\to\CK/\CK(-D)\to\CK/\CK(-D_b)\to0\]
gives rise to an isomorphism
\[\det{}\!^{-1}H^0(E,\CK(-D_b)/\CK(-D))\simeq
\det{}\!^{-1}H^0(D,\CK|_D)\otimes\det H^0(D_b,\CK|_{D_b}).\]
Hence we deduce an isomorphism
\[\tilde\bq^*\fL|_{(D_b,D_c)}\simeq\det{}\!^{-1}H^0(D,\CK|_D)\otimes\det{}\!^2H^0(D_b,\CK|_{D_b}).\]
In other words, \[\tilde\bq^*\fL\simeq\bp^*\det{}\!^2\varpi_{E^{(b)}*}\varpi_E^*\CK\otimes\bq^*\CM,\]
where $\varpi_{E^{(b)}}\colon\fD_b\to E^{(b)}$ is the universal divisor, and
$\varpi_E\colon\fD_b\to E$ is the natural projection, while
$\CM=\det^{-1}\varpi_{E^{(a)}*}\varpi_E^*\CK$.

From the natural isomorphisms
\[\det\varpi_{E^{(b)}*}\varpi_E^*\CK\simeq
\on{Nm}_{\fD_b/E^{(b)}}(\varpi_E^*\CK)\otimes\det\varpi_{E^{(b)}*}\CO_{\fD_b}\simeq
\CK^{(b)}\otimes\bomega_{E^{(b)}}\]
we deduce an isomorphism
$\tilde\bq^*\fL\simeq\bq^*\CM\otimes\bp^*\big((\CK^2)^{(b)}\otimes\bomega_{E^{(b)}}^2\big)$.
Summing up over all decompositions $b+c=a$ we get an isomorphism
\[\pi^a_*(\fL|_{\left(\ol{Z}{}_\CK^a\right)^T})\simeq\CM\otimes
\bigoplus_{b+c=a}\bq_*\bp^*\big((\CK^2)^{(b)}\otimes\bomega_{E^{(b)}}^2\big).\]
Using relative Serre duality for $\bq$ and an isomorphism of the relative dualizing sheaf for
$\bq$ with $\CO_{E^{(b)}\times E^{(c)}}(\Delta^{b,c})$ we get an isomorphism
\[\Big(\bq_*\bp^*\big((\CK^2)^{(b)}\otimes\bomega_{E^{(b)}}^2\big)\Big)^\vee\simeq
\bq_*\bp^*\big((\CK^{-2})^{(b)}\otimes\bomega_{E^{(b)}}^{-2}\big)\big(\Delta^{b,c}\big).\]
Finally, using the isomorphism $\bomega_{E^{(b)}}^{-2}\simeq\CO_{E^{(b)}}(\Delta^b)$
of~Lemma~\ref{well known}, we identify the RHS with the corresponding summand in $\BV_{\CK^{-2}}^a$.

The lemma is proved.  
\end{proof}

\subsubsection{Identification of $\ol{Z}{}_\CK^a$ with Mirkovi\'c zastava}
From Lemma~\ref{duo} we obtain an embedding of $\ol{Z}{}_\CK^a$ into
$\BP\left(\CM^\vee\otimes\BV_{\CK^{-2}}^a\right)\simeq\BP\left(\BV_{\CK^{-2}}^a\right)$. We want to
calculate this morphism explicitly away from the diagonals.

First, we find an explicit inverse of the isomorphism~(\ref{uno}) over an \'etale open in
$E^{(a)}$. In particular, we will work away from the diagonals. Also, we consider the pullback
of the corresponding schemes and vector bundles to $E^a$ (but we will keep the same notations
for the base change from $E^{(a)}$ to $E^a$). Let $D=w_1+\ldots+w_a$ with all the points distinct.
For every subset $\aleph\subset\{1,\ldots,a\}$ we set
\[D_\aleph:=\sum_{r\in\aleph}w_r,\ H_\aleph:=H^0\big(E,(\CK^{-1}(D_\aleph)/\CK^{-1})\oplus
(\CK(-D_\aleph)/\CK(-D))\big)\subset H_D.\]
To $H_\aleph$ we associate a section $\theta_\aleph$ of the determinant line bundle on the
Grassmannian $\Gr(a,H_D)$ vanishing precisely over the set of subspaces that are not transversal
to $H_\aleph$. Namely, for a subspace $S\subset H_D$, the value of $\theta_\aleph$ at $S$ is
the determinant of the composition of natural maps $S\to H_D\to H_D/H_\aleph$.
Thus $\theta_\aleph$ is a section of the line bundle with fibers $\det(H_D/H_\aleph)\otimes\det^{-1}(S)$.
Note that $\det(H_D)$ is canonically trivialized due to Serre duality between $H^0(D,\CK^{-1}(D)|_D)$
and $H^0(D,\CK|_D)$, so we can view $\theta_\aleph$ as a global section of
$\fL\otimes\varpi_{E^{(a)}}^*\det^{-1}(H_\aleph)$ on $\ol{Z}{}_\CK^a$.

Note that $H_\aleph$ and $H_\gimel$ are transversal iff $\gimel=\{1,\ldots,a\}\setminus\aleph$.
Thus $\theta_\aleph(H_\gimel)=0$ for $\gimel\ne\{1,\ldots,a\}\setminus\aleph$. On the other hand,
$\theta_\aleph(H_{\{1,\ldots,a\}\setminus\aleph})\in\fL|_{H_{\{1,\ldots,a\}\setminus\aleph}}\otimes\det^{-1}(H_\aleph)$
is the determinant of the isomorphism $H_{\{1,\ldots,a\}\setminus\aleph}\iso H_D/H_\aleph$.
Hence the composition
\[\bigoplus_{\aleph\subset\{1,\ldots,a\}}\det(H_\aleph)\xrightarrow{(\theta_\aleph)}
\varpi_{E^{(a)}*}\left(\fL|_{\ol{Z}{}_\CK^a}\right)\to
\bigoplus_{\aleph\subset\{1,\ldots,a\}}\det{}\!^{-1}(H_\aleph)\]
is an isomorphism that is a direct sum of the isomorphisms
\[\theta_\aleph(H_{\{1,\ldots,a\}\setminus\aleph})\colon\det(H_\aleph)\to\det{}\!^{-1}(H_\aleph).\]
It follows that the canonical embedding of $\ol{Z}{}_\CK^a$ into the projectivization of
$\left(\varpi_{E^{(a)}*}\big(\fL|_{\ol{Z}{}_\CK^a}\big)\right)^\vee\simeq
\bigoplus_{\aleph\subset\{1,\ldots,a\}}\det(H_\aleph)$ is the morphism
\[\ol{Z}{}_\CK^a\xrightarrow{(\theta_\aleph)}
\BP\left(\bigoplus_{\aleph\subset\{1,\ldots,a\}}\det{}\!^{-1}(H_\aleph)\right)
\xrightarrow{\big(\theta_{\{1,\ldots,a\}\setminus\aleph}^{-1}(H_\aleph)\big)}
\BP\left(\bigoplus_{\aleph\subset\{1,\ldots,a\}}\det(H_\aleph)\right)\]
(where we use the duality
$\theta_\aleph^*(H_{\{1,\ldots,a\}\setminus\aleph})=\theta_{\{1,\ldots,a\}\setminus\aleph}(H_\aleph)$).

\subsubsection{Explicit form of the identification of~Lemma~\ref{duo}}
\label{Res}
Now we are in a position to calculate the isomorphism of~Lemma~\ref{duo} over a point
$(w_1,\ldots,w_a)\in E^a$. This isomorphism takes form
\begin{equation}
  \label{chetyre}
  \bigoplus_{\aleph\subset\{1,\ldots,a\}}\det{}\!^{-1}(H_\aleph)\simeq
  \left(\bigotimes_{r=1}^a\CK^{-1}|_{w_r}\right)\otimes\bigoplus_{\aleph\subset\{1,\ldots,a\}}
  \left(\bigotimes_{r\in\aleph}\CK^2|_{w_r}\right).
\end{equation}
Note that $H_\aleph=\bigoplus_{r\in\aleph}\CK^{-1}(w_r)|_{w_r}\oplus\bigoplus_{r'\not\in\aleph}\CK|_{w_{r'}}$, so
\begin{multline}
  \label{this}
  \det{}\!^{-1}(H_\aleph)\simeq\left(\bigotimes_{r\in\aleph}\CK(-w_r)|_{w_r}\right)\otimes
\left(\bigotimes_{r'\not\in\aleph}\CK^{-1}|_{w_{r'}}\right)\\
\simeq\left(\bigotimes_{r=1}^a\CK^{-1}|_{w_r}\right)\otimes
\left(\bigotimes_{r\in\aleph}\CK^2(-w_r)|_{w_r}\right).
\end{multline}
One can check that the isomorphism~(\ref{chetyre}) is obtained from~(\ref{this}) by taking the
direct sum over $\aleph\subset\{1,\ldots,a\}$ and making use of the trivializations of
$\bomega_{w_r}\simeq\bomega_E(w_r)|_{w_r}\simeq\CO_E(w_r)|_{w_r}$.
Hence the dual isomorphism to~(\ref{chetyre}) is induced by the natural isomorphisms
\[\det(H_\aleph)\simeq\left(\bigotimes_{r\in\aleph}\CK^{-1}|_{w_r}\right)\otimes
\left(\bigotimes_{r'\not\in\aleph}\CK|_{w_{r'}}\right)\\
\simeq\left(\bigotimes_{r=1}^a\CK|_{w_r}\right)\otimes
\left(\bigotimes_{r\in\aleph}\CK^{-2}|_{w_r}\right).\]
Thus the image of a point $\varphi=(\CL,s,t)\in\oZ^a_\CK$ in
$\BP\left(\bigoplus_{\aleph\subset\{1,\ldots,a\}}\bigotimes_{r\in\aleph}\CK^{-2}|_{w_r}\right)$ is
obtained by first taking the point $\big(\theta_{\{1,\ldots,a\}\setminus\aleph}(\varphi)\big)\in
\BP\left(\bigoplus_{\aleph\subset\{1,\ldots,a\}}\det^{-1}(H_{\{1,\ldots,a\}\setminus\aleph})\right)$
and then applying the natural isomorphisms $\det^{-1}(H_{\{1,\ldots,a\}\setminus\aleph})\iso\det(H_\aleph)$
to each component.

Finally, let us calculate the values of $\theta_{\{1,\ldots,a\}\setminus\aleph}$ at a point
$\varphi=(\CL,s,t)\in\oZ^a_\CK$. By definition, the corresponding point of the Grassmannian
$\Gr(a,H_D)$ is the image of the map
\[H^0(D,\CK|_D)\xrightarrow{(t',1)}H^0(D,\CK^{-1}(D)|_D\oplus\CK|_D)=H_D.\]
Thus the value of $\theta_{\{1,\ldots,a\}\setminus\aleph}$ is given by the determinant of the composition
\[H^0(D,\CK|_D)\xrightarrow{(t',1)}H_D\to H_D/H_{\{1,\ldots,a\}\setminus\aleph}\simeq
\bigoplus_{r\in\aleph}\CK^{-1}|_{w_r}\oplus\bigoplus_{r'\not\in\aleph}\CK|_{w_{r'}},\]
that is equal to
\[\theta_{\{1,\ldots,a\}\setminus\aleph}(t')=\prod_{r\in\aleph}\Res_{w_r}(t')\in\bigotimes_{r\in\aleph}
\CK^{-2}|_{w_r}\simeq\det{}\!^{-1}H^0(D,\CK|_D)\otimes\det{}\!^{-1}(H_{\{1,\ldots,a\}\setminus\aleph}).\]
Therefore, the corresponding point in the projectivization of
$\bigoplus_{\aleph\subset\{1,\ldots,a\}}\bigotimes_{r\in\aleph}\CK^{-2}|_{w_r}$
(i.e.\ in $^{\on{Mir}}\oZ^a_\CK|_{E^{(a)}\setminus\Delta}$) is the point with the homogeneous coordinates
$\left(\prod_{r\in\aleph}\Res_{w_r}(t')\right)_{\aleph\subset\{1,\ldots,a\}}$. It is easy to see that this is
nothing but the image under Segre embedding of the point
\begin{equation}
  \label{1:Res}
  \big(1:\Res_{w_1}(t')\big),\ldots,\big(1:\Res_{w_a}(t')\big).
\end{equation}

\subsection{Coulomb zastava}
\label{coul zas}
In this section we modify the construction of~\S\ref{mir com}. In~Theorem~\ref{zas vs branch} below we will
show that the resulting zastava space is isomorphic to the elliptic Coulomb
branch of a quiver gauge theory (for the Dynkin quiver $Q$ of $G$) when all the line bundles $\CK_i$
are trivial.

We define the factorizable vector bundle
$\BU^\alpha_\CK$ on $E^\alpha$ as
\begin{equation}
  \label{coul bun}
  \BU_\CK^\alpha:=
\bigoplus_{\beta+\gamma=\alpha}\bq_*\left(\bp^*\CK^\beta
\otimes\CO_{E^\beta\times E^\gamma}\big(\sum_{h\in Q_1}\Delta_{\on{o}(h)\on{i}(h)}^{\beta,\gamma}\big)\right).
\end{equation}
It contains two codimension 1 subbundles: $\BU_{\CK,\on{low}}^\alpha$ and $\BU_\CK^{\alpha,\on{up}}$, where in
the above direct sum we omit summands corresponding to $\beta=0$ (resp.\ $\gamma=0$).

As in~\S\ref{mir com}, $\BP\BU_\CK^\alpha|_{E^\alpha_{\on{disj}}}$ contains
$\big(\prod_{i\in I}\BP(\CK_i\oplus\CO_E)^{a_i}|_{(\prod_{i\in I}E^{a_i})_{\on{disj}}}\big)/S_\alpha$.

\begin{defn} 
  \label{coul def}
  \textup{(a)} {\em Coulomb compactified zastava} $^C\ol{Z}{}^\alpha_\CK$ is defined as the
    closure of
    $\big(\prod_{i\in I}\BP(\CK_i\oplus\CO_E)^{a_i}|_{(\prod_{i\in I}E^{a_i})_{\on{disj}}}\big)/S_\alpha$ in
    $\BP\BU_\CK^\alpha$ (with the reduced closed subscheme structure).

    \textup{(b)} The {\em upper} (resp.\ {\em lower}) {\em boundary}
    $\partial_{\on{up}}{}^C\ol{Z}{}^\alpha_\CK$ (resp.\
    $\partial_{\on{low}}{}^C\ol{Z}{}^\alpha_\CK$) is defined as the intersection
    $^C\ol{Z}{}^\alpha_\CK\cap\BP\BU_\CK^{\alpha,\on{up}}$ (resp.\
    $^C\ol{Z}{}^\alpha_\CK\cap\BP\BU_{\CK,\on{low}}^\alpha$).

    \textup{(c)} {\em Coulomb zastava} $^C\!Z^\alpha_\CK$ is defined as the open
    subscheme in $^C\ol{Z}{}^\alpha_\CK$ obtained by removing the upper boundary
    $\partial_{\on{up}}{}^C\ol{Z}{}^\alpha_\CK$.

    \textup{(d)} {\em Coulomb open zastava} $^C\!\oZ^\alpha_\CK$ is defined as the open
    subscheme in $^C\!Z^\alpha_\CK$ obtained by further removing the lower boundary
    $\partial_{\on{low}}{}^C\ol{Z}{}^\alpha_\CK$.
\end{defn}

\subsection{Example of type $A_1$ \`a la Coulomb}
\label{A1 cou}
Inside the
symmetrized version $\Gr_{\GL(2),E^{(a)}}$ of Beilinson-Drinfeld Grassmannian of $G=\GL(2)$ of degree
$a$, we consider the moduli space $M^a$ of locally free rank 2 subsheaves
$\CW\subset\CO_Ev_1\oplus\CO_Ev_2$ such that $\on{length}\!\big((\CO_Ev_1\oplus\CO_Ev_2)/\CW\big)=a$.
We consider the determinant (relatively very ample) line bundle
$\fL$ on $\Gr_{\GL(2),E^{(a)}}$ and its restriction to $M^a$. The projection $M^a\to E^{(a)}$ is
denoted by $\pi^a$. We have
\[(\pi^a_*\fL)^\vee=\bfomega^{-1}_{E^{(a)}}\otimes\bigoplus_{b+c=a}\bq_*\CO_{E^{(b)}\times E^{(c)}}\]
(notation of~\S\ref{mir com}). Indeed, let $T\subset\GL(2)$ be the diagonal Cartan torus in
the basis $v_1,v_2$ of $\BC^2$, and let
$(M^a)^T$ be the fixed point subscheme of $M^a$. Then $(M^a)^T=\bigsqcup_{b+c=a}E^{(b)}\times E^{(c)}$:
to $D_b\in E^{(b)},\ D_c\in E^{(c)}$ we associate
\[\CW_{D_b,D_c}:=\CO_E(-D_b)v_1\oplus\CO_E(-D_c)v_2\subset\CO_Ev_1\oplus\CO_Ev_2.\]
The restriction to fixed points induces an isomorphism $\pi^a_*\fL\iso\pi^a_*(\fL|_{(M^a)^T})$.
The fiber $\fL_\CW$ is $\det^{-1}\!R\Gamma(E,\CW)$ by definition, so that
the fiber $\fL_{\CW_{D_b,D_c}}=\det(\CO_E/\CO_E(-D_b))\otimes\det(\CO_E/\CO_E(-D_c))$. The latter line is
canonically isomorphic to the fiber of $\bfomega_{E^{(b)}}\boxtimes\bfomega_{E^{(c)}}$ at
$(D_b,D_c)\in E^{(b)}\times E^{(c)}$. We conclude that
$\pi^a_*\fL=\bigoplus_{b+c=a}\bq_*(\bfomega_{E^{(b)}\times E^{(c)}})$. Furthermore, the dual vector
bundle of $\bq_*(\bfomega_{E^{(b)}\times E^{(c)}})$ is
$\bq_*\big(\bfomega^{-1}_{E^{(b)}\times E^{(c)}}(\Delta^{b,c})\big)$ by the relative Grothendieck-Serre
duality for $\bq$ since $\Delta^{b,c}$ is the ramification divisor of $\bq$. Finally,
$\bfomega^{-1}_{E^{(b)}\times E^{(c)}}(\Delta^{b,c})=\bq^*\bfomega^{-1}_{E^{(a)}}$, and we are done by the
projection formula.

\medskip

Generalizing the above example, for a line bundle $\CK$ on $E$ of degree 0, we consider
the moduli space $M^a_\CK$ 
of locally free rank 2 subsheaves $\CW\subset\CK\oplus\CK^{-1}$ such that
$\on{length}\!\big((\CK\oplus\CK^{-1})/\CW\big)=a$. The same argument as above provides an
isomorphism $^C\ol{Z}{}^a_\CK\simeq M^a_\CK$. Here the Dynkin graph consists of the unique
  vertex $i$, and in the definition of $^C\ol{Z}{}^a_\CK$ we set $\CK_i=\CK^{-\alpha^\svee}=\CK^{-2}$.
    Furthermore, let $\oM^a_\CK\subset M^a_\CK$
be the open subspace formed by all $\CW\subset\CK\oplus\CK^{-1}$ transversal to both $\CK$ and
$\CK^{-1}$. Then the isomorphism $^C\ol{Z}{}^a_\CK\simeq M^a_\CK$ restricts to an isomorphism
$^C\!\oZ^a_\CK\simeq\oM^a_\CK$. Finally, the argument in the proof of~Proposition~\ref{hilb vs zas}(b)
establishes an isomorphism
\begin{equation}
  \label{hilb vs coul}
  ^C\!\oZ^a_\CK\simeq\oM^a_\CK\simeq\Hilb^a_\tr(\oS_{\CK^{-2}}).
\end{equation}

\subsection{Example of type $A_2$ \`a la Coulomb}
\label{A2 cou}
In this section $I$ consists of two vertices $i,j$ connected by a single arrow $i\to j$, and
$\alpha=\alpha_i+\alpha_j$. Then\[\BU^\alpha_\CK=\CO_E\boxtimes\CO_E\oplus
(\CK_i\boxtimes\CO_E)(\Delta_{ij})\oplus\CO_E\boxtimes\CK_j\oplus\CK_i\boxtimes\CK_j,\]
a 4-dimensional vector bundle on $E\times E$. The Coulomb compactified zastava
$^C\ol{Z}{}^\alpha_\CK\subset\BP\BU^\alpha_\CK$ is the zero locus of the
section $s$ of \[\Sym{}\!^2(\BU^\alpha_\CK)^\vee\otimes\big((\CK_i\boxtimes\CK_j)(\Delta_{ij})\big)\]
defined as follows. First, we set \['\BU^\alpha_\CK:=\CO_E\boxtimes\CO_E\oplus
\CK_i\boxtimes\CO_E\oplus\CO_E\boxtimes\CK_j\oplus\CK_i\boxtimes\CK_j=
(\CO_E\oplus\CK_i)\boxtimes(\CO_E\oplus\CK_j).\]
Then $\Sym{}\!^2({}'\BU^\alpha_\CK)^\vee\otimes(\CK_i\boxtimes\CK_j)$ has a canonical section
$\sigma$ defined as follows. Let $w_i,w_j$ be local nonvanishing sections of $\CO_E$, and let
$u_i,u_j$ be local nonvanishing sections of $\CK_i^{-1},\CK_j^{-1}$. Then
\[\sigma=\big((w_i\boxtimes w_j)\cdot(u_i\boxtimes u_j)
-(w_i\boxtimes u_j)\cdot(u_i\boxtimes w_j)\big)\otimes(w_i^{-1}u_i^{-1}\boxtimes w_j^{-1}u_j^{-1}).\]
We have a tautological embedding \[\Sym{}\!^2({}'\BU^\alpha_\CK)^\vee\otimes(\CK_i\boxtimes\CK_j)
\hookrightarrow\Sym{}\!^2(\BU^\alpha_\CK)^\vee\otimes\big((\CK_i\boxtimes\CK_j)(\Delta_{ij})\big)\]
(arising from $\CO_{E\times E}\hookrightarrow\CO_{E\times E}(\Delta_{ij})$), and
$s$ is defined as the image of $\sigma$ under this embedding.

Thus the family $^C\ol{Z}{}^\alpha_\CK\subset\BP\BU^\alpha_\CK\to E\times E$ has fibers
$\BP^1\times\BP^1\subset\BP^3$ (smooth quadrics) away from the diagonal
$\Delta_{ij}\subset E\times E$ that degenerate to $\BP^2\cup_{\BP^1}\BP^2\subset\BP^3$
(singular reducible quadrics) over the diagonal $\Delta_{ij}$.

We choose an analytic neighbourhood $W$ of a point $e\in E$ with coordinate $w$,
and trivialize the line bundles \[\big((\CK_i^{-1}\boxtimes\CO_W)(-\Delta_{ij})\big)|_{W\times W},\
\big(\CO_W\boxtimes\CK_j^{-1}\big)|_{W\times W},\ \big(\CK_i^{-1}\otimes\CK_j^{-1}\big)|_{W\times W}\]
compatibly. We denote the coordinates along fibers of these trivialized line bundles by
$y_i,y_j,y_{ij}$ respectively.
Then $^C\!\oZ^\alpha_\CK|_{W\times W}\subset W\times W\times\BA^3$ is cut
out by a single equation $y_iy_j-y_{ij}(w_1-w_2)=0$ and an open condition $y_{ij}\ne0$.

\section{Elliptic Coulomb branch of a quiver gauge theory}
\label{piat}
In this section we discuss the elliptic analogue of the construction~\cite{bfn1} of the Coulomb
branch of a gauge theory. This construction made use of equivariant Borel-Moore homology of a certain
variety of triples, and we replace the Borel-Moore homology with its elliptic version.
The results of this section are not used in the rest of the paper, and serve as a motivation only.
We consider a quiver $Q=(Q_0,Q_1)$ with the set of vertices $Q_0$ and the set of arrows $Q_1$.
We use the following notation for the Laurent series field and the Taylor series ring:
$\bF=\BC\dprts{t}\supset\BC\dbkts{t}=\bO$.

\subsection{Basics}
Let $V=\oplus_{i\in Q_0}V_i,\ W=\oplus_{i\in Q_0}W_i$ be finite dimensional $Q_0$-graded complex
vector spaces. The group $\sG=\GL(V)=\prod_{i\in Q_0}\GL(V_i)$ acts naturally on
$\bN=\bigoplus_{i\in Q_0}\Hom(W_i,V_i)\oplus\bigoplus_{(i\to j)\in Q_1}\Hom(V_i,V_j)$.
The construction of~\cite[\S2(i)]{bfn1} associates to this representation of $\sG$ the {\em variety
  of triples} $\CR$ contained in an infinite rank vector bundle $\CT$ over $\Gr_\sG$.
We consider the equivariant elliptic Borel-Moore homology ring $H^{\sG_\bO}_\elli(\CR)$.

A few words about the latter notion are in order. A theory of $\sG$-equivariant elliptic
{\em cohomology} with values in quasicoherent sheaves of algebras over the moduli space of
semistable
$\sG$-bundles over $E$ was proposed in~\cite{gro,gkv}. After the proposal of~\cite{gro,gkv},
quite a few foundational papers appeared establishing the basic properties of equivariant
elliptic cohomology. We will use~\cite{gan} as a reference. For one thing, we restrict ourselves
to a product of general linear groups $\sG$ since the centralizers of commuting pairs in
$\sG$ are connected, and the base change in equivariant elliptic cohomology holds
true~\cite[Theorem 4.6, Corollary 4.10]{gan}.

Now the equivariant elliptic {\em Borel-Moore homology} $H^{\sG_\bO}_\elli(X)$ is defined
as $W$-invariants in the Cartan torus equivariant elliptic Borel-Moore homology, and these
in turn are defined by descent from the usual equivariant Borel-Moore homology or the equivariant
homological $K$-theory as in~\cite[\S3.3]{gan}. The details of the construction are to appear
in a forthcoming work of I.~Perunov and A.~Prikhodko.

We set $a_i=\dim V_i$, so that
$\alpha=\sum_{i\in Q_0}a_i\alpha_i\in\Lambda_{\on{pos}}$ is a positive coroot combination of the Kac-Moody
Lie algebra $\fg$ with Dynkin diagram $Q$.
Then the equivariant elliptic cohomology $H^\elli_{\sG_\bO}(\pt)=\CO_{E^\alpha}$,
where $E^\alpha=\prod_{i\in I}E^{(a_i)}$. The equivariant
elliptic Borel-Moore homology $H_\elli^{\sG_\bO}(\CR)$ is a quasicoherent sheaf of commutative
$\CO_{E^\alpha}$-algebras by construction of~\cite[\S3]{bfn1}. 
Its relative spectrum is denoted $\CM_C^\elli=\CM_C^\elli(\sG,\bN)$:
the elliptic Coulomb branch. By construction, $\CM_C^\elli$ is equipped with an affine morphism
$\varPi\colon\CM_C^\elli\to E^\alpha$.

\subsection{Compactified elliptic Coulomb branch}
From now on we assume that $Q$ is an oriented Dynkin diagram of an almost simple simply connected
simply laced complex algebraic group $G$. We also assume that $W=0$. We will denote $Q_0$ by $I$
to match the notation of Sections~\ref{dva},\ref{tri}.

As in~\cite[\S3(ii)]{bfn2}, we consider the subalgebra
$H_\elli^{\sG_\bO}(\CR^+)\subset H_\elli^{\sG_\bO}(\CR)$ (homology supported over the positive part
of the affine Grassmannian $\Gr^+_\sG\subset\Gr_\sG$), and its relative spectrum
$\CM_C^{\elli,+}\stackrel{\varPi}{\longrightarrow}E^\alpha$.
By construction, we have an open embedding $\CM_C\subset\CM_C^{\elli,+}$ of varieties over $E^\alpha$.

As in~\cite[Remark 3.7]{bfn2}, we define a certain {\em support multifiltration}
$F_\bullet H_\elli^{\sG_\bO}(\CR^+)$ numbered by the monoid $\Lambda^\vee_{\on{pos}}$ of nonnegative
integral combinations of positive roots of $G$.
The (multi)projective spectrum of its Rees algebra is
denoted $\ol\CM{}_C^\elli$: the {\em compactified elliptic Coulomb branch}. By construction,
it is equipped with a projective morphism $\varPi\colon\ol\CM{}_C^\elli\to E^\alpha$.
Also we have an open embedding $\CM_C^{\elli,+}\subset\ol\CM{}_C^\elli$ of varieties over $E^\alpha$.

By definition, \[F_{\sum_{i\in I}\alpha^\svee_i}H_\elli^{\sG_\bO}(\CR^+)=
\bigoplus_{\Lambda_{\on{pos}}\ni\beta=\sum b_i\alpha_i\leq\alpha}
H_\elli^{\sG_\bO}\big(\CR^+_{\sum_{i\in I}\varpi_{i,b_i}}\big)\]
(elliptic homology of the preimage in $\CR^+$ of all the fundamental $\sG_\bO$-orbits in $\Gr^+_\sG$;
here $\varpi_{i,n}$ stands for the $n$-th fundamental coweight of $\GL(V_i)$; in particular,
$\varpi_{i,0}=0$ and $\varpi_{i,a_i}=(1,\ldots,1)$). All the fundamental $\sG_\bO$-orbits in
$\Gr^+_\sG$ are closed; more precisely, $\Gr_{\GL(V_i)}^{\varpi_{i,n}}\cong\Gr(n,a_i)$ (the Grassmannian
of $n$-dimensional subspaces in $V_i$). We have
\[H^\elli_{\GL(V_i,\bO)}\big(\Gr(b_i,a_i)\big)=\bq_*(\CO_{E^{(b_i)}\times E^{(a_i-b_i)}})\] (the sheaf of elliptic
   {\em cohomology} on $E^{(a_i)}$, notation of~\S\ref{mir com}), and dually,
\[H_\elli^{\GL(V_i,\bO)}\big(\Gr(b_i,a_i)\big)=\big(\bq_*(\CO_{E^{(b_i)}\times E^{(a_i-b_i)}})\big)^\vee\]
(elliptic {\em homology}). It follows that for $\beta\leq\alpha$ and
$\gamma:=\alpha-\beta$ we have \[H_\elli^{\sG_\bO}(\CR^+_{\sum_{i\in I}\varpi_{i,b_i}})=
\left(\bq_*\Big(\CO_{E^\beta\times E^\gamma}\big(\sum_{h\in Q_1}\Delta_{\on{o}(h)\on{i}(h)}^{\beta,\gamma}\big)\Big)\right)^\vee\]
(notation of~\S\ref{mir com}; note that the divisor $\Delta_{ij}^{\beta,\gamma}$ in
$E^\beta\times E^\gamma$ is the pullback of the corresponding divisor in $E^\alpha$, so that the
twisting and pushforward commute by the projection formula). The twisting arises from the
elliptic analogue of~\cite[Theorem 4.1]{bfn1} and localization in elliptic homology, reducing
the calculation to the toric case.

All in all, we obtain a canonical isomorphism
$F_{\sum_{i\in I}\alpha^\svee_i}H_\elli^{\sG_\bO}(\CR^+)=(\BU^\alpha)^\vee$ (notation of~\S\ref{coul zas}, where
we set $\BU^\alpha:=\BU_\CK^\alpha$ for trivial line bundles $\CK_i=\CO_E$).
It induces a morphism $\varTheta\colon\ol\CM{}^\elli_C\to\BP\BU^\alpha$.

\begin{thm}
  \label{zas vs branch}
  \textup{(a)} $\varTheta$ is a closed embedding, and its image is $^C\ol{Z}{}^\alpha$
  (where we set $^C\ol{Z}{}^\alpha:={}^C\ol{Z}{}^\alpha_\CK$ for trivial line bundles $\CK_i=\CO_E$).

  \textup{(b)} The isomorphism $\varTheta\colon\ol\CM{}^\elli_C\iso{}^C\ol{Z}{}^\alpha$
  restricts to the same named isomorphism of the open subvarieties
  $\CM^{\elli,+}_C\iso{}^C\!Z^\alpha$.

  \textup{(c)} The isomorphism $\varTheta\colon\ol\CM{}^\elli_C\iso{}^C\ol{Z}{}^\alpha$
  restricts to the same named isomorphism of the open subvarieties
  $\CM^\elli_C\iso{}^C\!\oZ^\alpha$.
\end{thm}

\begin{proof}
  We consider the usual equivariant Borel-Moore homology ring $H^{\sG_\bO}_*(\CR^+)$. The argument in the
  proof of~\cite[Proposition 6.8]{bfn1} demonstrates that this ring is generated by
  $\bigoplus_{\Lambda_{\on{pos}}\ni\beta=\sum b_i\alpha_i\leq\alpha}
  H^{\sG_\bO}_*\big(\CR^+_{\sum_{i\in I}\varpi_{i,b_i}}\big)$.
  It follows that the corresponding Rees algebra is generated by
  $F_{\sum_{i\in I}\alpha^\svee_i}H^{\sG_\bO}_*(\CR^+)$. Since the elliptic cohomology coincides with the usual
  cohomology locally in the analytic topology of $E^\alpha$, it follows that the Rees algebra of
  $H^{\sG_\bO}_\elli(\CR^+)$ is generated by $F_{\sum_{i\in I}\alpha^\svee_i}H^{\sG_\bO}_\elli(\CR^+)$.
  Hence $\varTheta$ is a closed embedding. The image of $\varTheta$ over the complement to diagonals
  in $E^\alpha$ is readily identified with
  $\big(\prod_{i\in I}(E\times\BP^1)^{a_i}|_{(\prod_{i\in I}E^{a_i})_{\on{disj}}}\big)/S_\alpha$. We conclude that
  the image of the closed embedding $\varTheta$ coincides with $^C\ol{Z}{}^\alpha$.
  This completes the proof of (a), and (b,c) follow immediately.
  \end{proof}

\section{Reduced elliptic zastava}
\label{chetyr}

\subsection{Poisson structure}
\label{poisson}
According to~\S\ref{coul zas}, $^C\ol{Z}{}^\alpha_\CK$ contains an open smooth subvariety
$U_\alpha:=\big(\prod_{i\in I}\BP(\CK_i\oplus\CO_E)^{a_i}|_{(\prod_{i\in I}E^{a_i})_{\on{disj}}}\big)/S_\alpha$.
It has a covering
$\widetilde{U}_\alpha:=\prod_{i\in I}\BP(\CK_i\oplus\CO_E)^{a_i}|_{(\prod_{i\in I}E^{a_i})_{\on{disj}}}$,
an open subvariety of the product of the ruled surfaces
$\ol{U}\!_\alpha:=\prod_{i\in I}\BP(\CK_i\oplus\CO_E)^{a_i}$. Each ruled surface
$\BP(\CK_i\oplus\CO_E)$ contains an open subvariety $\oS_{\CK_i}$ (notation of~\S\ref{sl2 vs hilb}).
The canonical class of $\oS_{\CK_i}$ is trivial, and the trivialization is defined uniquely
by our choice of trivialization of the canonical bundle $\bomega_E$, see~Remark~\ref{triv can}.
In other words, $\oS_{\CK_i}$ carries a canonical symplectic form $\omega_{\CK_i}$.
More explicitly, we can trivialize $\CK_i$ \'etale locally and choose a function $w$ on $E$
such that $dw$ is the trivialization of $\bomega_E$ (Remark~\ref{triv can}).
Let $(w,y)$ be the corresponding
\'etale local coordinates on $\oS_{\CK_i}$ such that $y$ is invertible. We define the Poisson
bracket setting $\{y,x\}_{\CK_i}=y$. For this bracket we have $\{f(w)y,w\}_{\CK_i}=f(w)y$.
It follows that the brackets on the intersections of coordinate patches are all compatible,
so they give rise to a global bracket arising from a symplectic form $\Omega_{\CK_i}$.

Note that $\Omega_{\CK_i}$ is invariant with respect to the action of $\BC^\times$ by
fiberwise dilations. Note also that the symplectic structure on $\oS_{\CK_i}$ extends as a
Poisson structure to $\BP(\CK_i\oplus\CO_E)$ (vanishing along the zero and infinite sections).
Finally, the product Poisson structure on $\ol{U}\!_\alpha$ is clearly $S_\alpha$-invariant,
so by descent we obtain a Poisson structure on $U_\alpha$, to be denoted $\{,\}^\alpha_\CK$.

It is likely that the Poisson structure $\{,\}^\alpha_\CK$ on $U_\alpha$ extends as a Poisson
structure to the Coulomb compactified zastava $^C\ol{Z}{}^\alpha_\CK$. However, the proof would
require the normality property of $^C\ol{Z}{}^\alpha_\CK$ that we do not know at the moment.
Instead we restrict to an open subset $U^\circ_\alpha\subset U_\alpha$ removing the~0 and $\infty$
sections of the surface $\BP(\CK_i\oplus\CO_E)$.

\begin{prop}
  \label{pois}
  The Poisson structure $\{,\}^\alpha_\CK$ on $U^\circ_\alpha$ extends to a Poisson
  structure $\{,\}_\CK$ on Coulomb open zastava $^C\!\oZ^\alpha_\CK\subset{}^C\ol{Z}{}^\alpha_\CK$.
  Moreover the latter Poisson structure is symplectic.
  
\end{prop}

\begin{proof}
  The construction of Coulomb zastava being local, we can restrict our consideration to
  $^C\ol{Z}{}^\alpha_\CK|_{W^\alpha}$ where $W$ is an analytic open subset of $E$ with a global
  coordinate $w$ whose differential $dw$ coincides with the trivialization of $\bomega_E$
  (Remark~\ref{triv can});
  thus we fix an open analytic embedding $W\hookrightarrow\BA^1$.
  We can also trivialize all the line bundles $\CK_i|_W$.
  Combining~Theorem~\ref{zas vs branch} with~\cite[Theorem 3.1]{bfn2} we obtain
  an isomorphism between $^C\!\oZ^\alpha_\CK|_{W^\alpha}$ and $\oZ^\alpha|_{W^\alpha}$. Here
  $\oZ^\alpha\to\BA^\alpha$ is the usual open zastava studied in~\cite{bfn2}.
  In particular, the smoothness of $\oZ^\alpha$ implies the smoothness of $^C\!\oZ^\alpha_\CK$.

  In order to check that the rational Poisson structure $\{,\}^\alpha_\CK$ is symplectic on the
  Coulomb open zastava, it suffices to do this over the generic points of diagonals in $E^\alpha$
  (equivalently, over the generic points of diagonals in $W^\alpha$). The factorization isomorphism
  \[^C\!\oZ^\alpha_\CK|_{(E^\beta\times E^\gamma)_{\on{disj}}}
  \simeq({}^C\!\oZ^\beta_\CK\times{}^C\!\oZ^\gamma_\CK)|_{(E^\beta\times E^\gamma)_{\on{disj}}}\]
  is Poisson by construction. Hence it suffices to check the symplectic property of the
  Poisson structure over the generic points of diagonals in $E^\beta$ (equivalently,
  over the generic points of diagonals in $W^\beta$) for $|\beta|=2$.

  There are 3 cases to consider.
  If $\beta=\alpha_i+\alpha_j$, and $i,j$ are not connected by an arrow, there is nothing to
  check. If $\beta=\alpha_i+\alpha_j$, and $i,j$ are connected by an arrow $i\to j$, then
  the Coulomb open zastava over $W^\beta$ with its Poisson structure is nothing but the
  restriction of the rational open zastava $\oZ^\beta$ (for the group $\SL(3)$) with its
  Poisson structure to $W^\beta$. The latter one is symplectic e.g.\ by~\cite{fkmm}.
  More precisely, comparing (the last line of)~\S\ref{A2 cou} with e.g.~\cite[Remark 2.2]{bfn2} we
  get an explicit identification between the Coulomb open zastava $^C\!\oZ^\beta_\CK|_{W^\alpha}$ and
  the rational open zastava $\oZ^\beta|_{W^\alpha}$ sending
  $\{,\}^\beta_\CK$ to the standard Poisson structure on $\oZ^\beta|_{W^\alpha}$.
  If $\beta=2\alpha_i$, the identification of~\S\ref{A1 cou} and~\S\ref{sl2 vs hilb}
  between $^C\!\oZ^\beta_\CK$ and the corresponding Hilbert scheme sends $\{,\}^\beta_\CK$
  to the standard Poisson (symplectic) structure on the Hilbert scheme.

  This completes the proof of the proposition. 
\end{proof}

\subsection{Hamiltonian reduction}
\label{ham red}
We assume that $a_i>0$ for any $i\in I$. Let $T^{\on{ad}}$ act on $\CK_i$ via the homomorphism
$\alpha^\svee_i\colon T^{\on{ad}}\to\BC^\times$ and the fiberwise dilation action of $\BC^\times$
on $\CK_i$. Clearly, this action extends to a fiberwise action on $\BP(\CK_i\oplus\CO_E)$.
Furthermore, for any decomposition $\alpha=\beta+\gamma$ (where $\beta=\sum_{i\in I}b_i\alpha_i$),
the fiberwise action of $T^{\on{ad}}$ on $\CK_i$ induces its action on $\CK^\beta$ and hence on
the vector bundle $\bq_*\left(\bp^*\CK^\beta
\otimes\CO_{E^\beta\times E^\gamma}\big(\sum_{h\in Q_1}\Delta_{\on{o}(h)\on{i}(h)}^{\beta,\gamma}\big)\right)$.
Clearly, the resulting actions of $T^{\on{ad}}$ on
$U^\alpha$ (see~\S\ref{poisson}) and on $\BP\BU^\alpha_\CK|_{E^\alpha_{\on{disj}}}$ are compatible.
This way $^C\!\oZ^\alpha_\CK$ acquires an effective hamiltonian action of $T^{\on{ad}}$.

We have the Abel-Jacobi morphisms $E^{(a_i)}\to\on{Pic}^{a_i}E$ and their product
$\on{AJ}\colon E^\alpha\to\prod_{i\in I}\on{Pic}^{a_i}E$. We denote the composed morphism by
\[\on{AJ}_Z\colon{}^C\!\oZ^\alpha_\CK\to E^\alpha\to\prod_{i\in I}\on{Pic}^{a_i}E.\]
Given a collection $\CalD=(\CalD_i)_{i\in I}\in\on{Pic}^{a_i}E$, we define the
{\em reduced Coulomb open zastava}
$^C_\CalD\uoZ^\alpha_\CK$ as $\on{AJ}_Z^{-1}(\CalD)/T$ (stack quotient,
cf.~Definition~\ref{red zas}). It inherits a Poisson structure from $^C\!\oZ^\alpha_\CK$,
symplectic on the smooth locus of $^C_\CalD\uoZ^\alpha_\CK$.

\begin{thm}
\label{reductio}  
  For $\CalD=(\CalD_i)_{i\in I}\in\on{Pic}^{a_i}E$, the reduced open zastava
  $^{\vphantom{C}}_\CalD\uoZ^\alpha_\CK$ is naturally isomorphic to the reduced Coulomb open zastava
  $^C_\CalD\uoZ^\alpha_{\CK'}$, where
  $\CK'_i:=\CK^{-\alpha^\svee_i}\otimes\CalD_i\otimes\bigotimes_{h\in Q_1: i=\on{o}(h)}\CalD^{-1}_{\on{i}(h)}$.
  \end{thm}

The proof will be given in~\S\ref{the proof} after some preparation.
Throughout the proof we will make use of the identification
$^{\on{Mir}}\!\oZ^\alpha_\CK\cong\oZ^\alpha_\CK$ of~Theorem~\ref{mir thm}. Thus we will compare two types
of reduced zastava constructed from the Dynkin quiver $Q$ (as opposed to the group $G$).
Roughly speaking the idea of the proof is as follows. Before the reduction, both types of zastava spaces are closures of the images of certain Segre embeddings into
projective bundles over the configuration space. The key idea is to check that after restricting to the Abel-Jacobi fibers the two projective bundles become isomorphic up to a twist
and the Segre images correspond to each other. This identification is based on certain calculations with line bundles over the Abel-Jacobi fibers performed in Lemmas~\ref{(4)} and~\ref{(5)} below.

Note that $\on{AJ}^{-1}(\CalD)$ is isomorphic to the product of projective spaces
$\prod_{i\in I}\BP^{a_i-1}$. Hence for a sequence of integers $\nu=(n_i)_{i\in I}$ we have a
line bundle $\CO(\nu)=\boxtimes_{i\in I}\CO_{\BP^{a_i-1}}(n_i)$ on $\on{AJ}^{-1}(\CalD)$.

\begin{prop}
  \label{bundles}
  For any $\beta=\sum_{i\in I}b_i\alpha_i\leq\alpha$ we set $b'_i:=b_i-\sum_{j\to i}b_j$, and
  $\beta':=\sum_{i\in I}b'_i\alpha_i$.
  Then we have an isomorphism
 \begin{multline*}\bq_*\left(\bp^*\Big(\CK^\beta\big(\sum_{i\in I}\Delta_{ii}^\beta
   -\sum_{h\in Q_1}\Delta_{\on{o}(h)\on{i}(h)}^\beta\big)\Big)
   \big(\sum_{i\in I}\Delta^{\beta,\gamma}_{ii}\big)\right)\Big|_{\on{AJ}^{-1}(\CalD)}\\
\simeq \bq_*\left(\bp^*\CK^{\prime\beta}\otimes\CO_{E^\beta\times E^\gamma}
\big(\sum_{h\in Q_1}\Delta_{\on{o}(h)\on{i}(h)}^{\beta,\gamma}\big)\right)\Big|_{\on{AJ}^{-1}(\CalD)}\otimes\CO(\beta').
  \end{multline*}
\end{prop}

The proposition follows from the projection formula and Lemmas~\ref{(4)} and~\ref{(5)} below.
We denote by $X^{\beta,\gamma}$ the preimage
$\bq^{-1}(\on{AJ}^{-1}(\CalD))$. Its projection to $E^\beta$ (resp.\ to $\on{AJ}^{-1}(\CalD)$)
will be denoted by $\bp$ (resp.\ by $\bq$). We will also need some partial desymmetrizations
of $X^{\beta,\gamma}$. Namely, we have $E^\beta=E^{|\beta|}/S_\beta$, and we will identify
$E^{|\beta|}$ with $\prod_{i\in I}\prod_{r=1}^{b_i}E_{i,r}$, where $E_{i,r}$ is a copy of $E$.
We denote by $X^{|\beta|,\gamma}\stackrel{\rho}{\longrightarrow}X^{\beta,\gamma}$ the cartesian product
$X^{\beta,\gamma}\times_{E^\beta\times E^\gamma}(E^{|\beta|}\times E^\gamma)$. For any $i\in I,\ r\leq b_i$,
the composite morphism $X^{|\beta|,\gamma}\to E^{|\beta|}\times E^\gamma\to E_{i,r}\times E^{\alpha-\alpha_i}$
factors through
$\rho_{i,r}\colon X^{|\beta|,\gamma}\to X^{\alpha_i,\alpha-\alpha_i}\subset E_{i,r}\times E^{\alpha-\alpha_i}$.
Finally, recall the line bundle $\CalD^\beta:=\boxtimes_{i\in I}\CalD_i^{(b_i)}$ on
$E^\beta=\prod_{i\in I}E^{(b_i)}$. Here $\CalD_i^{(b_i)}$ is the descent of $\CalD_i^{\boxtimes b_i}$
from $E^{b_i}$ to $E^{(b_i)}$.

\begin{lem}
  \label{(4)}
  \textup{(a)} We have an isomorphism of line bundles on $X^{\beta,\gamma}$:
  \[\phi_{\beta,\gamma}\colon\bp^*(\CalD^\beta)\otimes\bq^*\CO(\beta)\iso
  \bp^*\left(\CO_{E^\beta}\big(\sum_{i\in I}\Delta_{ii}^\beta\big)\right)\otimes
  \CO_{X^{\beta,\gamma}}\big(\sum_{i\in I}\Delta_{ii}^{\beta,\gamma}\big).\]
  \textup{(b)} We can choose a collection of isomorphisms $\phi_{\beta,\gamma}$ in \textup{(a)}
  satisfying the following factorization property:
  \[\rho^*\phi_{\beta,\gamma}=
  \bigotimes_{i\in I}^{1\leq r\leq b_i}\rho_{i,r}^*\phi_{\alpha_i,\alpha-\alpha_i}\]
  away from the preimage of all the diagonals in $E^\alpha$.
\end{lem}

\begin{proof}
  (a) It suffices to construct the desired isomorphism when $I$ consists of a single element.
  So we will write $E^{(b)},E^{(c)},E^{(a)}$ in place of $E^\beta,E^\gamma,E^\alpha$.
  We denote by $E\stackrel{p_E}{\longleftarrow} E\times X\stackrel{p_X}{\longrightarrow}X:=X^{(b),(c)}$
  the projections. We consider the projections of the universal divisors 
  $E\stackrel{\varpi_E}{\longleftarrow} \fD_b\stackrel{\varpi_{E^{(b)}}}{\longrightarrow}E^{(b)}$ and
  $E\stackrel{\varpi_E}{\longleftarrow} \fD_c\stackrel{\varpi_{E^{(c)}}}{\longrightarrow}E^{(c)}$.
  We keep the notations $\fD_b\subset E\times X\supset \fD_c$ for the pullbacks to $X$ of
  the universal divisors over $E^{(b)}$ and $E^{(c)}$. We fix a point $e\in E$.
  It defines divisors $Y_b\subset E^{(b)},\ Y_c\subset E^{(c)},\ Y_a\subset E^{(a)}$ formed by
  all the configurations of points on $E$ meeting $e$.
  
  We have an isomorphism of line bundles on $E\times X$:
  \begin{equation}
    \label{(1)}
    \CO_{E\times X}(\fD_b+\fD_c)\simeq p_E^*\CalD\otimes p_X^*\bq^*\CO(1).
  \end{equation}
  More precisely, we have a canonical isomorphism
  \begin{equation}
    \label{new 2}
    \tau_b\colon \CO_{E\times X}(\fD_b+\fD_c)\iso p_E^*\CalD\otimes p_X^*\bq^*\CO_{E^{(a)}}(Y_a).
  \end{equation}
  Indeed, for any $x\in X$ the restrictions of both sides to $E\times\{x\}$ are isomorphic.
  Thus, there exists a line bundle $\CL_X$ on $X$ such that
  $\CO_{E\times X}(\fD_b+\fD_c)=p_E^*\CalD\otimes p_X^*\CL_X$. To determine $\CL_X$ we consider
  the restrictions to $e\times X$ and use the canonical isomorphisms
  \begin{multline*}\CO_{E\times X}(\fD_b)|_{e\times X}\cong\varpi^*_{E^{(b)}}\CO_{E^{(b)}}(Y_b),\
  \CO_{E\times X}(\fD_c)|_{e\times X}\cong\varpi^*_{E^{(c)}}\CO_{E^{(c)}}(Y_c),\\
  \bq^*\CO_{E^{(a)}}(Y_a)\cong\CO_{E^{(b)}}(Y_b)\boxtimes\CO_{E^{(c)}}(Y_c).\end{multline*}

  Since $\varpi_{E^{(b)}}\colon \fD_b\to E^{(b)}$ is finite flat, we have the norm morphism
  \[\on{Nm}_{\fD_b/E^{(b)}}\colon \on{Pic}(\fD_b)\to \on{Pic}(E^{(b)}).\]
  For any line bundle $\CK$ on $E$ we have an isomorphism
  \begin{equation}
    \label{3.2}
    \CK^{(b)}\simeq\on{Nm}_{\fD_b/E^{(b)}}(\varpi_E^*\CK).
  \end{equation}
  Indeed, we can cover $E$ with open affine charts $U_i$ such that $U_i^{(b)}$ cover $E^{(b)}$,
  and $\CK|_{U_i}$ is trivial. Then we claim that both sides are given by the same transition
  functions. In effect, this follows from the fact that for a regular function $u$ on a smooth
  affine curve $C=\on{Spec}(A)$, one has
  \[\on{Nm}_{\fD_C/C^{(b)}}(\varpi_C^*u)=u^{\otimes n}\in\Sym{}\!^n(A),\]
  where $C\stackrel{\varpi_C}{\longleftarrow}\fD_C\to C^{(b)}$ is the universal divisor. The latter
  claim easily reduces to the case when $u$ is the coordinate on the affine line.

  We denote by $\varpi\colon \fD_b\to X$ the natural projection. We have an isomorphism
  \begin{equation}
    \label{(2)}
    \CO_X(\Delta^{(b),(c)})\simeq
    \det\!\varpi_*\CO_{\fD_b}\otimes\det{}\!^{-1}\varpi_*\big(\CO_{E\times X}(-\fD_c)|_{\fD_b}\big).
  \end{equation}
  Indeed, one can identify $\Delta^{(b),(c)}$ with the locus where the morphism of vector bundles
  on $X$, $\varpi_*\big(\CO_{E\times X}(-\fD_c)|_{\fD_b}\big)\to\varpi_*\CO_{\fD_b}$ fails to be an
  isomorphism. Passing to determinants we get~(\ref{(2)}).

  Recall that for any finite flat morphism $f\colon Y\to Z$ and a line bundle $\CL$ on $Y$
  we have an isomorphism
  \begin{equation}
    \label{(3)}
    \det\!f_*\CL\simeq\on{Nm}_{Y/Z}(\CL)\otimes\det\!f_*\CO_Y.
  \end{equation}

  We have to construct an isomorphism
  \begin{equation}
    \label{3.4}
    \phi_{b,c}\colon\bp^*(\CalD^{(b)})\otimes\bq^*\CO(b)\iso\bp^*\bfomega_{E^{(b)}}^{-2}(\Delta^{(b),(c)}).
  \end{equation}

  Recall that $\bp^*\Omega^1_{E^{(b)}}\simeq\varpi_*\CO_{\fD_b}$, and hence
  $\bp^*\bfomega_{D^{(b)}}\simeq\det\varpi_*\CO_{\fD_b}$. The trivialization of $\bomega_E$
  (see~Remark~\ref{triv can}) induces an isomorphism of $\CO_{E\times X}(\fD_b)|_{\fD_b}$
  and the relative canonical line bundle for
  $\varpi\colon \fD_b\to X$. Hence, using~(\ref{(1)}) along with the relative Grothendieck-Serre
  duality for $\varpi$, we get an isomorphism on $E\times X$:
  \[\varpi_*\big(\CO(-\fD_c)|_{\fD_b}\big)\simeq
  \varpi_*\big((\CO_{E\times X}(\fD_b)\otimes p_E^*\CalD^{-1}\otimes\CM^{-1})|_{\fD_b}\big)
  \simeq\varpi_*\big((p_E^*\CalD\otimes\CM)|_{\fD_b}\big)^\vee,\]
  where $\CM=p_X^*\bq^*\CO(1)$. Since $\CM|_{\fD_b}\simeq\varpi^*\bq^*\CO(1)$, we get an isomorphism
  \[\det{}\!^{-1}\varpi_*\big(\CO_{E\times X}(-\fD_c)|_{\fD_b}\big)\simeq
  \det\!\big(\varpi_*(p_E^*\CalD|_{\fD_b})\otimes\bq^*\CO(1)\big)\simeq
  \det\!\varpi_*(p_E^*\CalD|_{\fD_b})\otimes\bq^*\CO(b).\]
  Using~(\ref{(3)}), we can rewrite this as
  \[\det{}\!^{-1}\varpi_*\big(\CO(-\fD_c)|_{\fD_b}\big)\simeq
  \on{Nm}_{\fD_b/X}(p_E^*\CalD)\otimes\det\!\varpi_*\CO_{\fD_b}\otimes\bq^*\CO(b).\]
  Plugging this into~(\ref{(2)}) we get
  \[\CO(\Delta^{(b),(c)})\simeq\on{Nm}_{\fD_b/X}(p_E^*\CalD)\otimes\det\!^2\varpi_*\CO_{\fD_b}\otimes\bq^*\CO(b)
  \simeq\on{Nm}_{\fD_b/X}(p_E^*\CalD)\otimes\bp^*\bfomega_{E^{(b)}}^2\otimes\bq^*\CO(b),\]
  which gives rise to the desired isomorphism~(\ref{3.4}) by the virtue of~(\ref{3.2}).

  This completes the proof of (a).

  \medskip

  (b) The isomorphism~(\ref{new 2}) can be viewed as a way to choose a section $s_{D,D'}$ of
  $\CalD$ vanishing on $D+D'$ for $(D,D')\in X$. Away from the diagonals, writing
  $D=w_1+\ldots+w_b$, we have a collection of restrictions $(s_{D,D'}|_{w_r})_{1\leq r\leq b}$ defining
  an isomorphism $H^0(D,\CO_E(D)|_D)\cong\bigoplus_{r=1}^b\CalD|_{w_r}$. Hence, the tensor product
  of these restrictions defines an isomorphism
  $\det H^0(D,\CO_E(D)|_D)\cong\bigotimes_{r=1}^b\CalD|_{w_r}$.
  More precisely, away from all the diagonals, the isomorphism $\tau_b$ of~(\ref{new 2}) gives
  rise to an isomorphism
  \[\sigma_b\colon \det\!\varpi_*\CO_{E\times X}(\fD_b)|_{\fD_b}\iso
  \det\!\varpi_*(p_E^*\CalD|_{\fD_b})\otimes\bq^*\CO_{E^{(a)}}(bY_a).\]
  Then over $X^{b,(c)}$ (notation introduced right before~Lemma~\ref{(4)}) we have an equality
  \begin{equation}
    \label{3.5}
    \rho^*\sigma_b=\bigotimes_{r=1}^b\rho_r^*\sigma_1.
  \end{equation}
  Indeed, let us consider the pullback of $\tau_b$ under
  $\on{Id}_E\times\rho\colon E\times X^{b,(c)}\to E\times X$. Away from the diagonals
  we have $\widetilde\fD_b:=(\on{Id}_E\times\rho)^{-1}(\fD_b)=\bigsqcup_{r=1}^b\widetilde\fD_b(r)$,
  where $\widetilde\fD_b(r):=(\on{Id}_E\times\rho_r)^{-1}(\fD_1)$. Note that the projection
  $\widetilde\fD_b(r)\to X^{b,(c)}$ is an isomorphism. Hence,
  \[(\on{Id}_E\times\rho)^*\CO_{E\times X}(\fD_b)|_{\widetilde\fD_b(r)}\cong
  (\on{Id}_E\times\rho_r)^*\CO_{E\times X}(\fD_1)|_{\widetilde\fD_b(r)}.\]
  But for any $r=1,\ldots,b$ we have
  \[(\on{Id}_E\times\rho)^*\tau_b|_{\widetilde\fD_b(r)}=(\on{Id}_E\times\rho_r)^*\tau_1\]
  (note that we can ignore $\fD_c$ since we are working away from diagonals). In effect, both
  sides have the same restrictions to $e\in E$.

  Now $\rho^*\varpi_*\CO_{E\times X}(\fD_b)|_{\fD_b}$ decomposes into a direct sum of the line bundles
  $(\on{Id}_E\times\rho)^*\CO_{E\times X}(\fD_b)|_{\widetilde\fD_b(r)}$ on
  $\widetilde\fD_b(r)\simeq X^{b,(c)}$, so taking the determinant of $\rho^*\varpi_*\tau_b|_{\fD_b}$
  corresponds to taking the product of restrictions to $\widetilde\fD_b(r)$ over $r=1,\ldots,b$.

  It follows that the isomorphisms~(\ref{3.4}) can be chosen in a factorizable fashion away from
  all the diagonals, that is satisfying
  \[\rho^*\phi_{b,c}=\bigotimes_{r=1}^b\rho_r^*\phi_{1,a-1}.\]
  Indeed, we replace $\CO(1)$ on $\on{AJ}^{-1}(\CalD)\cong\BP^{a-1}\subset E^{(a)}$
  by the isomorphic line bundle $\CO_{E^{(a)}}(Y_a)|_{\on{AJ}^{-1}(\CalD)}$ and use the canonical
  isomorphism~(\ref{new 2}). Going through the construction of isomorphisms~(\ref{3.4})
  restricted to the complement of all the diagonals, we see that each step is factorizable,
  the first step being dealt with in~(\ref{3.5}). The key point in the other steps is that
  the base change of the relative divisor $\fD_b$ over $X$ with respect to $X^{b,(c)}\to X$
  becomes a disjoint union of $b$ points. So the determinant of the push-forward decomposes as
  tensor product, as well as the norm of a line bundle, etc. Note finally that the
  isomorphism~(\ref{(2)}) reduces to the identity away from the diagonals.

  This completes the proof of (b).  
\end{proof}

In the next lemma it will be convenient to use the notation
$\bp_i\colon E^\beta\times E^\gamma\to E^{(b_i)}$ and
$\bq_i\colon E^\beta\times E^\gamma\to E^{(a_i)}$ for the compositions of $\bp,\bq$ with the
projections to the respective $i$-th factors.

\begin{lem}
  \label{(5)}
  \textup{(a)} We have an isomorphism of line bundles on $X^{\beta,\gamma}$:
  \begin{multline*}
    \psi_{\beta,\gamma}\colon\bigotimes_{i\in I}\bp_i^*\Big(\bigotimes_{h\in Q_1 : \on{o}(h)=i}
    \big(\CalD_{\on{i}(h)}^{-1}\big)^{(b_i)}\Big)
    \otimes\bigotimes_{h\in Q_1}\bq_{\on{i}(h)}^*\Big(\CO_{\BP^{a_{\on{i}(h)}-1}}(-b_{\on{o}(h)})\Big)\\
    \iso\CO_{X^{\beta,\gamma}}\Big(-\sum_{h\in Q_1}\Delta_{\on{o}(h)\on{i}(h)}^\beta
    -\sum_{h\in Q_1}\Delta_{\on{o}(h)\on{i}(h)}^{\beta,\gamma}\Big).
  \end{multline*}
  \textup{(b)} We can choose a collection of isomorphisms $\psi_{\beta,\gamma}$ in \textup{(a)}
  satisfying the following factorization property:
  \[\rho^*\psi_{\beta,\gamma}=
  \bigotimes_{h\in Q_1}^{1\leq r\leq b_{\on{o}(h)}}\rho_{\on{o}(h),r}^*\psi_{\alpha_{\on{o}(h)},\alpha-\alpha_{\on{o}(h)}}\]
  away from the preimage of all the diagonals in $E^\alpha$.
\end{lem}

\begin{proof}
  (a) It suffices to construct the desired isomorhism when $I$ consists of two vertices connected
  by an arrow as follows: $i\to j$. We denote by $\fD_{b_i},\fD_{b_j},\fD_{c_i},\fD_{c_j}\subset E\times X$
  the relative divisors pulled back from the universal divisors over the corresponding
  symmetric powers of $E$ (here $X=X^{\beta,\gamma}$). We denote by $\varpi\colon \fD_{b_i}\to X$
  the natural projection which is a finite flat morphism of degree $b_i$. Similarly to~(\ref{(2)}), we
  have isomorphisms
  \[\CO_X(-\Delta^\beta_{ij})\simeq
  \det\!^{-1}\varpi_*\CO_{\fD_{b_i}}\otimes\det\!\varpi_*\big(\CO_{E\times X}(-\fD_{b_j})|_{\fD_{b_i}}\big)\simeq
  \on{Nm}_{\fD_{b_i}/X}\big(\CO_{E\times X}(-\fD_{b_j})|_{\fD_{b_i}}\big),\]
  \[\CO_X(-\Delta^{\beta,\gamma}_{ij})\simeq
\det\!^{-1}\varpi_*\CO_{\fD_{b_i}}\otimes\det\!\varpi_*\big(\CO_{E\times X}(-\fD_{c_j})|_{\fD_{b_i}}\big)\simeq
  \on{Nm}_{\fD_{b_i}/X}\big(\CO_{E\times X}(-\fD_{c_j})|_{\fD_{b_i}}\big).\]
  Thus, we have an isomorphism
  \[\CO_X(-\Delta^\beta_{ij}-\Delta^{\beta,\gamma}_{ij})\simeq
  \on{Nm}_{\fD_{b_i}/X}\big(\CO_{E\times X}(-\fD_{b_j}-\fD_{c_j})|_{\fD_{b_i}}\big).\]
  Using the isomorphism (recall the projections
  $E\stackrel{p_E}{\longleftarrow} E\times X\stackrel{p_X}{\longrightarrow}X$)
  \[\CO_{E\times X}(\fD_{b_j}+\fD_{c_j})\simeq p_E^*\CalD_j\otimes p_X^*\bq^*\CO(0,1)\]
  together with~(\ref{3.2}), we get an isomorphism
  \begin{multline*}
    \on{Nm}_{\fD_{b_i}/X}\big(\CO_{E\times X}(-\fD_{b_j}-\fD_{c_j})|_{\fD_{b_i}}\big)\simeq
    \on{Nm}_{\fD_{b_i}/X}(p_E^*\CalD_j^{-1})\otimes\bq^*\CO(0,-b_i)\\
    \simeq\bp^*\big((\CalD_j^{-1})^{(b_i)}\boxtimes\CO_{E^{(c_i)}}\big)\otimes\bq^*\CO(0,-b_i),
    \end{multline*}
  and (a) follows.

  \medskip

  The proof of (b) is similar to the one of~Lemma~\ref{(4)}(b). It is still enough to consider
  the case when $I$ consists of two vertices connected by an arrow $i\to j$.
  We construct a factorizable collection of $\psi_{\beta,\gamma}$ in stages. At the first step we
  note that there is an evident morphism
  $\varrho\colon X^{\beta,\gamma}\to X^{b_i\alpha_i,c_i\alpha_i+(b_j+c_j)\alpha_j}$ (addition of $j$-colored
  divisors), and we choose
  $\psi_{\beta,\gamma}$ as $\varrho^*\psi_{b_i\alpha_i,c_i\alpha_i+(b_j+c_j)\alpha_j}$.
  So it suffices to construct a factorizable collection of $\psi_{\beta,\gamma}$ for the particular
  case when $\beta$ is a multiple of $\alpha_i$.

  Next, we have a cartesian diagram 
  \[\begin{CD}
  X^{\alpha_i,\gamma'} @<{\rho_{i,r}}<< X^{|b_i\alpha_i|,\gamma} @>>> E^{b_i}\times E^\gamma\\
  @. @V{\rho}VV @VVV\\
  @. X^{b_i\alpha_i,\gamma} @>>> E^{(b_i)}\times E^\gamma,
  \end{CD}\]
  where $\gamma'=\gamma+(b_i-1)\alpha_i$. We have to choose our isomorphisms $\psi$ so that
  $\rho^*\psi_{b_i\alpha_i,\gamma}=\bigotimes_{r=1}^{b_i}\rho_{i,r}^*\psi_{\alpha_i,\gamma'}$.
  To this end note that $X^{\alpha_i,\gamma'}\simeq E^{(c'_i)}\times\BP\Gamma(E,\CalD_j)$,
  and we can take $\psi_{\alpha_i,\gamma'}$ to be the pullback of the universal section in the space
  $\Gamma\big(E\times\BP\Gamma(E,\CalD_j),\CalD_j\boxtimes\CO(1)\big)$ under the projection
  $E^{(c'_i)}\to E$ sending $D\in E^{(c'_i)}$ to the unique $x\in E$ such that $D+x\sim\CalD_i$.

  The lemma is proved.
\end{proof}

\subsection{Segre embeddings involved in the definition of zastava spaces}
\label{segre}

Recall that both the zastava spaces we are interested in (Coulomb and Mirkovi\'c) are defined as closures of certain Serge embeddings in projective bundles over the configuration spaces.
In this subsection we write down the equations of the images of these Segre embeddings.

We redenote
\[E^\beta\xleftarrow{\bp}E^\beta\times E^\gamma\xrightarrow{\bq}E^\alpha\] by
\[E^\beta\xleftarrow{\bp_\beta}E^\beta\times E^\gamma\xrightarrow{\bq^\beta}E^\alpha\]
since $\beta$ will vary. The ruled surface $\BP(\CK_i\oplus\CO_E)\to E$ will be denoted $P_i\to E$.
We have the Segre embedding
\begin{equation}
  \label{Segre}\big(\prod_{i\in I}P_i^{a_i}\big)/S_\alpha\hookrightarrow
  \BP\big(\boxtimes_{i\in I}((\CK_i\oplus\CO_E)^{\boxtimes a_i})\big)/S_\alpha.
\end{equation}
For any vector bundle
$\CW$ over $E$ we have an isomorphism $\BP(\CW^{\boxtimes a})/S_a\simeq\BP(\CW^{(a)})$, where
$\CW^{(a)}$ is the subsheaf of $S_a$-invariants in the pushforward of $\CW^{\boxtimes a}$ from
$E^a$ to $E^{(a)}$. Thus, the RHS of~(\ref{Segre}) is equal to
$\BP\big(\boxtimes_{i\in I}(\CK_i\oplus\CO_E)^{(a_i)}\big)$. Furthermore, we have a decomposition
\[\boxtimes_{i\in I}(\CK_i\oplus\CO_E)^{(a_i)}=\bigoplus_{\beta+\gamma=\alpha}\bq^\beta_*\bp_\beta^*\CK^\beta\]
(recall that $\CK^\beta:=\boxtimes_{i\in I}\CK_i^{(b_i)}$). Thus we can rewrite the Segre map as
\begin{equation}
  \label{void}
  \big(\prod_{i\in I}P_i^{a_i}\big)/S_\alpha\hookrightarrow
  \BP\big(\bigoplus_{\beta+\gamma=\alpha}\bq^\beta_*\bp_\beta^*\CK^\beta\big).
\end{equation}
Let $(w_{i,r})_{i\in I}^{1\leq r\leq a_i}$ be a collection of distinct points of $E$. Then the fiber
of the RHS of~(\ref{void}) at the corresponding point of $E^\alpha$ is the projectivization of
\[\bigotimes_{i\in I}^{1\leq r\leq a_i}(\CK_i\oplus\CO_E)|_{w_{i,r}}=
\bigoplus_\aleph\bigotimes_{(i,r)\in\aleph}\CK_i|_{w_{i,r}},\]
where the summation runs over all the subsets $\aleph$ of the set of pairs
$(i,r)_{i\in I}^{1\leq r\leq a_i}$. For $s_{i,r}\in\CK_i|_{w_{i,r}}$ the Segre embedding is given by
\[\big((s_{i,r},1)_{i\in I}^{1\leq r\leq a_i}\big)\mapsto\bigotimes_{i\in I}^{1\leq r\leq a_i}(s_{i,r},1)=
\big(\bigotimes_{(i,r)\in\aleph}s_{i,r}\big)_\aleph.\]
The equations cutting out the image of Segre embedding can be formulated as a certain factorization
property of the sections' collection $(s_\aleph)$. More precisely, let us consider a morphism
\[\bq^\aleph\colon E^\aleph\times E^\gamma\to E^\alpha,\ \big((w_{i,r})_{(i,r)\in\aleph},D\big)\mapsto
\sum_{(i,r)\in\aleph}w_{i,r}+D,\]
where $\beta:=\sum_{(i,r)\in\aleph}\alpha_i$, and $\gamma:=\alpha-\beta$. Let also
$\bp_\aleph\colon E^\aleph\times E^\gamma\to E^\aleph$ denote the projection. Also, for any
$(i,r)\in\aleph$ we consider a morphism
\[\rho_{i,r}\colon E^\aleph\times E^\gamma\to E^{\alpha_i}\times E^{\alpha-\alpha_i},\
\big((w_{i,r})_{(i,r)\in\aleph},D\big)\mapsto(w_{i,r},\sum_{(j,s)\in\aleph\setminus\{(i,r)\}}w_{j,s}+D).\]
Note that $\bq^{\alpha_i}\circ\rho_{i,r}=\bq^\aleph$. Then we have natural morphisms of vector bundles
\[\kappa_\aleph\colon \bq^\beta_*\bp_\beta^*\CK^\beta\hookrightarrow
\bq^\aleph_*\bp_\aleph^*\boxtimes_{(i,r)\in\aleph}\CK_i=
\bq^\aleph_*\big(\bigotimes_{(i,r)\in\aleph}\rho_{i,r}^*\bp_{\alpha_i}^*\CK_i\big),\]
\begin{multline*}
  \varkappa_\aleph\colon \bigotimes_{(i,r)\in\aleph}(\bq^{\alpha_i}_*\bp_{\alpha_i}^*\CK_i)\hookrightarrow
\bigotimes_{(i,r)\in\aleph}(\bq^{\alpha_i}_*\rho_{i,r*}\rho_{i,r}^*\bp_{\alpha_i}^*\CK_i)\\
=\bigotimes_{(i,r)\in\aleph}(\bq^\aleph_*\rho_{i,r}^*\bp_{\alpha_i}^*\CK_i)\to
\bq^\aleph_*\big(\bigotimes_{(i,r)\in\aleph}\rho_{i,r}^*\bp_{\alpha_i}^*\CK_i\big).
\end{multline*}
We are finally able to state the Segre equations on the sections' collection $(s_\aleph)$.
We assume that the section $s_\emptyset$ corresponding to the empty subset $\aleph=\emptyset$
is identically equal to 1 (this assumption is harmless since we are working in the
projectivization.) Then the equations read
\begin{equation}
  \label{Segre eq}
  \kappa_\aleph(s_\aleph)=\varkappa_\aleph(\bigotimes_{(i,r)\in\aleph}s_{i,r}).
\end{equation}

\subsection{Proof of Theorem~\ref{reductio}}
\label{the proof}
According to Proposition~\ref{bundles}, the summands in $\BV_\CK^\alpha|_{\on{AJ}^{-1}(\CalD)}$
are isomorphic to the corresponding summands in $\BU_{\CK'}^\alpha|_{\on{AJ}^{-1}(\CalD)}$ twisted by
$\CO(\beta')$ where $\beta'$ depends linearly on $\beta$ numbering the summand.
The isomorphism is given by the tensor product of isomorphisms
$\phi_{\beta,\gamma}$~(Lemma~\ref{(4)}(a)) and $\psi_{\beta,\gamma}$~(Lemma~\ref{(5)}(a)). Comparing with
the definition of $T^{\on{ad}}$-action in the first paragraph of~\S\ref{ham red}, we see that
the quotients $\big(\BP\BV_\CK^\alpha|_{\on{AJ}^{-1}(\CalD)}\big)/T^{\on{ad}}$ and
$\big(\BP\BU_{\CK'}^\alpha|_{\on{AJ}^{-1}(\CalD)}\big)/T^{\on{ad}}$ coincide.

It remains to check that the closures of the images of Segre embeddings correspond to each other
under the
above identification. Let $X^\circ$ stand for the open subset of $\on{AJ}^{-1}(\CalD)$ defined
as the complement to all the diagonals in $E^\alpha$. The factorization properties
of~Lemma~\ref{(4)}(b) and~Lemma~\ref{(5)}(b) compared with the Segre equations~(\ref{Segre eq})
show that the isomorphism of the previous paragraph restricted to $X^\circ$ respects the
Segre embeddings.

The theorem is proved.

\section{Feigin-Odesskii moduli space}
\label{shest}

\subsection{A symplectic moduli stack}
\label{fo}
We fix a $G$-bundle $\CF_G$ on $E$ and a $T$-bundle $\CL_T$ of degree $-\alpha$ on $E$.
We denote by $M(\CF_G,\CL_T)$ the moduli stack of $B$-structures $\varphi$ on $\CF_G$ equipped
with an isomorphism $\Ind_B^T\varphi\iso\CL_T$. It can be upgraded to a derived stack equipped with a
(0-shifted) symplectic structure. Indeed, recall~\cite{ptvv} that both $\Bun_G$ and $\Bun_T$
(moduli stacks of $G$- and $T$-bundles on $E$) carry the canonical 1-shifted symplectic structures.
Furthermore,~\cite[Example~4.11]{saf} equips the correspondence $\Bun_B\to\Bun_G\times\Bun_T$
with a canonical Lagrangian structure. Finally, the embeddings of stacky points
$[\CF_G]=\pt/\!\on{Aut}(\CF_G)\to\Bun_G$ and $[\CL_T]=\pt/\!\on{Aut}(\CL_T)\to\Bun_T$ are equipped
with the natural Lagrangian structures similarly to~\cite[Theorem~3.18]{hp}.
We consider the homotopy fibre product
\begin{equation}
  \label{homotopy}
  \begin{CD}
M^\der(\CF_G,\CL_T) @>>> \Bun_B\\
@VVV  @VVV\\
[\CF_G]\times[\CL_T] @>>> \Bun_G\times\Bun_T.
  \end{CD}
\end{equation}
The truncation of $M^{\on{der}}(\CF_G,\CL_T)$ coincides with $M(\CF_G,\CL_T)$.

Now $M^\der(\CF_G,\CL_T)$ is a derived Lagrangian intersection and hence acquires a 0-shifted
symplectic structure by~\cite{ptvv}, cf.\ a similar construction~\cite{spa} for the base curve of
genus~0.

\subsection{Tangent spaces}
\label{tangent}
For a point $\varphi$ in $M^\der(\CF_G,\CL_T)$,
we denote by $\ft_\varphi\twoheadleftarrow\nobreak\fb_\varphi\hookrightarrow\nobreak\fg_\varphi$
the vector bundles
on $E$ associated with the adjoint representations of $B$ (clearly, $\ft_\varphi$ is trivial).
The tangent complex at the corresponding point $\CF_G$ of $\Bun_G$ is $R\Gamma(E,\fg_\varphi[1])$,
and the tangent complex at the corresponding point $\CL_T$ of $\Bun_T$ is
$R\Gamma(E,\ft_\varphi[1])$,
while the tangent complexes at the corresponding stacky points $[\CF_G]$ and $[\CL_T]$ are
the truncations $\tau_{<0}R\Gamma(E,\fg_\varphi[1])$ and $\tau_{<0}R\Gamma(E,\ft_\varphi[1])$
respectively. From~(\ref{homotopy}) we deduce the homotopy fibre square
\begin{equation}
  \label{fibre}
  \begin{CD}
T_\varphi M^\der(\CF_G,\CL_T) @>>> R\Gamma(E,\fb_\varphi[1])\\
@VVV  @VVV\\
\tau_{<0}R\Gamma(E,\fg_\varphi[1]\oplus\ft_\varphi[1]) @>>>
R\Gamma(E,\fg_\varphi[1]\oplus\ft_\varphi[1]).
  \end{CD}
\end{equation}
Hence the tangent space $T_\varphi M^\der(\CF_G,\CL_T)$ is canonically isomorphic to
the total complex
\begin{equation}
  \label{total}
  T_\varphi M^\der(\CF_G,\CL_T)\cong
  \big[R\Gamma(E,\fb_\varphi[1])\to\tau_{\geq0}R\Gamma(E,\fg_\varphi[1]\oplus\ft_\varphi[1])\big].
\end{equation}

Furthermore, we have an exact sequence of $B$-modules
$0\to\fb\to\fg\oplus\ft\to\fb^\svee\to0$
and the corresponding exact sequence of associated vector bundles
\begin{equation}
  \label{bgtb}
  0\to\fb_\varphi\to\fg_\varphi\oplus\ft_\varphi\to\fb_\varphi^\svee\to0.
\end{equation}
Replacing the right column of~(\ref{fibre}) by its cone $R\Gamma(E,\fb_\varphi^\svee[1])$,
we can rewrite
\begin{equation}
  \label{total'}
  T_\varphi M^\der(\CF_G,\CL_T)\cong\big[\tau_{\leq0}R\Gamma(E,\fg_\varphi\oplus\ft_\varphi)\to
    R\Gamma(E,\fb_\varphi^\svee)\big].
\end{equation}
On the other hand, the exact sequence~(\ref{bgtb}) is clearly selfdual, and the Serre duality
on $E$ gives rise to a perfect pairing between the RHS of~(\ref{total}) and~(\ref{total'}).
This perfect pairing on $T_\varphi M^\der(\CF_G,\CL_T)$ is nothing but the symplectic structure
of~\S\ref{fo}.

Equivalently, at a smooth point $\varphi$ in $M^\der(\CF_G,\CL_T)$, the Poisson bivector is defined
using the differential $d_2$ of the second page of the hypercohomology spectral sequence for
the complex $\fn_\varphi\to\fg_\varphi\to\fg_\varphi/\fb_\varphi$ of vector bundles on $E$.

\begin{rem}
The original definition of Feigin-Odesskii in \cite{fo} is that of the Poisson bivector on the (smooth points of) moduli stack of
$B$-bundles (or more generally $P$-bundles where $P$ is a parabolic subgroup), which is constructed similarly
to our definition above. As we have discussed in~\S \ref{fo}, this Poisson bivector is a classical shadow of the $0$-shifted
Poisson structure on $\Bun_B$ associated with the natural Lagrangian structure on $\Bun_B\to \Bun_G\times \Bun_T$.
So the truncation of the smooth part of $M^\der(\CF_G,\CL_T)$ is a symplectic leaf of the original Feigin-Odesskii
Poisson structure.
\end{rem}

\subsection{Regular induced case}
We consider a special case when a $G$-bundle $\CF_G$ is induced from a degree zero $T$-bundle
$\CK_T\colon \CF_G=\Ind_T^G\CK_T$. Moreover, we assume that $\CK_T$ is {\em regular}, that is,
for any root $\alpha^\svee\in R^\vee$, the associated line bundle $\CK^{\alpha^\svee}$ is nontrivial.
Then for any dominant weight $\lambda^\svee\in\Lambda^{\vee+}$ the corresponding vector
bundle $\CV_\CF^{\lambda^\svee}$ (associated to the irreducible $G$-module $V^{\lambda^\svee}$)
canonically splits into direct sum of its weight components. In particular, we have a projection
$\xi^{\lambda^\svee}\colon\CV_\CF^{\lambda^\svee}\twoheadrightarrow\CK^{\lambda^\svee}$ onto the lowest weight
component line bundle (associated to the character $w_0\lambda^\svee\colon T\to\BC^\times$).
The collection of $\xi^{\lambda^\svee},\ \lambda^\svee\in\Lambda^{\vee+}$, is subject to Pl\"ucker
relations. If we act on our data by an automorphism of $\CK_T$ given by an element $t\in T$,
the projection $\xi^{\lambda^\svee}$ will change to
$\lambda^\svee(t)\cdot\xi^{\lambda^\svee}$, cf.~Definition~\ref{ell zas}(5).
Since $\on{Aut}(\Ind_T^G\CK_T)=T$ by regularity of $\CK_T$ (see e.g.~\cite[Proposition 3.10]{fmw}
and~\cite[Theorem 4.1(i)]{fm}), the collection of projections
$\xi^{\lambda^\svee}\colon\CV_\CF^{\lambda^\svee}\twoheadrightarrow\CK^{\lambda^\svee}$ subject to
Pl\"ucker relations is well defined up to the action of $T$.

Another piece of data in the definition of the Feigin-Odesskii moduli space $M(\CF_G,\CL_T)$
is the $T$-bundle $\CL_T$. For a fundamental weight $\omega^\svee_i$ we consider the associated
line bundle $\CL^{\omega^\svee_i}$, and we set $\CalD_i:=\CL^{-\omega^\svee_i}\otimes\CK^{\omega^\svee_i}$.
We have $\CalD_i\in\on{Pic}^{a_i}E$, where $\alpha=\sum a_i\alpha_i$ (recall that $-\alpha$ is the
degree of $\CL_T$). We set $\CalD=(\CalD_i)_{i\in I}$.

We consider an open substack $\oM^\der(\Ind_T^G\CK_T,\CL_T)\subset M^\der(\Ind_T^G\CK_T,\CL_T)$
given by the condition that the compositions
$\CL^{\lambda^\svee}\hookrightarrow\CV_\CF^{\lambda^\svee}\stackrel{\xi^{\lambda^\svee}}\twoheadrightarrow
\CK^{\lambda^\svee}$ never vanish identically. Ignoring the derived structure we obtain an
open substack $\oM(\Ind_T^G\CK_T,\CL_T)\subset M(\Ind_T^G\CK_T,\CL_T)$.

\begin{prop}
  \label{feod}
For a regular $T$-bundle $\CK_T$, we have a natural isomorphism
\[^{\vphantom{\alpha}}_\CalD\uoZ^\alpha_\CK\cong\oM(\Ind_T^G\CK_T,\CL_T).\]
\end{prop}

\begin{proof}
Comparing with~Definitions~\ref{ell zas},\ref{red zas}, we see that the collection of projections
$\xi^{\lambda^\svee}\colon\CV_\CF^{\lambda^\svee}\twoheadrightarrow\CK^{\lambda^\svee}$ along with
the collection of embeddings $\CL^{\lambda^\svee}\hookrightarrow\CV_\CF^{\lambda^\svee}$ defines a
point of reduced zastava $^{\vphantom{\alpha}}_\CalD\uoZ^\alpha_\CK$. Thus we obtain a morphism
$\Upsilon\colon\oM(\Ind_T^G\CK_T,\CL_T)\to{}^{\vphantom{\alpha}}_\CalD\uoZ^\alpha_\CK$.
We have to check that $\Upsilon$ is an isomorphism. To this end note that a twisted
$U_-$-structure on a $G$-bundle $\CF$ defines a filtration on the associated vector bundle
$\CV^{\lambda^\svee}_\CF$ for any dominant weight $\lambda^\svee$. The successive quotients
of this filtration are of the form
$\CK^{\mu^\svee}\otimes V^{\lambda^\svee}(w_0\mu^\svee)$ for the weights $\mu^\svee$ of the irreducible
$G$-module $V^{\lambda^\svee}$. The regularity condition on $\CK_T$ ensures that this filtration
splits canonically, i.e.\
$\CV^{\lambda^\svee}_\CF\cong\bigoplus_{\mu^\svee}\CK^{\mu^\svee}\otimes V^{\lambda^\svee}(w_0\mu^\svee)$.
This collection of splittings defines a reduction of $\CF$ to $T\subset G$, that is a
canonical isomorphism $\CF\cong\Ind_T^G\CK_T$. This construction provides a morphism
$^{\vphantom{\alpha}}_\CalD\uoZ^\alpha_\CK\to\oM(\Ind_T^G\CK_T,\CL_T)$ inverse to $\Upsilon$.
\end{proof}

\begin{rem}
The conclusion of~Proposition~\ref{feod} breaks down if $\CK_T$ is not regular. For example, if
$\CK_T$ is trivial, and hence $\CF_G$ is a trivial $G$-bundle for $G=\SL(2)$, the LHS
$^{\vphantom{\alpha}}_\CalD\uoZ^\alpha_{\CK_{\on{triv}}}$ contains a point corresponding to a line subbundle
$\CL\subset\CV$ in a rank~2 vector bundle $\CV$
on $E$ that is a {\em nontrivial} extension of $\CO_E$ with $\CO_E$. But the RHS
$\oM(\Ind_T^G\CK_T,\CL_T)$ does not contain such a point.
\end{rem}

\subsection{Comparison of symplectic structures}
The reduced zastava space $^{\vphantom{\alpha}}_\CalD\uoZ^\alpha_\CK$
carries a symplectic structure by~Theorem~\ref{reductio}
and~Proposition~\ref{pois}, while the Feigin-Odesskii moduli space
$\oM(\Ind_T^G\CK_T,\CL_T)$ carries a symplectic structure
by~\S\ref{fo}. The rest of this Section is devoted to an identification of these two
symplectic structures. Namely, let $\{,\}_{\CK'}$ denote the Poisson bracket on
$^{\vphantom{\alpha}}_\CalD\uoZ^\alpha_\CK\simeq{}^C_\CalD\uoZ^\alpha_{\CK'}$ defined as the Hamiltonian
reduction of the bracket of~Proposition~\ref{pois}. Let $\{,\}_{FO}$ denote the Poisson
bracket on $\oM^\der(\Ind_T^G\CK_T,\CL_T)$ defined in~\S\ref{fo}. It restricts to the same named
Poisson bracket on the smooth open locus of $\oM^\der(\Ind_T^G\CK_T,\CL_T)$ where 
the derived structure is trivial.

\begin{thm}
  \label{myksas}
  The isomorphism of Proposition~\ref{feod} restricted to the smooth open loci of
  $^{\vphantom{\alpha}}_\CalD\uoZ^\alpha_\CK$ and $ \oM(\Ind_T^G\CK_T,\CL_T)$
  takes the Poisson structure $\{,\}_{\CK'}$ to $\{,\}_{FO}$.
\end{thm}

\begin{rem}
  The stack $^{\vphantom{\alpha}}_\CalD\uoZ^\alpha_\CK$ can be upgraded to a derived stack
  $({}^{\vphantom{\alpha}}_\CalD\uoZ^\alpha_\CK)^\der$ by its very definition (since the Abel-Jacobi
  morphism $\on{AJ}_Z\colon\oZ^\alpha_\CK\to\prod_{i\in I}\on{Pic}^{a_i}E$ is not smooth in general
  for $\on{rk}G>1$, its level set acquires a natural derived structure).
  Similarly, the stack of reduced Coulomb zastava $^C_\CalD\uoZ^\alpha_\CK$ can be upgraded
  to a derived stack $({}^C_\CalD\uoZ^\alpha_\CK)^\der$. The isomorphism of~Proposition~\ref{feod}
  can be upgraded to an isomorphism of derived stacks
  $({}^{\vphantom{\alpha}}_\CalD\uoZ^\alpha_\CK)^\der\cong\oM^\der(\Ind_T^G\CK_T,\CL_T)$.
  We also expect but cannot prove that the isomorphism of~Theorem~\ref{reductio} can be
  upgraded to an isomorphism of derived stacks
  $({}^{\vphantom{\alpha}}_\CalD\uoZ^\alpha_\CK)^\der\cong({}^C_\CalD\uoZ^\alpha_\CK)^\der$.
  Thus we expect a symplectomorphism of derived symplectic stacks
  $({}^C_\CalD\uoZ^\alpha_\CK)^\der\cong\oM^\der(\Ind_T^G\CK_T,\CL_T)$.
\end{rem}

\subsection{Compatibility of reduced zastava with Levi factors}
Given a subset $J\subset I$, we denote by $G\supset L_J\supset T$ the corresponding Levi factor. 
For $\alpha=\sum_{i\in I}a_i\alpha_i$, we define
$\alpha_J:=\sum_{i\in J}a_i\alpha_i$. The factorization of zastava for a decomposition
$\alpha=\alpha_J+\alpha_{I\setminus J}$ is a birational isomorphism
\[\oZ^\alpha_\CK\stackrel{\sim}{\dasharrow}\oZ^{\alpha_J}_\CK\times\oZ^{\alpha_{I\setminus J}}_\CK.\]
Composing with the projection onto $\oZ^{\alpha_J}_\CK$ we get a rational dominant morphism
$\oZ^\alpha_\CK\dasharrow\oZ^{\alpha_J}_\CK$.

Note that the derived subgroup $L'_J=[L_J,L_J]$ is also simply connected, and we
can consider its zastava space $\oZ^{\alpha_J}_{\CK_J}(L'_J)$, cf.~Remark~\ref{reductive}.
Here $\CK_J:=(\CK_i)_{i\in J}$ (recall
that $\CK_i=\CK^{-\alpha_i^\svee}$). The natural morphism $\oZ^{\alpha_J}_{\CK_J}(L'_J)\to\oZ^{\alpha_J}_\CK$
is an isomorphism, and we will use it to identify these moduli spaces.

The rational dominant morphism $\oZ^\alpha_\CK\dasharrow\oZ^{\alpha_J}_{\CK_J}(L'_J)$ induces
a rational dominant morphism of reduced zastava
\[\varPi^Z_J\colon{}^{\vphantom{\alpha}}_\CalD\uoZ^\alpha_\CK\dasharrow
         {}^{\vphantom{\alpha}}_{\CalD_J}\uoZ^{\alpha_J}_{\CK_J}(L'_J).\]
Here $\CalD_J$ stands for $(\CalD_i)_{i\in J}$.

Furthermore, the factorization property of Coulomb zastava similarly gives rise to a rational
dominant morphism $^C\!\oZ^\alpha_\CK\dasharrow{}^C\!\oZ^{\alpha_J}_{\CK_J}(L'_J)$ that in turn gives rise
to a rational dominant morphism of reduced Coulomb zastava
\[\varPi^C_J\colon{}^C_\CalD\uoZ^\alpha_\CK\dasharrow{}^{\hphantom{D}C}_{\CalD_J}\!\uoZ^{\alpha_J}_{\CK_J}(L'_J).\]
Both morphisms $\varPi^Z_J,\varPi^C_J$ are Poisson by construction.

\subsection{Compatibility of Feigin-Odesskii moduli spaces with Levi factors}
\label{compa}
For a degree zero regular $T$-bundle $\CK_T$ and $J\subset I$ we consider
the Feigin-Odesskii moduli stack
$\oM_J^\der(\Ind_T^{L_J}\CK_T,\CL_T)$ for the Levi factor $L_J$. We have a rational dominant morphism
\[\varPi^M_J\colon \oM^\der(\Ind_T^G\CK_T,\CL_T)\dasharrow\oM_J^\der(\Ind_T^{L_J}\CK_T,\CL_T)\] constructed
as follows.

Let $P_J\supset B$ denote the corresponding parabolic subgroup, and let $U_J$ denote
the unipotent radical of $P_J$. Then the coinvariants $V^{\lambda^\svee}_{U_J}$ carry a natural
action of $L_J$ and form an irreducible $L_J$-module with lowest weight $w_0\lambda^\svee$
(and with highest weight $w_Jw_0\lambda^\svee$). The natural projection
$V^{\lambda^\svee}\to V^{\lambda^\svee}_{U_J}$ gives rise to the projection
$\xi^{\lambda^\svee}_J\colon\CV^{\lambda^\svee}_\CF\to\CV^{\lambda^\svee}_{\CF,U_J}$. Composing with the
embedding $\CL^{\lambda^\svee}\hookrightarrow\CV^{\lambda^\svee}_\CF$ we obtain a morphism
$\CL^{\lambda^\svee}\to\CV^{\lambda^\svee}_{\CF,U_J}$. However, this morphism is not necessarily an
embedding of a line subbundle; in general it is only an embedding of an invertible subsheaf.
Hence in general it gives rise to a {\em generalized} $B$-structure in the $L_J$-bundle
$\Ind_T^{L_J}\CK_T$. Thus we obtain a morphism
\[\varPi^M_J\colon \oM^\der(\Ind_T^G\CK_T,\CL_T)\to\ol{M}{}_J^\der(\Ind_T^{L_J}\CK_T,\CL_T)\]
to the Drinfeld closure of $\oM_J^\der(\Ind_T^{L_J}\CK_T,\CL_T)$. The latter closure
is defined as the open substack in the homotopy fibre product
of $[\Ind_T^{L_J}\CK_T]\times[\CL_T]$ and $\ol\Bun_{B_J}$ over
$\Bun_{L_J}\times\Bun_T$ (cf.~(\ref{homotopy})) given by the condition that the generalized
$B_J$-structure is generically transversal to the tautological $U_{J-}$-structure in
$\Ind_T^{L_J}\CK_T$.

It remains to check that $\varPi^M_J$ is dominant, i.e.\ gives rise to the desired rational
morphism from $\oM^\der(\Ind_T^G\CK_T,\CL_T)$ to $\oM_J^\der(\Ind_T^{L_J}\CK_T,\CL_T)$.
This follows from~Lemma~\ref{two diagrams}(b) below, i.e.\ compatibility of $\varPi^M_J$ with
$\varPi^Z_J$, along with the dominance property of $\varPi^Z_J$.

Comparing with construction of Poisson structure $\{,\}_{FO}$ in~\S\ref{fo},\ref{tangent} we
see that $\varPi^M_J$ is a Poisson morphism. Indeed, we have to check that for a smooth point
$\varphi\in\oM(\Ind_T^G\CK_T,\CL_T)$ such that $\varPi^M_J$ is regular at $\varphi$,
the Poisson bivector $P_J\colon T^*_{\varPi^M_J\varphi}\oM_J(\Ind_T^{L_J}\CK_T,\CL_T)\to
T_{\varPi^M_J\varphi}\oM_J(\Ind_T^{L_J}\CK_T,\CL_T)$ equals the composition
$d\varPi^M_J\circ P_I\circ(d\varPi^M_J)^*$. To this end note that we have a natural projection
of vector bundles on $E$:
\[\Xi\colon\fg_\varphi\twoheadrightarrow(\fl_J)_\varphi,\]
and the condition that $\varPi^M_J$ is regular at $\varphi$ guarantees that
$\Xi(\fb_\varphi)=(\fb_{L_J})_{\varPi^M_J\varphi}$ and $\Xi(\fn_\varphi)=(\fn_{L_J})_{\varPi^M_J\varphi}$.
Moreover, under the identification~(\ref{total}), the differential
\[d\varPi^M_J\colon T_\varphi\oM(\Ind_T^G\CK_T,\CL_T)\to
T_{\varPi^M_J\varphi}\oM_J(\Ind_T^{L_J}\CK_T,\CL_T)\]
is induced by $\Xi$. Furthermore, under the identification~(\ref{total'}), $d\varPi^M_J$ is
also induced by $\Xi$, provided we identify $\fb_\varphi^\svee$ with $\fg_\varphi/\fn_\varphi$.
The Poisson property of $\varPi^M_J$ follows.

\begin{lem}
  \label{two diagrams}
  The following diagrams commute:

  \textup{(a)}
  $\begin{CD}
^{\vphantom{\alpha}}_\CalD\uoZ^\alpha_\CK @>>{\varPi^Z_J}> ^{\vphantom{\alpha}}_{\CalD_J}\uoZ^{\alpha_J}_{\CK_J}(L'_J)\\
  @VV{\wr}V @V{\wr}VV \\
  ^C_\CalD\uoZ^\alpha_{\CK'} @>{\varPi^C_J}>> ^{\hphantom{D}C}_{\CalD_J}\!\uoZ^{\alpha_J}_{\CK'_J}(L'_J),
  \end{CD}$

  \textup{(b)}
  $\begin{CD}
^{\vphantom{\alpha}}_\CalD\uoZ^\alpha_\CK @>>{\varPi^Z_J}> ^{\vphantom{\alpha}}_{\CalD_J}\uoZ^{\alpha_J}_{\CK_J}(L'_J)\\
  @VV{\wr}V @V{\wr}VV \\
  \oM(\Ind_T^G\CK_T,\CL_T) @>{\varPi^M_J}>> \oM_J(\Ind_T^{L_J}\CK_T,\CL_T).
  \end{CD}$
\end{lem}

\begin{proof}
  (a) follows from the fact that the isomorphism of Theorem~\ref{reductio} is compatible with
  factorization.

  (b) follows from the definition of factorization isomorphism, cf.\ the proof
  of~\cite[Proposition~3.2]{bdf}.
\end{proof}

\subsection{Proof of Theorem~\ref{myksas} for $G=\SL(2)$}
\label{prf sl2}
The only vertex of the Dynkin diagram is denoted by $i$. The corresponding simple root and
fundamental weight are denoted simply by $\alpha^\svee$ and $\omega^\svee$.
A regular $T$-bundle is a line bundle $\CK=\CK^{\omega^\svee}$ such that $\CK_i=\CK^{-2}=\CK^{-\alpha^\svee}$
is nontrivial. We fix a line bundle $\CL$ of degree $-a$, and we set $\CalD=\CL^{-1}\CK$.
A point $\varphi$ of $^{\vphantom{\alpha}}_\CalD\uoZ^{a}_\CK\cong\oM(\CK\oplus\CK^{-1},\CL)$ is
represented by a short exact sequence
\[0\to\CL\xrightarrow{(s,t)}\CK\oplus\CK^{-1}\xrightarrow{(-t,s)}\CL^{-1}\to0.\]
The associated adjoint vector bundle has a 2-step filtration
$0\subset\fn_\varphi\subset\fb_\varphi\subset\fg_\varphi$ with associated graded quotients
$\fn_\varphi\simeq\CL^2,\ \fb_\varphi/\fn_\varphi\simeq\CE nd(\CL)\simeq\CO_E,\
\fg_\varphi/\fb_\varphi\simeq\CL^{-2}$. It gives rise to the connecting homomorphisms
\[\delta\colon H^0(E,\CL^{-2})\to H^1(E,\CE nd(\CL)),\ H^0(E,\CE nd(\CL))\to H^1(E,\CL^2).\]
If $\varphi$ is a smooth point of $\oM(\CK\oplus\CK^{-1},\CL)$, then the tangent space is
\[T_\varphi\oM(\CK\oplus\CK^{-1},\CL)=
\on{Ker}\big(H^0(E,\CL^{-2})\to H^1(E,\CE nd(\CL))\big)/\BC s\circ t,\]
and dually the cotangent space is \[T_\varphi^*\oM(\CK\oplus\CK^{-1},\CL)=
\big((\BC s\circ t)^\perp\subset H^1(E,\CL^2)\big)/H^0(E,\CE nd(\CL)).\]
Also, we have a splitting
\begin{multline}
  \label{quatro}
  H^0(E,\CL^{-1}\CK)/\BC s\oplus H^0(E,\CL^{-1}\CK^{-1})/\BC t\iso
  T_\varphi\oM(\CK\oplus\CK^{-1},\CL),\\
  (\varpi,\varrho)\mapsto s\circ\varrho-t\circ\varpi.
\end{multline}
Explicitly, given $\varpi\in 
\Hom(\CL,\CK)$ and $\varrho\in 
\Hom(\CL,\CK^{-1})$, we construct an infinitesimal deformation
$(s_\varepsilon,t_\varepsilon)$ of $(s,t)\colon\CL\to\CK\oplus\CK^{-1}$ over
$\BC[\varepsilon]/(\varepsilon^2)$ as
\[s_\varepsilon=s+\varpi\varepsilon\colon\CL\to\CK,\
t_\varepsilon=t+\varrho\varepsilon\colon\CL\to\CK^{-1}.\]

\subsubsection{Coordinates}
\label{coor}
Let $D$ be the zero divisor of $s\in\Hom_E(\CL,\CK)$; we assume that $D$ is multiplicity free and we
choose a numbering $w_1,\ldots,w_a$ of its points.
The functions $w_1,\ldots,w_{a-1}\colon\oM(\CK\oplus\CK^{-1},\CL)\to E$ are defined \'etale locally
(and $w_a$ is determined by $w_1,\ldots,w_{a-1}$ since the sum
$\sum_{r=1}^aw_r\in E$ is fixed).

We also fix a section $u$ of $\CL^{-1}\CK^{-1}$ with zeros disjoint from $D$ and define the
homogeneous functions $y_r:=\frac{t}{u}|_{w_r}\colon\oM(\CK\oplus\CK^{-1},\CL)\to\BC^\times$. 
Since the reduced zastava is a quotient by the
$\BG_m$-action, only the ratios of $y$-coordinates are well defined (\'etale locally).
Alternatively, we can normalize $t$ in such a way that $\sum_{r=1}^a\frac{t}{u}|_{w_r}=1$, and
consider the resulting functions $y_1,\ldots,y_{a-1}$ together with $w_1,\ldots,w_{a-1}$ as \'etale
local coordinates on $\oM(\CK\oplus\CK^{-1},\CL)$.
The above normalization of $t$ is possible (the sum does not vanish identically) since
$\CL^{-1}\CK^{-1}$ is not isomorphic to $\CL^{-1}\CK$, hence the restriction map
$H^0(E,\CL^{-1}\CK^{-1})\to\BC^a,\ t\mapsto t|_D$, is an isomorphism.

The tangent space to $E^{(a)}$ at $D$ can be identified with
$H^0(D,\CO_E(D)|_D)=H^1(E,\CO_E\to\CO_E(D))$ (the complex $\CO_E\to\CO_E(D)$ lives in degrees
$0,1$). The tangent vector corresponding to the infinitesimal deformation $D_\varepsilon$ equal
to the zero divisor of the section $s_\varepsilon$ (considered right after~(\ref{quatro}))
is given by the 1-cocycle $(0,\frac{\varpi}{s})$. In other words, the corresponding element
of $H^0(D,\CO_E(D)|_D)$ is the polar part of $-\frac{\varpi}{s}$. Note that this is the same
as the polar part of $\frac{(s\circ\varrho-t\circ\varpi)}{st}$.
Thus the tangent map to the factorization morphism $\oM(\CK\oplus\CK^{-1},\CL)\to E^{(a)},\
(s,t)\mapsto D$, sends $s\circ\varrho-t\circ\varpi$ to 
$\frac{(s\circ\varrho-t\circ\varpi)|_D}{s't|_D}$, where $s'$ is the nowhere vanishing section
of $\CL^{-1}\CK(-D)$ corresponding to $s$. It means that the image of the tangent vector
$\partial/\partial w_r$ under the composition
\[T_\varphi\oM(\CK\oplus\CK^{-1},\CL)\to H^0(E,\CL^{-1}\CK)/\BC s\to
H^0(D,\CL^{-1}\CK|_D)\xrightarrow{1/s'|_D}H^0(D,\CO_E(D)|_D)\]
is the principal part of the unique (up to an additive constant) rational function on $E$ that
has a simple pole with residue 1 at $w_r$ and a simple pole with residue $-1$ at $w_a$ and no
other poles (we use the trivialization of $\bomega_E$, see~Remark~\ref{triv can}).

Dually, $dw_r$ is the image of $(1|_{w_r}-1|_{w_a})$ under the composition
\begin{multline*}H^0(D,\CO_E|D)\xrightarrow{1/s'|_D}H^0(D,\CL\CK^{-1}(D)|_D)\xrightarrow{1/t|_D}
  H^0(D,\CL^2(D)|_D)/H^0(E,\CO_E)\\
  \to H^1(E,\CL^2)/H^0(E,\CO_E),
\end{multline*}
where the last arrow is the connecting homomorphism for the short exact sequence
\[0\to\CL^2\to\CL^2(D)\to\CL^2(D)|_D\to0.\]

The image of the tangent vector $\partial/\partial y_r$ under the composition
\begin{multline}
  \label{cinquo}
  T_\varphi\oM(\CK\oplus\CK^{-1},\CL)\to H^0(E,\CL^{-1}\CK^{-1})/\BC t\to
  H^0(D,\CL^{-1}\CK^{-1}|_D)/\BC t|_D\\
  \xrightarrow{1/u|_D}H^0(D,\CO_E|_D)/\BC\frac{t}{u}\Big|_D
\end{multline}
is $1|_{w_r}-1|_{w_a}\pmod{\frac{t}{u}|_D}$. Indeed, at a point of $\oM(\CK\oplus\CK^{-1},\CL)$
given by a pair of maps $(s,t):\CL\to \CK \oplus \CK^{-1}$, the tangent vector
$\partial/\partial y_r$ is represented by the linear term of the infinitesimal
deformation $(s_\varepsilon,t_\varepsilon)\colon\CL\to \CK \oplus \CK^{-1}$, where $s_\varepsilon=s,\
t_\varepsilon(w_i)=t(w_i)$ for $i\not=r,a$, while $t_\varepsilon(w_r)=t(w_r)+\varepsilon u(w_r)$, and
$t_\varepsilon(w_a)=t(w_a)-\varepsilon u(w_a)$. Restricting this linear term to $D$
and dividing by $u|_D$, we obtain $1|_{w_r}-1|_{w_a}$. 

\subsubsection{Computation of the Feigin-Odesskii bracket}
According to the last paragraph of~\S\ref{tangent}, the Feigin-Odesskii Poisson bracket is defined
using the differential $d_2$ of the second page of the hypercohomology spectral sequence for
the complex
\[\CL^2\xrightarrow{(-t^2,st,s^2)}\CK^{-2}\oplus\CO_E\oplus\CK^2\xrightarrow{(s^2,2st,-t^2)}\CL^{-2}.\]
Consider the commutative diagram
\begin{equation}
  \label{septo}
  \begin{CD}
    \CL^2 @>(-t^2,st,s^2)>> \CK^{-2}\oplus\CO_E\oplus\CK^2 @>(s^2,2st,-t^2)>> \CL^{-2}\\
    @VVV @VVV @| \\
    \CL^2(D) @>(-t^2,s't,s's)>> \CK^{-2}(D)\oplus\CO_E\oplus\CK^2 @>(s's,2st,-t^2)>> \CL^{-2}\\
    @VVV @VV(-(t|_D)^{-2},0,0)V @.\\
    \CL^2(D)|_D @= \CL^2(D)|_D. @.
  \end{CD}
\end{equation}
We set $\CH:=\fg_\varphi=\CK^{-2}\oplus\CO_E\oplus\CK^2,\ \CH':=\CK^{-2}(D)\oplus\CO_E\oplus\CK^2$.
One can show by the diagram chase that the following diagram commutes:
\begin{equation}
  \label{octo}
  \begin{CD}
    \on{Ker}\big(H^1(E,\CL^2)\to H^1(E,\CH)\big) @>{d_2}>> H^0(E,\CL^{-2})/H^0(E,\CH)\\
    @AAA @A{s's}AA \\
    \on{Ker}\big(H^0(D,\CL^2(D)|_D)\to H^1(E,\CH)\big) @>>> H^0(E,\CH')/H^0(E,\CH).
  \end{CD}
\end{equation}

Recall that the Hamiltonian vector field $h_r$ of $dw_r$ is the image of $1|_{w_r}-1|_{w_a}$
under the composition
\[\begin{CD}
@. \on{Ker}\big(H^1(E,\CL^2)\to H^1(E,\CH)\big) @>{d_2}>> H^0(E,\CL^{-2})/H^0(E,\CH)\\
@. @AAA @. \\
H^0(D,\CO_E|_D) @>{1/s't|_D}>> \on{Ker}\big(H^0(D,\CL^2(D)|_D)\to H^1(E,\CH)\big). @.
\end{CD}\]
Due to commutativity of~(\ref{octo}), we can replace this composition with
\[\begin{CD}
@. @. H^0(E,\CL^{-2})/H^0(E,\CH)\\
@. @. @A{s's}AA \\
H^0(D,\CO_E|_D) @>{1/s't|_D}>> \on{Ker}\big(H^0(D,\CL^2(D)|_D)\to H^1(E,\CH)\big) @>>>
H^0(E,\CH')/H^0(E,\CH).
\end{CD}\]
It follows that $h_r$ gives a section of $\CL^{-2}$ divisible by $s$, say $h_r=s\circ\varrho$.
This means that in the splitting~(\ref{quatro}), $h_r$ lies in the second summand. In particular,
\[\{w_r,w_{r'}\}_{FO}=0\ \on{for}\ \on{any}\ r,r'.\] Furthermore, one can see from~(\ref{septo}) that
$\varrho$ is the section of $\CL^{-1}\CK^{-1}$ taking value $t|_{w_r}$ at $w_r$ and
$-t|_{w_a}$ at $w_a$. Composing this claim with~(\ref{cinquo}) we get
\[\{y_r,x_{r'}\}_{FO}=0\ \on{for}\ r\ne r',\ \on{and}\ \{y_r,w_r\}_{FO}=y_r.\]
The remaining brackets
\[\{y_r,y_{r'}\}_{FO}=0.\]
Indeed, we have proved that $d_2$ sends the first
summand of the splitting~(\ref{quatro}) to the second one in the dual splitting. But the splitting
is symmetric with respect to swapping the roles of $s$ and $t$ (and replacing the divisor $D$
with the zero divisor of $t$). This shows that $d_2$ sends the second summand to the first one,
so the brackets of $y$-coordinates vanish.

\subsubsection{Comparison with the reduced transversal Hilbert scheme}
According to~Proposition~\ref{hilb vs zas}(c), the reduced zastava is isomorphic to the
reduced transversal Hilbert scheme $^{\vphantom{\alpha}}_\CalD\ul\Hilb{}_\tr^a(\oS_{\CK'})$, where
$\CK':=\CK'_i=\CK^{-\alpha^\svee}\otimes\CalD$. The symplectic structure $\omega_{\CK'}$ on the surface
$\oS_{\CK'}$ defined in~\S\ref{poisson} gives rise to a symplectic structure on the
transversal Hilbert scheme and on its reduction. The corresponding bracket is denoted
$\{,\}_{\CK'}$. On the other hand, according to~(\ref{hilb vs coul}), the (reduced) transversal
Hilbert scheme is nothing but the (reduced) Coulomb open zastava, and this identification respects
the Poisson brackets.

To compare $\{,\}_{\CK'}$ with $\{,\}_{FO}$ we match the local coordinates. We choose a local
trivialization
$\eta$ of $\CK'=\CK'_i=\CL^{-1}\CK^{-1}$. We denote by $p\colon\oS_{\CK'}\to E$ the projection.
The corresponding local coordinate $z$ on $\oS_{\CK'}$ is $z=\eta_{\on{can}}/p^*\eta$, where
$\eta_{\on{can}}$ is the tautological section of $p^*\CK'$. On the \'etale open
$(\oS_{\CK'})^a\setminus\Delta\to\Hilb^a(\oS_{\CK'})$ we have the induced local coordinates
$w_1,\ldots,w_a,z_1,\ldots,z_a$. We have $\{z_r,w_r\}_{\CK'}=z_r$, and all the other brackets vanish.

On the reduced transversal Hilbert scheme $^{\vphantom{\alpha}}_\CalD\ul\Hilb{}_\tr^a(\oS_{\CK'})$
we have the constraint that $w_1+\ldots+w_a$ is a fixed point of $E$. These coordinates clearly
match the same named coordinates on the reduced zastava of the previous subsections.

Now recall that the identification of reduced zastava with the reduced Hilbert scheme
in~Proposition~\ref{hilb vs zas}(c) is obtained in the following way. Given a point of reduced
zastava represented by $\varphi=(s,t)$ we fix an isomorphism $\varsigma\colon\CO_E(D)\iso\CalD$
and consider the image of $D\times\{1\}\subset D\times\BG_m$ under the isomorphism
\[(\varsigma\cdot t/s)|_D\colon D\times\BG_m\iso\CK'|_D\]
considered up to $\BG_m$-action ($\CK'|_D$ stands for the total space of the line bundle).
Here we view $t/s$ as a section of $\CK^{-2}(D)$. In fact, we can take $\varsigma=s$, so that
our point corresponds to the trivialization of $\CK'|_D$ given by
$t\in H^0(E,\CL^{-1}\CK^{-1})=H^0(E,\CK')$.

But if we use a local section $u\in H^0(E,\CL^{-1}\CK^{-1})$ as in~\S\ref{coor} to define
the local trivialization $\eta$ above, the value of the above coordinate $z_r$ at $\varphi=(s,t)$
equals $t/u(w_r)$. This coincides with the value of the coordinate $y_r$ of~\S\ref{coor} at $\varphi$.
In other words, the identification of reduced zastava with reduced transversal Hilbert scheme
takes the $(w,y)$-coordinates to $(w,z)$-coordinates, and the bracket $\{,\}_{FO}$ to $\{,\}_{\CK'}$.

This completes the proof of Theorem~\ref{myksas} for $G=\SL(2)$.

\subsection{Proof of Theorem~\ref{myksas} for $G=\SL(3)$}
\label{prf sl3}
The vertices of the Dynkin diagram are denoted by $i,j$. A regular $T$-bundle $\CK_T$ is specified by
the line bundles $\CK^{\omega_i^\svee}$ and $\CK^{\omega_j^\svee}$ such that
$\CK^{\alpha_i^\svee}=\CK^{2\omega_i^\svee}\CK^{-\omega_j^\svee},\
\CK^{\alpha_j^\svee}=\CK^{2\omega_j^\svee}\CK^{-\omega_i^\svee},\
\CK^{\alpha_i^\svee+\alpha_j^\svee}=\CK^{\omega_i^\svee}\CK^{\omega_j^\svee}$ are all nontrivial.
We fix line bundles $\CL_i=\CL^{\omega_i^\svee}$ and $\CL_j=\CL^{\omega_j^\svee}$ of degrees $-a_i,-a_j$,
we set $\alpha=a_i\alpha_i+a_j\alpha_j$ and $\CalD_i=\CL_i^{-1}\CK^{\omega_i^\svee},\
\CalD_j=\CL_j^{-1}\CK^{\omega_j^\svee}$.
We set $\CV=\CV^{\omega_i^\svee}=\CK^{\omega_i^\svee}\oplus\CK^{\omega_j^\svee-\omega_i^\svee}\oplus\CK^{-\omega_j^\svee}
=:\CK_1\oplus\CK_2\oplus\CK_3$. A point $\varphi$ of
$^{\vphantom{\alpha}}_\CalD\uoZ^{\alpha}_\CK\cong\oM(\Ind_T^G\CK_T,\CL_T)$ is represented by a complex
\begin{equation}
  \label{LVL}
  \CL_i\xrightarrow{(s_1,s_2,s_3)}\CV\xrightarrow{(t_1,t_2,t_3)}\CL_j^{-1}.
\end{equation}
Here $s_c\in H^0(E,\CL_i^{-1}\CK_c)$ and $t_d\in H^0(E,\CK_d^{-1}\CL_j^{-1})$ have no common zeros
and satisfy the equation $s_1t_1+s_2t_2+s_3t_3=0\in H^0(E,\CL_i^{-1}\CL_j^{-1})$.
The associated adjoint vector bundle
\begin{equation}
  \label{ad decomp}
  \fg_\varphi=\CO_E^{\oplus2}\oplus\bigoplus_{1\leq c\ne d\leq 3}\CK_c\CK_d^{-1}
\end{equation}
(traceless endomorphisms
of $\CV$) has a 2-step filtration $0\subset\fn_\varphi\subset\fb_\varphi\subset\fg_\varphi$,
and the Poisson bivector $\{,\}_{FO}$ comes from the differential $d_2$ of the second page of the
hypercohomology spectral sequence for the complex
\begin{equation}
  \label{9}
  \fn_\varphi\to\fg_\varphi\to\fg_\varphi/\fb_\varphi.
\end{equation}

\subsubsection{Coordinates}
\label{coordi}
We use the morphisms $\varPi^M_{\{i\}}$ and $\varPi^M_{\{j\}}$ of~\S\ref{compa}. The targets are
the Feigin-Odesskii moduli spaces of type $A_1$ studied in~\S\ref{prf sl2}. In particular,
the coordinates on them are defined in~\S\ref{coor}, and we define the coordinates on
$^{\vphantom{\alpha}}_\CalD\uoZ^{\alpha}_\CK\cong\oM(\Ind_T^G\CK_T,\CL_T)$ as the pullbacks of the
coordinates of~\S\ref{coor}. Thus we get the \'etale local coordinates
$w_{i,1},\ldots,w_{i,a_i}$ (subject to the condition that their sum in $E$ is fixed),
$w_{j,1},\ldots,w_{j,a_j}$ (also subject to the condition that their sum in $E$ is fixed),
$y_{i,1},\ldots,y_{i,a_i}$ (homogeneous, i.e.\ only the ratios are well defined),
$y_{j,1},\ldots,y_{j,a_j}$ (also homogeneous).

More explicitly, $w_{i,1},\ldots,w_{i,a_i}$ are the zeros of $s_1$, while
$w_{j,1},\ldots,w_{j,a_j}$ are the zeros of $t_3$. We impose the genericity assumption that all
the points $w_{i,1},\ldots,w_{i,a_i},w_{j,1},\ldots,w_{j,a_j}$ are distinct.
Furthermore, we choose sections
$u_i\in H^0(E,\CL_i^{-1}\CK_2)$ and $u_j\in H^0(E,\CK_2^{-1}\CL_j^{-1})$. We consider the open
substack of $\oM(\Ind_T^G\CK_T,\CL_T)$ specified by the condition that all the $w$'s are
distinct and also distinct from the zeros of $u_i$ and $u_j$.
Finally, $y_{i,r}=\frac{s_2}{u_i}|_{w_{i,r}},\ y_{j,r}=\frac{t_2}{u_j}|_{w_{j,r}}$.

The only nonvanishing Feigin-Odesskii brackets of $i$-coordinates (resp.\ $j$-coordinates) are
$\{y_{i,r},x_{i,r}\}_{FO}=y_{i,r}$ (resp.\ $\{y_{j,r},x_{j,r}\}_{FO}=y_{j,r}$) since
$\varPi^M_{\{i\}}$ (resp.\ $\varPi^M_{\{j\}}$) is Poisson. It remains to compute the brackets of
$i$-coordinates with $j$-coordinates. This computation will occupy the rest of this Section.

\subsubsection{Brackets with $w$-coordinates}
We extend the complex~(\ref{9}) to a diagram
\begin{multline*}
\CH om((\CK_2^{-1}\oplus\CK_3^{-1})/\CL_j,\CL_j)\to\CH om(\CV^\vee/\CL_j,\CL_j)\to\fn_\varphi\to\fg_\varphi\\
  \to\fg_\varphi/\fb_\varphi\to\CH om(\CL_i,\CV/\CL_i)\to\CH om(\CL_i,(\CK_1\oplus\CK_2)/\CL_i).
\end{multline*}
Note that we have isomorphisms of line bundles
$(\CK_2^{-1}\oplus\CK_3^{-1})/\CL_j\simeq\CL_j^{-1}\CK_2^{-1}\CK_3^{-1}$ and
$(\CK_1\oplus\CK_2)/\CL_i\simeq\CL_i^{-1}\CK_1\CK_2$. Hence composing the first three and the
last three arrows in the above diagram we obtain a complex
\[\CL_j^2\CK_2\CK_3\xrightarrow{A}\fg_\varphi\xrightarrow{B}\CL_i^{-2}\CK_1\CK_2.\]
With respect to the decomposition~(\ref{ad decomp})
\[\fg_\varphi\subset\begin{matrix}
&\CK_1^{-1}\CK_1&\oplus&\CK_2^{-1}\CK_1&\oplus&\CK_3^{-1}\CK_1 \\
\oplus&\CK_1^{-1}\CK_2&\oplus&\CK_2^{-1}\CK_2&\oplus&\CK_3^{-1}\CK_2 \\
\oplus&\CK_1^{-1}\CK_3&\oplus&\CK_2^{-1}\CK_3&\oplus&\CK_3^{-1}\CK_3,
\end{matrix}\]
the matrix elements of $A$ (resp.\ $B$) are
\[\begin{matrix}
\begin{pmatrix}
0&0&0 \\
-t_1t_3 & -t_2t_3 &-t_3t_3 \\
t_1t_2 & t_2t_2 & t_3t_2
\end{pmatrix} 
&{\rm resp.}&
\begin{pmatrix}
-s_1s_2&-s_2s_2&-s_3s_2 \\
s_1s_1&s_1s_2& s_1s_3 \\
0&0&0
\end{pmatrix} 
\end{matrix}\]
(notation of~(\ref{LVL})).

Hence the first and the third rows do not contribute to the differential $d_2$ of the second
page of the hypercohomology spectral sequence, and this differential equals the one for a
simpler complex
\begin{equation}
  \label{twelve}
  \CL_j^2\CK_2\CK_3\xrightarrow{(-t_1t_3,-t_2t_3,-t_3t_3)}\CK_1^{-1}\CK_2\oplus\CO_E\oplus\CK_3^{-1}\CK_2
  \xrightarrow{(s_1s_1,s_1s_2,s_1s_3)}\CL_i^{-2}\CK_1\CK_2.
  \end{equation}
In particular, the image of $d_2$ is always divisible by $s_1$.

This implies $\{w_{i,r},w_{j,r'}\}_{FO}=\{w_{i,r},y_{j,r'}\}_{FO}=\{y_{i,r},w_{j,r'}\}_{FO}=0$ for any $r,r'$.

\subsubsection{Type $A_1$ revisited}
In order to compute $\{y_{i,r},y_{j,r'}\}_{FO}$, we need some preparation on the tangent bundle
of the Levi Feigin-Odesskii moduli space $\oM_{\{j\}}(\Ind_T^{L_{\{j\}}}\CK_T,\CL_T)$.

Recall from~\S\ref{prf sl2} that
\[T_{\varPi^M_{\{j\}}\varphi}\oM_{\{j\}}(\Ind_T^{L_{\{j\}}}\CK_T,\CL_T)=
\on{Ker}\big(H^0(E,\CL_j^{-2}\CK_2^{-1}\CK_3^{-1})\to H^1(E,\CE nd(\CL_j^{-1}))\big)/\BC t_2t_3,\]
\[T^*_{\varPi^M_{\{j\}}\varphi}\oM_{\{j\}}(\Ind_T^{L_{\{j\}}}\CK_T,\CL_T)=
\big( (\BC t_2t_3)^\perp \subset H^1(E,\CL_j^{2}\CK_2\CK_3)\big)/  H^0(E,\CE nd(\CL_j^{-1})).\]

Splitting \eqref{quatro} can be rewritten as follows:
\begin{equation}
  \label{eq:tangent_splitting}
  \begin{array}{ccc}
   H^0(E,\CL_j^{-1}\CK_2^{-1})/\BC t_2\oplus H^0(E,\CL_j^{-1}\CK_3^{-1})/\BC t_3&\iso&
  T_{\varPi^M_{\{j\}}\varphi}\oM_{\{j\}}(\Ind_T^{L_{\{j\}}}\CK_T,\CL_T), \\
    (\varpi,\varrho)&\longmapsto& t_3\varrho-t_2\varpi.
  \end{array}
\end{equation}

Applying Serre duality to the splitting~\eqref{eq:tangent_splitting} of
$T_{\varPi^M_{\{j\}}\varphi}\oM_{\{j\}}(\Ind_T^{L_{\{j\}}}\CK_T,\CL_T)$, we obtain the following splitting of
$T^*_{\varPi^M_{\{j\}}\varphi}\oM_{\{j\}}(\Ind_T^{L_{\{j\}}}\CK_T,\CL_T)$:
\begin{equation}\label{eq:cotangent_splitting_1}
\hspace{-1cm}
\begin{array}{ccc}
	T^*_{\varPi^M_{\{j\}}\varphi}\oM_{\{j\}}(\Ind_T^{L_{\{j\}}}\CK_T,\CL_T) & \iso & \Big( (\BC t_2)^\perp \subset H^1(E,\CL_j \CK_2)) \Big) \oplus \Big( (\BC t_3)^\perp \subset H^1(E,\CL_j \CK_3)) \Big), \\
	\upsilon &\longmapsto & (\upsilon t_3, -\upsilon t_2).
\end{array}
\end{equation}

It will be useful to rewrite the first summand of the splitting~\eqref{eq:cotangent_splitting_1} as
\[\on{Ker}\Big(H^0\left(D_{t_3},\CL_j\CK_2(D_{t_3})|_{D_{t_3}}\right)
\xrightarrow{\on{Res}(t_2|_{D_{t_3}}\cdot?)}H^1(E,\CO_E)\Big).\]
This is done by dualizing the first summand of \eqref{eq:tangent_splitting}, using the pairing
between $H^0(E,\CL_j^{-1} \CK_2^{-1})$ and $H^0(D_{t_3},\CL_j \CK_2(D_{t_3})|_{D_{t_3}})$ given by the
sum of residues of the product 
(as always, we use the trivialization of $\bomega_E$ in~Remark~\ref{triv can}).
The identification
\begin{multline*}
\hspace{-0.5cm}
  \on{Ker}\Big(H^0\left(D_{t_3},\CL_j\CK_2(D_{t_3})|_{D_{t_3}}\right)
  \xrightarrow{\on{Res}(t_2|_{D_{t_3}}\cdot?)}H^1(E,\CO_E)\Big)\iso
  \Big( (\BC t_2)^\perp \subset H^1(E,\CL_j\CK_2)\Big)
\end{multline*}
is induced by the connecting homomorphism for the short exact sequence
\begin{equation}
  \label{thirteen}
  0\to\CL_j\CK_2\xrightarrow{t_3}\CL_j\CK_2(D_{t_3})\to\CL_j\CK_2(D_{t_3})|_{D_{t_3}}\to0.
  \end{equation}

\subsubsection{Brackets of $y$-coordinates: \v{C}ech cocycles}
In order to compute $\{y_{i,r},y_{j,r'}\}_{FO}$, we need to compute the composition
\begin{multline*}
  \on{Ker}\Big(H^0\left(D_{t_3},\CO_E(D_{t_3})|_{D_{t_3}}\right)
    \xrightarrow{\on{Res}(\frac{t_2}{u_j}|_{D_{t_3}}\cdot?)}H^1(E,\CO_E)\Big)\\
    \xrightarrow{\frac{1}{u_j}|_{D_{t_3}}}
   \on{Ker}\Big(H^0\left(D_{t_3},\CL_j\CK_2(D_{t_3})|_{D_{t_3}}\right)
  \xrightarrow{\on{Res}(t_2|_{D_{t_3}}\cdot?)}H^1(E,\CO_E)\Big)\\
  \to\on{Ker}\big(H^1(E,\CL_j^2\CK_2\CK_3)\to H^1(E,\CO_E)\big)\xrightarrow{d_2}
   H^0(E,\CL_i^{-2}\CK_1\CK_2)/\BC s_1s_2\\
  \to H^0(E,\CL_i^{-1}\CK_2)/\BC s_2\xrightarrow{\frac{1}{u_i}|_{D_{s_1}}}
  H^0(D_{s_1},\CO_{D_{s_1}})\big/\BC\frac{s_2}{u_j}\Big|_{D_{s_1}},
   \end{multline*}
where $u_i,u_j$ were defined in~\S\ref{coordi}, and $d_2$ comes from~(\ref{twelve}).

We rewrite the above composition as follows:
\begin{multline}
  \label{fourteen}
  \on{Ker}\Big(H^0\left(D_{t_3},\CO_E(D_{t_3})|_{D_{t_3}}\right)
    \xrightarrow{\on{Res}(\frac{t_2}{u_j}|_{D_{t_3}}\cdot?)}H^1(E,\CO_E)\Big)\\
    \xrightarrow{\frac{1}{u_j}|_{D_{t_3}}}
   \on{Ker}\Big(H^0\left(D_{t_3},\CL_j\CK_2(D_{t_3})|_{D_{t_3}}\right)
  \xrightarrow{\on{Res}(t_2|_{D_{t_3}}\cdot?)}H^1(E,\CO_E)\Big)\\
  \to\on{Ker}\big(H^1(E,\CL_j\CK_2)\to H^1(E,\CO_E)\big)\xrightarrow{d_2}
   H^0(E,\CL_i^{-2}\CK_1\CK_2)/\BC s_1s_2\\
  \to H^0(E,\CL_i^{-1}\CK_2)/\BC s_2\xrightarrow{\frac{1}{u_i}|_{D_{s_1}}}
  H^0(D_{s_1},\CO_{D_{s_1}})\big/\BC\frac{s_2}{u_j}\Big|_{D_{s_1}},
\end{multline}
where the second arrow is the connecting homomorphism coming from~(\ref{thirteen}),
and $d_2$ is the differential in the hypercohomology spectral sequence of the complex
\begin{equation}
  \label{fifteen}
  \CL_j\CK_2\xrightarrow{(-t_1,-t_2,-t_3)}\CK_1^{-1}\CK_2\oplus\CO_E\oplus\CK_3^{-1}\CK_2
  \xrightarrow{(s_1s_1,s_1s_2,s_1s_3)}\CL_i^{-2}\CK_1\CK_2.
\end{equation}

To perform computations with the first cohomology we introduce a \v{C}ech cover of $E$ by
two opens $U_{t_2}:=E\setminus D_{t_2}$ and $U_{t_3}:=E\setminus D_{t_3}$. We represent
$dy_{j,r},\ 1\leq r<a_j$, as
the element of $H^0(D_{t_3},\CO_E(D_{t_3})|_{D_{t_3}})$ given by the principal part of
\[\frac{1}{x-w_{j,r}}\Big|_{w_{j,r}}-\left(\frac{t_2}{u_j}\Big|_{w_{j,r}}\right)
\left(\frac{t_2}{u_j}\Big|_{w_{j,a_j}}\right)^{-1}\frac{1}{x-w_{j,a_j}}\Big|_{w_{j,a_j}}.\]
Then the corresponding 1-cocycle in $H^1(E,\CL_j\CK_2)$ is given by a section
$f\in H^0(U_{t_2}\cap U_{t_3},\CL_j\CK_2)$ having simple poles at points of $D_{t_3}$ (and perhaps
some other poles at $D_{t_2}$ that we do not care about) such that the principal part of $f$
at $w_{j,r}$ is $\frac{1}{u_j}|_{w_{j,r}}\frac{1}{x-w_{j,r}}$ and the principal part of $f$
  at $w_{j,a_j}$ is $-\left(\frac{t_2}{u_j}\Big|_{w_{j,r}}\right)
  \left(t_2\Big|_{w_{j,a_j}}\right)^{-1}\frac{1}{x-w_{j,a_j}}$, while the principal parts
  of $f$ at $w_{j,r'}$ for $r\ne r'\ne a_j$ vanish. Furthermore, we apply the left morphism
  in~(\ref{fifteen}) to the above 1-cocycle to obtain a 1-cocycle
  $(g_1,g_2,g_3)\in H^1(E,\CK_1^{-1}\CK_2\oplus\CO_E\oplus\CK_3^{-1}\CK_2)$, where
  $g_1=-t_1f,\ g_2=-t_2f,\ g_3=-t_3f$. Then $g_3$ has no poles at $D_{t_3}$, and
  $g_2$ has the principal part $\frac{-t_2}{u_j}|_{w_{j,r}}\frac{1}{x-w_{j,r}}$ at $w_{j,r}$,
  and the principal part $\left(\frac{t_2}{u_j}\Big|_{w_{j,r}}\right)\frac{1}{x-w_{j,a_j}}$
  at $w_{j,a_j}$,
  while the principal parts of $g_2$ at $w_{j,r'}$ for $r\ne r'\ne a_j$ vanish.

\subsubsection{Brackets of $y$-coordinates: Weierstra\ss\ $\zeta$-function} 
Below we write formulas in terms of the Weierstra\ss\ zeta function $\zeta(x)$
(see e.g.~\cite[Appendix A]{p}) which is defined on the uniformization of $E$. However, the linear
combinations we consider descend to rational functions on $E$.
In particular, the function
\[\Theta_{w_{j,r},w_{j,a_j}}(x):=\zeta(x-w_{j,r})-\zeta(x-w_{j,a_j})\]
on $E$ is a rational function with a simple pole at $w_{j,r}$ with residue 1 and a simple pole at
$w_{j,a_j}$ with residue $-1$, regular away from $w_{j,r},w_{j,a_j}$.

Using this function we can express the 1-cocycle $(g_1,g_2,g_3)$ as a coboundary
$(g'_1,g'_2,g'_3)-(g''_1,g''_2,g''_3)$ where
$(g'_1,g'_2,g'_3)\in H^0(U_{t_3},\CK_1^{-1}\CK_2\oplus\CO_E\oplus\CK_3^{-1}\CK_2)$ and
$(g''_1,g''_2,g''_3)\in H^0(U_{t_2},\CK_1^{-1}\CK_2\oplus\CO_E\oplus\CK_3^{-1}\CK_2)$.
In particular, we have
\[g'_3=0,\ g'_2=\left(\frac{t_2}{u_j}\Big|_{w_{j,r}}\right)\Theta_{w_{j,r},w_{j,a_j}}.\]
Furthermore, by definition of $d_2$ in~(\ref{fourteen}), we have
\[d_2(f)=s_1^2g'_1+s_1s_2g'_2+0\pmod{s_1s_2}\]
(note that $d_2(f)$ is actually a regular section of $\CL_i^{-2}\CK_1\CK_2$ since
$s_1t_1+s_2t_2=-s_3t_3$). Hence we have
\[d_2(f)=s_1^2g'_1-s_1s_2\left(\frac{t_2}{u_j}\Big|_{w_{j,r}}\right)\Theta_{w_{j,r},w_{j,a_j}}\pmod{s_1s_2}.\]
The composition with the last two arrows in~(\ref{fourteen}) annihilates the summand $s_1^2g'_1$,
and we are left with
\[-\left(\frac{t_2}{u_j}\Big|_{w_{j,r}}\right)\sum_{r'=1}^{a_i}\left(\frac{s_2}{u_i}\Big|_{w_{i,r'}}\right)
\Theta_{w_{j,r},w_{j,a_j}}(w_{i,r'}).\]
Pairing this expression with $dy_{i,r'}$ we finally arrive at
\begin{equation}
  \label{yy}
  \{y_{j,r},y_{i,r'}\}_{FO}=
  -y_{j,r}y_{i,r'}\big(\Theta_{w_{j,r},w_{j,a_j}}(w_{i,r'})-\Theta_{w_{j,r},w_{j,a_j}}(w_{i,a_i})\big).
\end{equation}
To be more precise, recall that our coordinates include $w_{i,1},\ldots,w_{i,a_i-1}$, but not
$w_{i,a_i}$. However, $w_{i,a_i}$ can be determined from $w_{i,1},\ldots,w_{i,a_i-1}$ and the
constraint that $\sum_{r'=1}^{a_i}w_{i,r'}$ is fixed in $E$. The same applies to $w_{j,a_j}$.
Now, instead of normalizing the $y$-coordinates by fixing their sum, let us view them as
homogeneous coordinates, so that only their ratios matter.
From~(\ref{yy}) one can deduce
\begin{equation}
  \label{yyyy}
  \Big\{\frac{y_{i,r'}}{y_{i,p'}},\frac{y_{j,r}}{y_{j,p}}\Big\}_{FO}=
  \frac{y_{i,r'}}{y_{i,p'}}\cdot\frac{y_{j,r}}{y_{j,p}}
  \big(\zeta(w_{i,r'}-w_{j,r})-\zeta(w_{i,r'}-w_{j,p})-\zeta(w_{i,p'}-w_{j,r})+\zeta(w_{i,p'}-w_{j,p})\big).
\end{equation}
  
\subsubsection{Comparison with the reduced Coulomb zastava}
To compare the bracket $\{,\}_{\CK'}$ on the reduced Coulomb zastava $^C_\CalD\uoZ^\alpha_{\CK'}$
with the Feigin-Odesskii bracket we write down the isomorphism of~Theorem~\ref{reductio} 
explicitly in coordinates. To this end we envoke the uniformization
$\fP\colon\BC\to E=\BC/(\BZ\oplus\BZ\tau)$. We denote by $w$ the coordinate on $\BC$ such that
the trivialization of $\bomega_E$ given by $dw$ coincides with the one
of~Remark~\ref{triv can}. We denote by $\theta(w)$ the theta-function
of degree~1 for the lattice $\BZ\oplus\BZ\tau$ such that $\theta(0)=0$. We use the standard
trivialization of the pullback $\fP^*\CalD_j$ such that $\prod_{r=1}^{a_j}\theta(w-w_r)$ descends
to a section of $\CalD_j$ whenever $\CO_E\big(\sum_{r=1}^{a_j}\fP(w_i)\big)\simeq\CalD_j$.

The common part of the \'etale coordinate systems on $^C\!\oZ^\alpha_{\CK'}$ and $\oZ^\alpha_\CK$
is formed by $(w_{i,r'},w_{j,r})_{r'=1,\ldots,a_i}^{r=1,\ldots,a_j}$ (we now think of them as of points
in $\BC$ rather than their images in $E$). The additional coordinates on $\oZ^\alpha_\CK$
are $(\sy_{i,r'},\sy_{j,r})_{r'=1,\ldots,a_i}^{r=1,\ldots,a_j}$, where
$\sy_{i,r'}\in\CK_i|_{w_{i,r'}},\ \sy_{j,r}\in\CK_j|_{w_{j,r}}$, and
$\CK_i=\CK^{-\alpha_i^\svee},\ \CK_j=\CK^{-\alpha_j^\svee}$.
The additional coordinates on $^C\!\oZ^\alpha_\CK$ are $(z_{i,r'},z_{j,r})_{r'=1,\ldots,a_i}^{r=1,\ldots,a_j}$,
where $z_{i,r'}\in\CK'_i|_{w_{i,r'}},\ z_{j,r}\in\CK'_j|_{w_{j,r}}$, and
$\CK'_i=\CK^{-\alpha_i^\svee}\CalD_i\CalD_j^{-1},\ \CK_j=\CK^{-\alpha_j^\svee}\CalD_j$.

On the reduced zastava the $w$-variables are constrained to have a fixed sum, while the
$\sy$-variables (resp.\ $z$-variables) are homogeneous, i.e.\ only their ratios are well defined.
The isomorphism of~Theorem~\ref{reductio} has form
\begin{equation}
  \label{dozen}
\sy_{i,r'}=z_{i,r'}\phi_{i,r'}(w_{i,1},\ldots,w_{i,a_i})\psi(w_{i,r'};w_{j,1},\ldots,w_{j,a_j}),\
\sy_{j,r}=z_{j,r}\phi_{j,r}(w_{j,1},\ldots,w_{j,a_j}),
\end{equation}
where $\psi(w_{i,r'};w_{j,1},\ldots,w_{j,a_j})$ descends to a section of $\CalD_j$ (unique up to
rescaling) that vanishes at all the points $w_{j,1},\ldots,w_{j,a_j}$. Note that rescaling
$\psi(w_{i,r'};w_{j,1},\ldots,w_{j,a_j})$ does not change the ratios $\sy_{i,q}/\sy_{i,p}$, so
the above transformation is well defined. The exact definition of $\phi_{i,r'},\phi_{j,r}$ is not
important for our purposes (we observe only that $\phi_{i,r'}$ is a nonzero element of
$\CalD_i^{-1}|_{w_{i,r'}}$). Thus we can take
\[\psi(w_{i,r'};w_{j,1},\ldots,w_{j,a_j})=\prod_{r=1}^{a_j}\theta(w_{i,r'}-w_{j,r}).\]

Now recall the coordinates $y_{i,r'}$ of~\S\ref{coor}. They depend on a choice of a trivialization
$u$ of $\CK_i\CalD_i$ and are defined as $y_{i,r'}=\frac{t}{u}|_{w_{i,r'}}$ (recall that $t$ is also
a section of $\CK_i\CalD_i$). On the other hand, $\sy_{i,r'}=\Res_{w_{i,r'}}\frac{t}{s}$, where
$s$ is a section of $\CalD_i$ with zeros $w_{i,1},\ldots,w_{i,a_i}$, see~(\ref{1:Res}). Hence
\begin{equation}
  \label{17}
  \sy_{i,r'}=y_{i,r'}\cdot\Res_{w_{i,r'}}\frac{u}{s}
\end{equation}
(where we use the trivialization of $\bomega_E$, see~Remark~\ref{triv can}).
Using the uniformization $\fP\colon\BC\to E$ and trivializing $\fP^*\CalD_i$
we can view $u$ as a trivialization of $\CK_i$. Then we can write
$s(w)=\prod_{r'=1}^{a_i}\theta(w-w_{i,r'})$, so that~(\ref{17}) becomes
\[\sy_{i,r'}=y_{i,r'}\cdot\frac{u(w_{i,r'})}{\theta'(0)\prod_{p\ne r'}\theta(w_{i,r'}-w_{i,p})}.\]
Thus viewing $u$ as a trivialization of $\CK_i$ and combining this with our trivialization of
$\fP^*\CalD_i$ we can view $z_{i,r'}$ as actual coordinates taking values in $\BC$,
and from~(\ref{dozen}) we get
\[y_{i,r'}=z_{i,r'}\phi'_{i,r'}(w_{i,1},\ldots,w_{i,a_i})\prod_{r=1}^{a_j}\theta(w_{i,r'}-w_{j,r}),\
y_{j,r}=z_{j,r}\phi'_{j,r}(w_{j,1},\ldots,w_{j,a_j}),\]
where once again, the exact form of $\phi'_{i,r'},\phi'_{j,r}$ is not important for our purposes.

We get \[\{y_{i,r'},y_{j,r}\}_{\CK'}=
y_{i,r'}y_{j,r}\cdot\frac{\partial_{w_{j,r}}\psi(w_{i,r'};w_{j,1},\ldots,w_{j,a_j})}
{\psi(w_{i,r'};w_{j,1},\ldots,w_{j,a_j})}=y_{i,r'}y_{j,r}\cdot\zeta(w_{i,r'}-w_{j,r}).\]
This in turn implies
\[\Big\{\frac{y_{i,r'}}{y_{i,p'}},\frac{y_{j,r}}{y_{j,p}}\Big\}_{\CK'}=
\frac{y_{i,r'}}{y_{i,p'}}\cdot\frac{y_{j,r}}{y_{j,p}}
\big(\zeta(w_{i,r'}-w_{j,r})-\zeta(w_{i,r'}-w_{j,p})-\zeta(w_{i,p'}-w_{j,r})+\zeta(w_{i,p'}-w_{j,p})\big).\]
Comparing with~(\ref{yyyy}) we see that the brackets $\{,\}_{\CK'}$ and $\{,\}_{FO}$ match
on $y$-coordinates. It is easy to check that they also match on the brackets involving
$w$-coordinates.

This completes the proof of Theorem~\ref{myksas} for $G=\SL(3)$.

\subsection{Proof of Theorem~\ref{myksas} for arbitrary simply laced $G$}
The \'etale local coordinates on $^{\vphantom{\alpha}}_\CalD\uoZ^\alpha_\CK$ are
$(w_{i,r},y_{i,r})_{i\in I}^{1\leq r\leq a_i}$ (as always, $w$-coordinates are constrained, and
$y$-coordinates are homogeneous). We have to compare $\{f,g\}_{FO}$ and $\{f,g\}_{\CK'}$,
where $f$ is a coordinate function from the $i$-th group, and $g$ is a coordinate function
from the $j$-th group (it may happen that $i=j$). We consider the Levi subgroup of rank 1 or 2
corresponding to the Dynkin subdiagram on vertices $i,j$. The rational projection $\varPi$ to
the corresponding Levi zastava spaces being Poisson, it suffices to compare the brackets in
question for the Levi zastava spaces. This comparison was already made in~\S\ref{prf sl2}
for rank 1 and in~\S\ref{prf sl3} for rank 2.

This completes the proof of Theorem~\ref{myksas} for arbitrary simply laced $G$. \hfill $\Box$

\end{document}